\documentclass[a4paper,oneside,11pt, reqno]{amsart}

\usepackage{graphicx}
\usepackage{float}
\usepackage{pstricks}
\usepackage{amscd}
\usepackage{amsmath}
\usepackage{amsfonts}
\usepackage{breqn}
\usepackage{amsxtra}
\usepackage{amsrefs}
\usepackage[T1]{fontenc}
\usepackage[utf8]{inputenc}
\usepackage{lipsum}
\usepackage{tikz}
\usetikzlibrary{arrows}
\usetikzlibrary{calc}

\usepackage{hyperref}

\usepackage{amsfonts,amssymb,amsmath,amsthm}
\usepackage{url}

\newcommand{\strutstretchdef}{\newcommand{\strutstretch}{\vphantom{\raisebox{1pt}{$\big($}\raisebox{-1pt}{$\big($}}}}
\theoremstyle{plain}
\newtheorem{theorem}{Theorem}[section]

\newtheorem{lemma}[theorem]{Lemma}
\newtheorem{proposition}[theorem]{Proposition}
\newtheorem{corollary}[theorem]{Corollary}

\theoremstyle{definition}
\newtheorem{definition}[theorem]{Definition}

\newtheorem{example}[theorem]{Example}

\theoremstyle{remark}
\newtheorem{remark}[theorem]{Remark}

\numberwithin{equation}{section}

\newlength{\struh}
\newlength{\textminustop}
\strutstretchdef
\newcommand*{\Ge}{\geqslant}

\newcommand*{\Le}{\leqslant}

\usepackage{mathrsfs}
\usepackage{ragged2e}
\usepackage{epsfig}
\usepackage[active]{srcltx}

\newcommand{\ncom}{\newcommand}
\ncom{\bq}{\begin{equation}}
\ncom{\eq}{\end{equation}}
\ncom{\beqn}{\begin{eqnarray*}}
\ncom{\eeqn}{\end{eqnarray*}}
\ncom{\beq}{\begin{eqnarray}}
\ncom{\eeq}{\end{eqnarray}}
\ncom{\nno}{\nonumber}
\ncom{\rar}{\rightarrow}
\ncom{\Rar}{\Rightarrow}
\ncom{\noin}{\noindent}
\ncom{\im}{{\rm Im\,}}
\ncom{\sgn}{{\rm sgn\,}}
\ncom{\eop}{\hfill{{\rule{2.5mm}{2.5mm}}}}
\ncom{\eof}{\hfill{{\rule{1.5mm}{1.5mm}}}}
\ncom{\inp}[2]{\langle{#1},\,{#2} \rangle}
\ncom{\Inp}[2]{\Langle{#1},\,{#2} \Langle}

\keywords{Hartogs triangle, domain of holomorphy, reproducing kernel, commutant, multiplier, Hardy space, Taylor spectrum, determinant operator, von Neumann's inequality, subnormality, tensor product}

\subjclass[2010]{Primary 47A13, 46E22; Secondary 32Q02, 32A10}

\begin{document}
\title[Operator theory on generalized Hartogs triangles]{Operator theory on generalized \\ Hartogs triangles}

\author[S. Chavan, S. Jain and P. Pramanick]{Sameer Chavan, Shubham Jain and Paramita Pramanick}

\address{Department of Mathematics and Statistics 
Indian Institute of Technology Kanpur, India}

\email{chavan@iitk.ac.in, shubjain@iitk.ac.in, paramitap@iitk.ac.in}

\thanks{The third named author was supported by NBHM Grant 0204-9-2022-R$\&$D-II-5885.} 

%We associate with $P$ a Reinhardt domain $\triangle^{\!n}_{_P}$ to be called as the generalized Hartogs triangle. 
%A basic family of these domains, where $P = (P_{1, a}, \ldots, P_{n, a})$ and $P_{j, a}(z) = z_j + a \prod_{k=1}^n z_k,~ j=1, \ldots, n,$ is given by 
%\beqn
%\triangle^{\!n}_a = \Big\{z \in \mathbb C \times  \mathbb C^{n-1}_* : |z_j|^2  < |z_{j+1}|^2(1-a|z_1|^2), ~j=1, \ldots, n-1, &&\\ 
% |z_n|^2 + a|z_1|^2 < 1\Big\}, \quad a \Ge 0.
% \eeqn
%has a jump discontinuity of value $\infty$ at the serious singularity $0$ of the boundary of $\triangle^{\!n}_{_P}.$ 

\begin{abstract} 
We consider the family $\mathcal P$ of $n$-tuples $P$ consisting of polynomials $P_1, \ldots, P_n$ with nonnegative coefficients which satisfy $\partial_i P_j(0) = \delta_{i, j},$  
$i, j=1, \ldots, n.$ 
With any such $P,$ we associate a Reinhardt domain $\triangle^{\!n}_{_P}$ that we will call the generalized Hartogs triangle. We are particularly interested in the choices 
$P_a = (P_{1, a}, \ldots, P_{n, a}),$ $a \Ge 0,$ where $P_{j, a}(z) = z_j + a \prod_{k=1}^n z_k,~ j=1, \ldots, n.$ The generalized Hartogs triangle associated with $P_a$ is given by 
\beqn
\triangle^{\!n}_a = \Big\{z \in \mathbb C \times  \mathbb C^{n-1}_* : |z_j|^2  < |z_{j+1}|^2(1-a|z_1|^2), ~j=1, \ldots, n-1, &&\\ 
 |z_n|^2 + a|z_1|^2 < 1\Big\}.
 \eeqn
The domain $\triangle^{\!2}_{_0}$ is the Hartogs triangle. Unlike most domains relevant to the multi-variable operator theory, the domain $\triangle^{\!n}_{_P},$ $n \Ge 2,$ is never polynomially convex. However, $\triangle^{\!n}_{_P}$ is always holomorphically convex. With any $P \in \mathcal P$ and $m \in \mathbb N^n,$ we associate a positive semi-definite kernel $\mathscr K_{_{P, m}}$ on $\triangle^{\!n}_{_P}.$ This combined with the Moore's theorem yields a reproducing kernel Hilbert space $\mathscr H^2_m(\triangle^{\!n}_{_P})$ of holomorphic functions on $\triangle^{\!n}_{_P}.$ We study the space $\mathscr H^2_m(\triangle^{\!n}_{_P})$ and the multiplication $n$-tuple $\mathscr M_z$ acting on $\mathscr H^2_m(\triangle^{\!n}_{_P}).$ It turns out that $\mathscr M_z$ is never rationally cyclic, but $\mathscr H^2_m(\triangle^{\!n}_{_P})$ admits an orthonormal basis consisting of rational functions on $\triangle^{\!n}_{_P}.$ Although 
the dimension of the joint kernel of $\mathscr M^*_z-\lambda$ is constant of value $1$ for every $\lambda \in \triangle^{\!n}_{_P}$, it has jump discontinuity at the serious singularity $0$ of the boundary of $\triangle^{\!n}_{_P}$ with value equal to $\infty.$ 
%A characterization of the jointly subnormal multiplication $n$-tuples $\mathscr M_z$ on $\mathscr H^2_m(\triangle^{\!n}_{_P})$ helps us 
We capitalize on the notion of joint subnormality to define a Hardy space $\mathscr H^2(\triangle^{\!n}_{_0})$ on the $n$-dimensional Hartogs triangle $\triangle^{\!n}_{_0}.$ This in turn gives an analog of the von Neumann's inequality for $\triangle^{\!n}_{_0}.$
\end{abstract}

\maketitle

\tableofcontents

\section{Introduction, preliminaries and layout of the paper}
A motivation for the present investigations comes from some recent developments pertaining to the function theory and operator theory on the Hartogs triangle and related domains (see \cites{AMY2020, CJP2022, GGLV2021, M2021, P2019, T2022}). It is worth noting that the work carried out in \cites{GGLV2021, M2021} provides a framework to construct Hardy spaces on a class of domains that include the Hartogs triangle.  Another motivation for this paper comes from the works \cites{CV1993, MV1993, CV1995, A1998, P1999, P2019} providing a foundation for the functional calculus and model theory for commuting tuples of Hilbert space operators. The purpose of this paper is to discuss function theory of the so-called generalized Hartogs triangles and study a class of commuting operator tuples naturally associated with these domains. In particular, we construct Hardy-Hilbert spaces of holomorphic functions on the generalized Hartogs triangles and discuss operator theory on these spaces.  

% review

%\section{Function-theoretic preliminaries}

Denote by $\mathbb Z, \mathbb Z_+$ and $\mathbb N,$ the sets of integers, nonnegative integers and positive integers, respectively. The set of real  numbers is denoted by $\mathbb R.$ 
We use the notation $\mathbb C$ (resp., $\mathbb C_*$) for the set of complex numbers (resp., nonzero complex numbers). 
%For sets $A$ and $B,$ let 
%$A \times B$ denote the Cartesian product of $A$ with $B.$ 
Let $\mathbb T$ and $\mathbb D$ denote the unit circle and unit disc in the complex plane $\mathbb C,$ respectively. 
For $n \in \mathbb N$ and a nonempty set $A,$ let $A^n$ denote the $n$-fold Cartesian product of $A$ with itself. Let $\varepsilon_k$ denote the $n$-tuple in $\mathbb C^n$ with $k$-th entry equal to $1$ and zeros elsewhere. 
%Sometimes $z=(z_1, \ldots, z_n) \in \mathbb C^n$ is rewritten as $(z_1, z'),$ where $z'=(z_2, \ldots, z_n) \in \mathbb C^{n-1}.$  
For $z =(z_1, \ldots, z_n) \in \mathbb C^n$ and $\alpha = (\alpha_1, \ldots, \alpha_n) \in \mathbb Z^n_+,$ $z^\alpha$ denotes the complex number $\prod_{j=1}^n z^{\alpha_j}_j.$ In what follows, we use the conventions that the product over an empty set is $1$ and the sum over an empty set is $0.$ For $\beta = (\beta_1, \ldots, \beta_n) \in \mathbb Z^n,$ we write $\alpha  \Le \beta$ if $\alpha_j \Le \beta_j$ for every $j=1, \ldots, n.$ 
For  $w =(w_1, \ldots, w_n) \in \mathbb C^n,$ the {\it Hadamard product} $z \diamond {w}$ of $z$ and $w$ is defined~by
\beq \label{H-prod}
z \diamond  {w} = (z_1{w}_1, \ldots, z_n{w}_n).
\eeq
For $a =(a_1, \ldots, a_n) \in \mathbb C^n$ and an $n$-tuple $r=(r_1, \ldots, r_n)$ of positive real numbers, let 
\beqn 
\mathbb D^n(a, r)&=&\{z \in \mathbb C^n : |z_j-a_j| < r_j, ~j=1, \ldots, n\}, \\  {\mathbb D}^n_*(a, r) &=& \{z \in \mathbb C^n : 0 <|z_j-a_j| < r_j, ~j=1, \ldots, n\}. 
\eeqn
For simplicity, we use the notation $\mathbb D(a, r)$ (resp. $\mathbb D_*(a, r)$) for $\mathbb D^1(a, r)$ (resp. $\mathbb D^1_*(a, r)$).  
Also, we use the notation 
$\mathbb D_*$ for ${\mathbb D}_*(0, 1).$ 
The ring of complex polynomials in the variables $z_1, \ldots, z_n$ is denoted by $\mathbb C[z_1, \ldots, z_n]$ or sometimes $\mathbb C[z]$ for brevity. 
For $Q \in \mathbb C[z],$ let 
\beq \label{Q-ball-new}
\mathbb B^n_{_{Q}}=\{z \in \mathbb C^n : |Q(z  \diamond  \overline{z})| < 1\},
\eeq 
where $\overline{z}=(\overline{z}_1, \ldots, \overline{z}_n)$ denotes the $n$-tuple of complex conjugates $\overline{z}_1, \ldots, \overline{z}_n$ of $z_1, \ldots, z_n.$  
Following \cite{P1999}, we refer to $\mathbb B^n_{_{Q}}$ as the {\it $Q$-ball}. For simplicity, we use the notation $\mathbb D_{_{Q}}$ for $\mathbb B^1_{_{Q}}.$ In general, the $Q$-ball $\mathbb B^n_{_{Q}},$ $n \Ge 2,$ is not bounded (e.g. $Q(z_1, z_2)=z_1$).

A {\it domain $\Omega$ in $\mathbb C^n$} is a nonempty open connected subset  of $\mathbb C^n.$  
The {\it absolute space} of a domain $\Omega$ in $\mathbb C^n$ is given by $\{ (|z_1|, \ldots, |z_n|) \in \mathbb R^n : z \in \Omega\}.$
By a {\it Reinhardt domain}, we understand a domain which is $\mathbb T^n$-invariant$:$
\beqn
  (\zeta_1 z_1, \ldots, \zeta_n z_n)\in \Omega~\mbox{whenever~} (z_1, \ldots, z_n) \in \Omega~\mbox{and~}(\zeta_1, \ldots, \zeta_n) \in \mathbb T^n.
\eeqn
A Reinhardt domain $\Omega$ is {\it complete} if for every $z \in \Omega,$ the closed polydisc centered at the origin and of the polyradius $(|z_1|, \ldots, |z_n|)$ is contained in $\Omega.$ 

%We record below a couple of facts pertaining to the $Q$-balls$:$
\begin{lemma} \label{application-CS}
Let $Q$ be a polynomial in the complex variables $z_1, \ldots, z_n$ with nonnegative coefficients. If $Q(0)=0,$ then the $Q$-ball $\mathbb B^n_{_Q}$ is a complete Reinhardt domain. 
\end{lemma}
\begin{proof}
Assume that $Q(0)=0.$ Clearly, $\mathbb B^n_{_Q}$ is open and $\mathbb T^n$-invariant. Since $Q$ has nonnegative coefficients, $\mathbb B^n_{_Q}$ is a complete Reinhardt domain. 
\end{proof}

%We record below a couple of facts pertaining to the $Q$-balls$:$
%
%\begin{lemma} \label{application-CS}
%For a polynomial $Q$ in the complex variables $z_1, \ldots, z_n$ with nonnegative coefficients, the following statements are valid$:$
%%$\inp{z}{w}_0 = P(z \diamon \bar{w}),$ $z, w \in \mathbb C^n,$ 
%\begin{enumerate}
%\item[$\mathrm{(i)}$] for every $z, w \in \mathbb C^n,$ 
%\allowdisplaybreaks
%\beqn 
%|Q(z \diamond w)| \Le \sqrt{Q(z \diamond \overline{z})}\sqrt{Q(w \diamond \overline{w})}, 
%\eeqn
%\item[$\mathrm{(ii)}$] if $Q(0)=0,$ then the $Q$-ball $\mathbb B^n_{_Q}$ is a complete Reinhardt domain that contains the origin. 
%\end{enumerate}
%\end{lemma}
%\begin{proof}
%(i) This is a simple application of the Cauchy-Schwarz inequality. 
%
%(ii)  Assume that $Q(0)=0.$ Clearly, $\mathbb B^n_{_Q}$ is open and $\mathbb T^n$-invariant. Since $Q$ has nonnegative coefficients, $\mathbb B^n_{_Q}$ is a complete Reinhardt domain. 
%\end{proof}
For a bounded domain $\Omega$ in $\mathbb C^n,$ let $\mathscr H^{\infty}(\Omega)$ denote the Banach algebra of bounded holomorphic functions on $\Omega$ endowed with the sup norm $\|\cdot\|_{\infty, \Omega}.$

\subsection{The $n$-dimensional Hartogs triangle}

For a positive integer $n,$ consider the domain
\beqn
\triangle^{\!n}_0 =  \{z \in \mathbb C^n : |z_{1}|  < |z_2| < \ldots < |z_n| < 1\}.
\eeqn
We refer to $\triangle^{\!n}_0$ as 
the {\it $n$-dimensional Hartogs triangle}. Note that $\triangle^{\!1}_0$ is the open unit disc $\mathbb D,$ while $\triangle^{\!2}_0$ is commonly known as the {\it Hartogs triangle}. 
For $n \Ge 2,$
 consider the automorphism $\varphi$ of $\mathbb C \times \mathbb C^{n-1}_*$ given by 
\beq \label{tilde}
\varphi(z) = \Big(\frac{z_1}{z_2}, \ldots, \frac{z_{n-1}}{z_n}, z_n\Big), \quad 
%\quad \quad\quad \quad\quad \quad \notag \\  \quad \quad \quad \quad\quad \quad\quad \quad\quad \quad 
z  =(z_1, \ldots, z_n) \in \mathbb C \times \mathbb C^{n-1}_*.
\eeq
Note that $\varphi^{-1}$ is given by
\beq \label{phi-inverse}
\varphi^{-1}(z) = \Big(\prod_{j=1}^n z_j,  \ldots, \prod_{j=n-1}^n  z_j, z_n\Big), \quad z =(z_1, \ldots, z_n) \in \mathbb C \times \mathbb C^{n-1}_*, 
\eeq
which, being a polynomial mapping, extends naturally to $\mathbb C^n.$
Note that $\varphi$ maps $\triangle^{\!n}_0$ biholomorphically onto $\mathbb D \times \mathbb D^{n-1}_*.$ In particular, $\triangle^{\!n}_0$ is a domain of holomorphy.  
It is easy to see that the {\it Jacobian} $J_{_{\varphi^{-1}}}$ of $\varphi^{-1}$ is given by
\beq \label{Jaco}
J_{_{\varphi^{-1}}}(z) = \prod_{j=2}^n z^{j-1}_j, \quad z=(z_1, \ldots, z_n) \in \mathbb C \times \mathbb C^{n-1}_*.
\eeq 

\subsection{Operator theoretic prerequisites}
For a complex Hilbert space $H,$ let $\mathcal B(H)$ denote the  $C^*$-algebra of bounded linear operators on $H.$ 
Let $\mathscr H$ be a Hilbert space of functions $f : \Omega \rar \mathbb C.$ 
A function $\phi : \Omega \rar \mathbb C$ is a said to be a {\it multiplier} of $\mathscr H$ if $\phi f \in \mathscr H$ for every $f \in \mathscr H.$ 
If $\mathscr H$ is a reproducing kernel Hilbert space (refer to \cite{PR2016} for the definition), then by the closed graph theorem, for every multiplier $\phi,$ the linear operator $\mathscr M_\phi(f)=\phi f,$ $f \in \mathscr H,$  defines a bounded linear operator on $\mathscr H,$ and hence  
the {\it multiplier algebra} of $\mathscr H$ is the subalgebra $\{\mathscr M_\phi : \phi ~\text{is a multiplier of}~\mathscr H\}$ of $\mathcal B(\mathscr H)$ (endowed with the operator norm).
For $A, B \in \mathcal B(H),$ the {\it commutator} of $A$ and $B,$ denoted by $[A, B],$ is the operator $AB-BA.$
By a {\it commuting $n$-tuple $T$ on $H$}, we understand the tuple $(T_1, \ldots, T_n)$ of commuting operators $T_1, \ldots, T_n$ in $\mathcal B(H).$
The joint kernel $\ker(T)$ of $T$ is given by $\cap_{j=1}^n \ker(T_j).$ 
In what follows, the symbols $\sigma_p(T),$ $\sigma(T)$ and $r(T)$ denote the point spectrum, the Taylor spectrum and the spectral radius of a commuting $n$-tuple $T$ on $H,$ respectively. It turns out that $\sigma(T)$ is a nonempty compact subset of $\mathbb C^n$ and $r(T)$ is a nonegative real number (see \cite{T1970}; the reader is referred to \cite{C1988} for a treatise on the multivariable spectral theory).

\subsection{Layout of the paper}

In Section~\ref{S2}, we introduce, in line with \cite{P1999}, the notion of positive regular polynomial $n$-tuple $P$ (see Definition~\ref{reg-def}) and associate a domain with it, which we call a $P$-triangle or a generalized Hartogs triangle (see \eqref{gen-H-tri} and Figure~\ref{fig}). The generalized Hartogs triangles are bounded Reinhardt domains. Also, 
if $P$ is an admissible polynomial $n$-tuple, then  
the associated $P$-triangle is a product domain up to a biholomorphism (see Proposition~\ref{DoH}). 
The generalized Hartogs triangles are domains of holomorphy (see Proposition~\ref{always-DoH}). However, these are neither polynomially convex nor hyperconvex (see Corollary~\ref{coro-h-convex}). For the choice $P=(z_1, \ldots, z_n),$ the $P$-triangle is nothing but the $n$-dimensional Hartogs triangle (see Examples~\ref{exm-triangle} and \ref{exam-Delta-P-r}). 
%We conclude the section with Pick’s interpolation theorem on the Hartogs triangle (see Proposition~\ref{Pick}). 

In Section~\ref{S3}, with a positive regular polynomial $n$-tuple $P$ and $m \in \mathbb N^n,$ we associate a scalar-valued positive semi-definite kernel  
$\mathscr K_{_{P, m}}$ on $\triangle^{\!n}_{_P}$ (see \eqref{exp-2}). This combined with the Moore's theorem yields a Hardy-Hilbert space $\mathscr H^2_m(\triangle^{\!n}_{_P})$ of holomorphic functions on the generalized Hartogs triangle $\triangle^{\!n}_{_P}.$ The space $\mathscr H^2_m(\triangle^{\!n}_{_P})$ admits an orthonormal basis consisting of rational functions (see Proposition~\ref{on-basis}).
We also provide a simplified form of the reproducing kernel $\mathscr K_{a, m}$ of the Hilbert space $\mathscr H^2_m(\triangle^{\!n}_{a})$ (see Example~\ref{exm-triangle-2}). 

In Section~\ref{S4}, we show that $z_1, \ldots, z_n$ are multipliers of $\mathscr H^2_m(\triangle^{\!n}_{_P})$ (see Proposition~\ref{bddness}). In particular,  
the multiplication $n$-tuple $\mathscr M_z$ defines a commuting $n$-tuple on $\mathscr H^2_m(\triangle^{\!n}_{_P}).$  We also provide an upper bound for the operator norm $\|\mathscr M_{z_j}\|$ of $\mathscr M_{z_j},$ $j=1, \ldots, n$ (see \eqref{norm-estimate}).  
In case $P$ is $n$-admissible, we provide a lower bound for $\|\mathscr M_{z_j}\|,$ $j=1, \ldots, n$ 
(see Corollary~\ref{norm-attained}). 
%Moreover, we localize the Taylor spectrum of $\mathscr M_z$ on $\mathscr H^2_m(\triangle^{\!n}_{_P})$ (see Corollary~\ref{eigen-v}). 
We also introduce the notion of the unilateral weighted multishift associated with $\mathscr M_z$ and discuss some  applications to the spectral theory (see Remark~\ref{after-action-adjoint} and Proposition~\ref{spectral-coro-new}).

In Section~\ref{S5}, we obtain several basic properties of the multiplication $n$-tuple $\mathscr M_z$ on $\mathscr H^2_m(\triangle^{\!n}_{_P}).$ In particular, 
we show that $\mathscr M_z$ fails to be doubly commuting and finitely rationally cyclic (see Proposition~\ref{not-F}). Moreover, we exhibit a family of positive regular polynomial $n$-tuples $P$ for which  $\mathscr M_{z_n}$ on $\mathscr H^2_m(\triangle^{\!n}_{_P})$ is not essentially normal (see Proposition~\ref{not-en} and Corollary~\ref{coro-not-en}). Further, we show that  
$\mathscr M_z$ on $\mathscr H^2_m(\triangle^{\!n}_{_P})$ is always strongly circular (see Proposition~\ref{p-circular}). As a consequence, we show that the Taylor spectrum of $\mathscr M_z$ is a rotation invariant connected subset of $\mathbb C^n$ that contains $\overline{\triangle^{\!n}_{_P}}$ (see Corollary~\ref{circular-spectral}). Furthermore, under the assumption that $P$ is an admissible $2$-tuple, we characterize the trace-class membership of the positive determinant operator of the multiplication $2$-tuple $\mathscr M_z$ on $\mathscr H^2_m(\triangle^{\!2}_{_P})$ and provide a formula for its trace (see Proposition~\ref{B-S-Hartogs}). Also, we introduce the class of $\triangle^{\!n}_{_P}$-contractions of order $m$ and show that $\mathscr M^*_z$ on $\mathscr H^2_m(\triangle^{\!n}_{_P})$ has this property provided  $\frac{1}{\mathscr K_{_{P, m}}}$ is a hereditary polynomial (see Proposition~\ref{positivity-prop}). In particular, $\mathscr M^*_z$ on $\mathscr H^2_{m}(\triangle^{\!n}_{_a})$ is a $\triangle^{\!n}_{_a}$-contraction of order $m$ (see Corollary~\ref{Coro-H-contraction}). 

In Section~\ref{S6}, we show that the commutant of $\mathscr M_z$ is a maximal abelian subalgebra of $\mathcal B(\mathscr H^2_m(\triangle^{\!n}_{_P})),$ it coincides with the multiplier algebra of $\mathscr H^2_m(\triangle^{\!n}_{_P})$ and it is contractively contained in $\mathscr H^\infty(\triangle^{\!n}_{_P})$ (see Theorem~\ref{mult-h-infty}). As a consequence, we show that $\mathscr M_z$ is irreducible (see Corollary~\ref{irrducible}). An important step in the proof of Theorem~\ref{mult-h-infty} needs the explicit computation of the joint kernel of $\mathscr M^*_z - w,$ $w \in \triangle^{\!n}_{_P} \cup \{0\}$ (see Lemma~\ref{spectral-lem}). It turns out that $w \mapsto \dim  \ker (\mathscr M^*_z - w)$ is constant $1$ on $\triangle^{\!n}_{_P},$ yet it has a jump discontinuity at the origin with value equal to $\infty.$ These facts allow us to  show that the point spectrum of $\mathscr M^*_z$ on $\mathscr H^2_m(\triangle^{\!n}_{_P})$ intersects $\mathbb C \times \mathbb C^{n-1}_*$ precisely at the points in  $\triangle^{\!n}_{_P}$ (see Theorem~\ref{theorem-pt-spectrum}). Interestingly, in case of $n=2,$ the point spectrum of $\mathscr M^*_z$ is equal to $\triangle^{\!2}_{_P} \cup \{0\}$ (see Corollary~\ref{pt-spec-in-2}).

In Section~\ref{S7}, we characterize jointly subnormal $n$-tuples $\mathscr M_z$ on $\mathscr H^2_m(\triangle^{\!n}_{_P})$ in terms of the positive regular polynomial $n$-tuples $P$ and $m \in \mathbb N^n$ (see Proposition~\ref{joint-s}). If, in addition, $P$ is an admissible $n$-tuple, 
%{\color{blue} 
%$\partial_i P_j =0$ for every $1 \Le i \neq j \Le n,$} 
then this characterization becomes more handy (see Corollary~\ref{coro-subnormal}). This combined with a solution of a Hausdorff moment problem (see Example~\ref{exm-triangle-3}) helps us define a Hardy space of the $n$-dimensional Hartogs triangle $\triangle^{\!n}_{0}$ (see Proposition~\ref{norm-Hardy}). In turn, this yields von Neumann's inequality for $\triangle^{\!n}_{0}$ (see Theorem~\ref{vn-H}). As a consequence, we show that for any $\triangle^{\!n}_{_0}$-contraction $T=(T_1, \ldots, T_n)$ on $H$ such that $\sigma(T) \subset \triangle^{\!n}_{_{0}}$ satisfies $T^*_1T_1 \Le \ldots \Le T^*_nT_n \Le I$ (see Corollary~\ref{coro-v-N-i}). 

In Section~\ref{S8}, we associate with every admissible $n$-tuple $P$ and $m \in \mathbb N^n,$ a Hilbert space $\mathscr H^2_{m}( \mathbb D^n_{_{\widetilde{P}}})$ of holomorphic functions on the polydisc $\mathbb D^n_{_{\widetilde{P}}}$ and disclose its intimate relation with $\mathscr H^2_m(\triangle^{\!n}_{_P})$ (see Proposition~\ref{Mz-tilde-Mz}). The notions of separate subnormality and joint subnormality coincide for the multiplication $n$-tuples $\widetilde{\mathscr M}_{z}$ on $\mathscr H^2_{m}( \mathbb D^n_{_{\widetilde{P}}}),$ and in particular, this fact yields a handy criterion for the joint subnormality of $\mathscr M_z$ on $\mathscr H^2_m(\triangle^{\!n}_{_P})$
(see Corollary~\ref{sep-joint-sub}). 
Moreover, we capitalize on the work \cite{CV1978} together with Proposition~\ref{Mz-tilde-Mz} to compute the Taylor spectrum of $\mathscr M_z$ on $\mathscr H^2_m(\triangle^{\!n}_{_P})$ provided $P$ is an admissible $n$-tuple (see Proposition~\ref{prop-spectrum}). If, in addition, $\mathscr M_z$ is separately hyponormal, then the Taylor spectrum of $\mathscr M_z$  is equal to $\overline{\triangle^{\!n}_{_{P_{\bf r}}}},$ where ${\bf r}=\varphi(\|\mathscr M_{z_1}\|, \ldots, \|\mathscr M_{z_n}\|)$ (see Corollary~\ref{spec-precise}). 

In Section~\ref{S9}, we discuss the Pick's interpolation for the Hartogs triangle and show how it can be deduced from the corresponding result for the bidisc (see Proposition~\ref{Pick}). Next, we present several examples of $\triangle^{\!n}_{_0}$-isometries (see Example~\ref{last-ex}) and discuss the connection between the $\triangle^{\!n}_{_0}$-isometries and toral $n$-isometries (see \eqref{connection}).  
We conclude the paper with some possible directions for future work.

\section{Generalized Hartogs triangles} \label{S2}

Before we introduce the so-called generalized Hartogs triangles, we need a definition. 
\begin{definition} \label{reg-def}
{\it For a positive integer $n,$ let $P_1, \ldots, P_n \in \mathbb C[z_1, \ldots, z_n].$
The polynomial $n$-tuple $P=(P_1, \ldots, P_n)$ is said be {\it positive regular} if for every $j=1, \ldots, n,$ 
\begin{enumerate}
\item $P_j$ has nonnegative coefficients and
\item $P_j$ has the expression of the form 
\allowdisplaybreaks
\beqn P_j(z)= a_jz_j + \sum_{k = 2}^{N_j} Q_{jk}(z), 
%\quad z=(z_1, \ldots, z_n) \in \mathbb C^n,
\eeqn for some $N_j \in \mathbb N,$ $a_j > 0$ and a homogeneous polynomial $Q_{jk}$ of degree $k.$    
\end{enumerate}
Let $d \in \mathbb N$ and assume that $n \Ge 2.$ 
A positive regular polynomial $n$-tuple $P$ is said to be {\it $d$-admissible} if for every $j=1, \ldots, n,$
\beqn 
(\partial^{\alpha} P_j)(0) =0, \quad \alpha \in \mathbb Z^n_+, ~|\alpha| \Le d, ~\alpha \neq m\varepsilon_j, ~m =1, \ldots, d,
\eeqn
where $\partial^\alpha=\partial^{\alpha_1}_1 \cdots \partial^{\alpha_n}_n$ denotes the $\alpha^{\mbox{\tiny th}}$ mixed partial derivative. We say that $P$ is {\it admissible} if it is $d$-admissible with  
$d=\max\{\deg P_1, \ldots, \deg P_n\},$ where $\deg$ denotes the degree of a polynomial.} 
\end{definition}
\begin{remark} \label{rmk-d-admissible}
Let $P=(P_1, \ldots, P_n)$ be any polynomial $n$-tuple. A routine argument using the Taylor series expansion around the origin shows that  
%%$P$ is admissible if and only if $P$ has nonnegative coefficients, $\partial_j P_j(0) > 0$ and 
%\beqn
%\partial_i P_j =0, \quad 1 \Le i \neq j \Le n.
%\eeqn
$P$ is $d$-admissible if and only if for every $j=1, \ldots, n,$
 $P_j$ has nonnegative coefficients, $\partial_j P_j(0) > 0$ and 
\beq \label{eqn-d-admissible}
P_j(z) = \sum_{k=1}^d \frac{(\partial^k_j P_j)(0)}{k!}z^k_j + \displaystyle \sum_{k = d+1}^{N_j} Q_{jk}(z), 
%\quad z=(z_1, \ldots, z_n) \in \mathbb C^n,
\eeq
where $N_j \in \mathbb N$ and $Q_{jk}$ is a homogeneous polynomial of degree $k.$ 
Thus,  
\beq 
\label{P-admissible-rmk}
\mbox{$P$ is admissible $\Longleftrightarrow$}~
P_j(z) = \sum_{k=1}^{\deg P_j} \frac{(\partial^k_j P_j)(0)}{k!}z^k_j. 
%\quad z=(z_1, \ldots, z_n) \in \mathbb C^n.
\eeq
%\beqn
%\partial_i P_j =0, \quad 1 \Le i \neq j \Le n, ~j=1, \ldots, n. 
%\eeqn
For $l \in \mathbb N,$ let $\mathcal P_l$ denote the set of $l$-admissible polynomial $n$-tuples. Clearly,   
$$\mathcal P_l = \mathcal P_k \subseteq \mathcal P_{k-1} \subseteq  \cdots\subseteq \mathcal P_1,$$ where $l$ is any integer bigger than $k=\max\{\deg P_1, \ldots, \deg P_n\}.$
\hfill $\diamondsuit$
\end{remark}

Motivated by the construction in \cite[Eqn~(2.2)]{P1999}, one can 
associate the Reinhardt domain $\triangle^{\!n}_{_P}$ in $\mathbb C^n,$ $n \Ge 2,$ with each positive regular polynomial $n$-tuple $P:$
\beq \label{gen-H-tri}
\triangle^{\!n}_{_P} = \Big\{z \in \mathbb C \times  \mathbb C_*^{n-1} : P_j\Big(\varphi(z) \diamond \overline{\varphi(z)}\Big) <1, ~j=1, \ldots, n\Big\},
\eeq
where $\varphi$ is given by \eqref{tilde} and $\diamond$ is the Hadamard product as defined in \eqref{H-prod}.
We refer to $\triangle^{\!n}_{_P}$ as the {\it $P$-triangle}. If there is no role of a particular choice of $P,$ then we refer to $\triangle^{\!n}_{_P}$ as a {\it generalized Hartogs triangle}
\footnote{This term, used here for convenience only, should not be confused with the literature usage for various Hartogs-like domains, which include fat and thin Hartogs triangles.}. 

%\footnote{This term, used here for convenience only, should not be confused with the literature usage for various Hartogs-like domains. In particular, fat and thin Hartogs triangles are not generalized Hartogs triangles.}. 

Unlike the case of the $Q$-ball (see Lemma~\ref{application-CS}), the $P$-triangle is never a complete domain (since it does not contain the origin).
%We decipher below the relationship between 
In the following proposition, we collect some geometric  properties of generalized Hartogs triangles.   
\begin{proposition} 
\label{DoH}
Let $P=(P_1, \ldots, P_n)$ be a positive regular polynomial $n$-tuple. The $P$-triangle $\triangle^{\!n}_{_P}$ is a bounded Reinhardt domain contained in $\cap_{j=1}^n \varphi^{-1}(\mathbb B^n_{_{P_j}}),$ where $\varphi$ is given by \eqref{tilde}. 
Moreover, if $\widetilde{P_j}$ is the polynomial in the complex variable $w$ given by 
\beq \label{tilde-Pj-new}
\widetilde{P_j}(w)=P_j(w \varepsilon_j), \quad  j=1, \ldots, n,
\eeq
and $\mathbb D^*_{_{\widetilde{P}_j}}  := \mathbb D_{_{\widetilde{P}_j}} \backslash \{0\},$ then  
the following statements hold$:$
\begin{enumerate}
\item[$\mathrm{(i)}$] $\triangle^{\!n}_{_P} \subseteq \varphi^{-1}  \left(\mathbb D_{_{\widetilde{P}_1}} \times \mathbb D^*_{_{\widetilde{P}_2}} \times \cdots \times \mathbb D^*_{_{\widetilde{P}_n}}\right),$
%if $Q=\frac{1}{n}\sum_{j=1}^nP_j,$ then 
%$\triangle^{\!n}_{_P} \subseteq \varphi^{-1}(\mathbb B_{Q}),$
\item[$\mathrm{(ii)}$] $P$ is admissible if and only if 
\beqn
\triangle^{\!n}_{_P} =\varphi^{-1}  \Big(\mathbb D_{_{\widetilde{P}_1}} \times \mathbb D^*_{_{\widetilde{P}_2}} \times \cdots \times \mathbb D^*_{_{\widetilde{P}_n}}\Big).
\eeqn 
\end{enumerate}
%\marginpar{{\color{blue}converse to (ii)}}
\end{proposition}
\begin{proof} Clearly, $\triangle^{\!n}_{_P}$ is an open subset of $\mathbb C^n.$ Since $\lim_{w \rar 0} P_j(w \diamond \overline{w})=0,$ by the continuity of $P,$ $\triangle^{\!n}_{_P}$ is nonempty. 
%By Remark~\ref{rmk-bdd}, 
To see that $\triangle^{\!n}_{_P}$ is bounded,
let 
$z \in \triangle^{\!n}_{_P}$ and set $w=(w_1, \ldots, w_n)=\varphi(z).$ By \eqref{gen-H-tri}, $|w_j| < 1/{\sqrt{a_j}}$ for every $j=1, \ldots, n.$
%\beqn
%|w_j| < \frac{1}{\sqrt{a_j}}, \quad j=1, \ldots, n.
%\eeqn
This combined with $z_j= \prod_{l=j}^{n}w_l$ (see \eqref{phi-inverse}) shows that  
$$|z_j| < \prod_{l=j}^{n}\frac{1}{\sqrt{a_l}}, \quad j=1, \ldots, n,$$ and hence  
 $\triangle^{\!n}_{_P}$ is bounded. 
To see the connectedness of $\triangle^{\!n}_{_P} ,$ note that  
\beqn 
%\label{inclusion-p-ball}  
\varphi(\triangle^{\!n}_{_P}) 
%&=& \{\varphi(z) \in \mathbb C \times  \mathbb C_*^{n-1} : P_j\Big(\varphi(z) \diamond \overline{\varphi(z)}\Big) <1, ~j=1, \ldots, n\} \\ 
&=& \{w \in \mathbb C \times  \mathbb C_*^{n-1} : P_j(w \diamond \overline{w}) <1, ~j=1, \ldots, n\}.
\eeqn
It now follows from \eqref{Q-ball-new} that
\beq \label{inter-minus-zero}
\varphi(\triangle^{\!n}_{_P}) = \Big(\bigcap_{j=1}^n \mathbb B^n_{_{P_j}}\Big) \setminus Z(\prod_{j=2}^n z_j),
\eeq
where $Z(\cdot)$ denotes the zero set. 
This shows that $\triangle^{\!n}_{_P} \subseteq \cap_{j=1}^n \varphi^{-1}(\mathbb B^n_{_{P_j}}).$  
By Lemma~\ref{application-CS}, 
$\bigcap_{j=1}^n \mathbb B^n_{_{P_j}}$ is a complete Reinhardt domain, and hence it is connected. This 
together with \cite[Corollary~4.2.2]{S2005} and \eqref{inter-minus-zero} now shows that $\varphi(\triangle^{\!n}_{_P})$ is connected. Since $\varphi^{-1}$ is continuous on $\mathbb C^n,$ $\triangle^{\!n}_{_P}$ is also connected. This together with the $\mathbb T^n$-invariance of $\triangle^{\!n}_{_P}$ completes the verification of the first part.  

(i) Let $Q_{jk}$ be as in Definition~\ref{reg-def}, and write
\beqn
Q_{jk}(z)=\sum_{\substack{\alpha \in \mathbb Z^n_+ \\ |\alpha|=k}}a_{_{\alpha, jk}}z^\alpha, \quad j=1, \ldots, n, ~k \Ge 2,
\eeqn
where $a_{\alpha, jk}$ are nonnegative real numbers. 
Note that for any $j=1, \ldots, n$ and $w \in \mathbb C^n,$
\allowdisplaybreaks
\beq 
\label{Pj-Pj-tilde}
P_j(w \diamond \overline{w}) - \widetilde{P_j}(|w_j|^2) &=& 
%a_j|w_j|^2 +\displaystyle \sum_{k \Ge 2}Q_{jk}(w \diamond \overline{w}) - P_j(|w_j|^2 \varepsilon_j) \\
%&=& 
\sum_{k = 2}^{N_j} \Big(Q_{jk}(w \diamond \overline{w}) - Q_{jk}(|w_j|^2 \varepsilon_j)\Big) \notag \\
&=& \sum_{k = 2}^{N_j} \sum_{\substack{\alpha \neq k\varepsilon_j \\ \alpha \in \mathbb Z^n_+, \, |\alpha|=k}}a_{_{\alpha, jk}}|w^\alpha|^2.
\eeq
This together with \eqref{inter-minus-zero} implies that 
\beq 
\label{equality-holds-new}
\varphi(\triangle^{\!n}_{_P}) \subseteq \mathbb D_{_{\widetilde{P}_1}} \times \mathbb D^*_{_{\widetilde{P}_2}} \times \cdots \times \mathbb D^*_{_{\widetilde{P}_n}}. 
\eeq 
This yields (i). 

(ii) If $P$ is admissible, then 
by \eqref{P-admissible-rmk}, $P_j(w \diamond \overline{w}) = \widetilde{P_j}(|w_j|^2)$, and hence 
equality holds in \eqref{equality-holds-new}. Conversely, if equality holds in \eqref{equality-holds-new}, then it is easy to see from \eqref{inter-minus-zero} and \eqref{Pj-Pj-tilde} that $P_j(w \diamond \overline{w}) = \widetilde{P_j}(|w_j|^2).$ Hence, by \eqref{P-admissible-rmk}, $P$ is admissible, which completes  the proof.
\end{proof}
\begin{remark} 
%\label{P-disc-is-disc} 
For a positive integer $n$ and a positive regular polynomial $n$-tuple $P=(P_1, \ldots, P_n),$ let
$\widetilde{P}_j$ be as defined in \eqref{tilde-Pj-new}. 
For $j=1, \ldots, n,$ consider the function $G_j : [0, \infty) \rar [0, \infty)$ given by
\beqn
G_j(t) =
\widetilde{P}_j(t)=a_j t +\displaystyle \sum_{k = 2}^{N_j} Q_{jk}(t\varepsilon_j), \quad t \in [0, \infty). 
\eeqn
Since the coefficients of $\widetilde{P}_j$ are nonnegative and $a_j > 0,$  $G_j$ 
is a one-to-one function. It follows that 
\beq \label{Q-disc-is-disc}
\mathbb D_{_{\widetilde{P}_j}} = \mathbb D(0,  r_j),~\mbox{where $r_j=\sqrt{G^{-1}_j(1)},$ $j=1, \ldots, n.$}
\eeq
This combined with Proposition~\ref{DoH}(ii) shows that if $P$ is admissible, then $$\varphi(\triangle^{\!n}_{_P})=\mathbb D(0,  r_1) \times \mathbb D_*(0,  r_2) \times \cdots \times \mathbb D_*(0,  r_n).$$ 
Moreover, \eqref{Q-disc-is-disc} together with Proposition~\ref{DoH}(i) and \eqref{phi-inverse} gives another verification of the boundedness of $\triangle^{\!n}_{_P}.$   
\hfill $\diamondsuit$
\end{remark}

Let $\Omega$ be a domain in $\mathbb C^n$ and let 
$K$ be a compact subset of $\mathbb C^n.$ Recall that the monomial convex hull 
$\hat{K}_\Omega$ of  $K$ in $\Omega$ is given by 
$$\hat{K}_\Omega=\Big\{w \in \Omega : |w^{\alpha}| \, \Le \, \sup_{z \in K}\,|z^{\alpha}|~\mbox{for every~}\alpha \in \mathbb Z^n_+\Big\}.$$ 
Contrary to the cases of unit ball, polydisc or symmetrized polydisc (see \cite[Lemma~2.1]{AY1999}), the domain $\triangle^{\!n}_{_P}$ is never monomially convex. To see this, note that since $\triangle^{\!n}_{_P}$ is a Reinhardt domain, if $w \in \triangle^{\!n}_{_P},$ then  $\triangle^{\!n}_{_P}$ contains the compact set $K:=\{z \in \mathbb C^n : |z_j|=|w_j|,~j=1, \ldots, n\}.$ On the other hand, $\hat{K}_{\triangle^{\!n}_{_P}}=\{z \in \triangle^{\!n}_{_P} : |z_j| \Le |w_j|, ~j=1, \ldots, n\},$ which is clearly not closed in $\mathbb C^n$ (it does not contain the origin, which is a limit point of $\hat{K}_{\triangle^{\!n}_{_P}}$). 
% by the Cauchy estimates (see \cite[Theorem~1.3.3]{S2005}), any Reinhardt monomially convex domain necessarily contains the origin, and $\triangle^{\!n}_{_P}$ does not contain the origin. 
Interestingly, the generalized Hartogs triangles are always holomorphically convex. This is a consequence of the following proposition and the Cartan-Thullen theorem (see \cite[Theorem~7.4.5]{S2005}). 
\begin{proposition} \label{always-DoH}
The $P$-triangle $\triangle^{\!n}_{_P}$ is a domain of holomorphy. 
\end{proposition}
\begin{proof} Let $Q$ be a polynomial with nonnegative coefficients such that $Q(0)=0.$ We contend that 
\beq \label{Q-ball-DoH} 
\mbox{the $Q$-ball $\mathbb B^n_{_Q}$ is a domain of holomorphy.} 
\eeq
Since $\mathbb B^n_{_Q}$ is a complete Reinhardt domain (see Lemma~\ref{application-CS}), by the virtue of \cite[Theorem 7.4.5]{S2005}, it suffices to check that $\mathbb B^n_{_Q}$ is monomially convex. To see this, suppose that $Q(z)=\sum_{ \alpha \in F}a_\alpha z^\alpha$ for some finite subset $F$ of $\mathbb Z^n_+$ and $a_\alpha \in (0, \infty).$ 
Let $K$ be a compact subset of $\mathbb B^n_{_Q}.$ 
Note that there exists $z_0=(z_{0, 1}, \ldots, z_{0, n}) \in K$ such that 
\beq \label{sup-attained}
\sup_{z \in K}\,|z_j|=|z_{0, j}|, \quad j=1, \ldots, n.
\eeq 
It follows that for any $w \in \hat{K}_{\mathbb C^n},$ 
\beqn
Q(w \diamond \overline{w})  &=& \sum_{ \alpha \in F}a_\alpha |w^\alpha|^2 
 ~\Le ~ \sum_{ \alpha \in F}a_\alpha \sup_{z \in K}|z^\alpha|^2 \\
& \overset{\eqref{sup-attained}}= & \sum_{ \alpha \in F}a_\alpha |z^\alpha_0|^2 
~=~ 
Q(z_0 \diamond \overline{z}_0),
\eeqn
which is less than $1$ (since $z_0 \in K \subseteq  \mathbb B^n_{_Q}$). 
Hence, $\hat{K}_{\mathbb C^n} \subseteq  \mathbb B^n_{_Q}.$ 
This shows that $\hat{K}_{\mathbb B^n_{Q}}=\hat{K}_{\mathbb C^n}.$ Since $\hat{K}_{\mathbb C^n}$ is closed and bounded, $\hat{K}_{\mathbb B^n_{Q}}$ is compact, and hence $\mathbb B^n_{_Q}$  is monomially convex, which completes the verification of \eqref{Q-ball-DoH}. 

We now complete the proof. Since $\varphi$ is an automorphism of $\mathbb C \times \mathbb C^{n-1}_*,$ by \eqref{inter-minus-zero}, it suffices to check that 
\beq \label{inter-set-minus-DOH}
\mbox{$\Big(\bigcap_{j=1}^n \mathbb B^n_{_{P_j}}\Big) \setminus Z(\prod_{j=2}^n z_j)$ is a domain of holomorphy.}
\eeq
Since $P_j$ is a polynomial with nonnegative coefficients and $P_j(0)=0,$ by \eqref{Q-ball-DoH}, $\mathbb B^n_{_{P_j}}$ is a domain of holomorphy for every $j=1, \ldots, n.$  This combined with the fact that the intersection of any two domains of holomorphy is again a domain of holomorphy 
shows that $\bigcap_{j=1}^n \mathbb B^n_{_{P_j}}$ is a domain of holomorphy. Now an application of 
 \cite[Corollary~7.2.4]{S2005} yields \eqref{inter-set-minus-DOH}. 
\end{proof}

Recall from \cite[Definition~1.17.1]{JP2008} that an open set $\Omega \subseteq \mathbb C^n$ is {\it hyperconvex} if there exists
a plurisubharmonic function $u : \Omega \rar \mathbb R$ such that $u < 0$ and $\{z \in \Omega : u(z) < t\}$ is relatively compact in $\Omega$ for every $t <0.$  
It is well-known that a hyperconvex domain is pseudoconvex (see \cites{KR1981, Z2000}). In view of Proposition~\ref{always-DoH} and the fact that a domain is pseudoconvex  if and only if it is a domain of holomorphy (see \cite[Theorem*~1.16.1]{JP2008}), the following is worth noting. 
\begin{corollary} \label{coro-h-convex}
The $P$-triangle $\triangle^{\!n}_{_P}$ is never hyperconvex. 
\end{corollary}
\begin{proof} Since $\triangle^{\!n}_{_P}$ is a bounded Reinhardt domain of holomorphy (see Propositions~\ref{DoH} and \ref{always-DoH}) and $\overline{\triangle^{\!n}_{_P}}$ contains the origin, this is immediate from 
\cite[Corollary 2.6.11]{Z2000} (applied to $j = n$). 
\end{proof}

%follows from the facts that any open subset of $\mathbb C$ is a domain of holomorphy and the Cartesian product of domains of holomorphy is again a domain of holomorphy. 

%domain of holomorphy or not polynomial convexity

We now present below some examples of $P$-triangles that are central to this article.

\begin{example}  \label{exm-triangle}
Let $n \Ge 2$ be an integer. 
For a real number $a \Ge 0,$ consider the positive regular polynomial $n$-tuple $P_a = (P_{1, a}, \ldots, P_{n, a})$ given by
\beq 
\label{def-P-a}
P_{j, a}(z) = z_j + a \prod_{k=1}^n z_k, \quad j=1, \ldots, n.
\eeq
By Remark~\ref{rmk-d-admissible}, $P_a$ is an $(n-1)$-admissible $n$-tuple, and  
\beqn 
%\label{Pa-adm-0}
\mbox{$P_a$ is admissible if and only if $a=0.$}
\eeqn 
Note that by \eqref{tilde},
\allowdisplaybreaks
\beqn
P_{j, a}\Big(\varphi(z) \diamond \overline{\varphi(z)}\Big) &=& P_{j, a}\big(|z_1/z_2|^2, \ldots, |z_{n-1}/z_n|^2, |z_n|^2\big) \\
&=& \begin{cases} |z_j/z_{j+1}|^2 + a |z_1|^2 & \mbox{if~} j=1, \ldots, n-1,\\
|z_n|^2 + a|z_1|^2 & \mbox{if~} j=n.
\end{cases}
\eeqn
Thus the $P_a$-triangle, which we denote for simplicity $\triangle^{\!n}_a,$ is given by
\beqn
 \triangle^{\!n}_a =   \Big\{z \in \mathbb C \times  \mathbb C^{n-1}_* : |z_j|^2  < |z_{j+1}|^2(1-a|z_1|^2), ~j=1, \ldots, n-1, &&\\ 
 |z_n|^2 + a|z_1|^2 < 1\Big\}. &&
\eeqn
Let us see $\triangle^{\!n}_{{a}}$ for particular choices of $a$ and $n\!$ (see Figure~\ref{fig})$:$
\begin{enumerate}
\item[$\bullet$] In case of $a=0,$ 
\beqn
\triangle^{\!n}_a =  \{z \in \mathbb C \times  \mathbb C_*^{n-1} : |z_{1}|  < |z_2| < \ldots < |z_n| < 1\}.
\eeqn
\item[$\bullet$] In case of $a=1$ and $n=2,$
\beqn
\triangle^{\!n}_a =  \{z \in \mathbb C \times  \mathbb C_* : |z_{1}|^2  < |z_2|^2(1-|z_1|^2), ~|z_1|^2 + |z_2|^2 < 1\}.
\eeqn
\end{enumerate}
The domain $\triangle^{\!n}_{_0}$ is the $n$-dimensional Hartogs triangle, while  $\triangle^{\!2}_{_0}$ is the Hartogs triangle.
\eof
\end{example}

\begin{figure}[t]
\begin{tikzpicture}[scale=.8, transform shape]
\tikzset{vertex/.style = {shape=circle,draw,minimum size=1em}}
\tikzset{edge/.style = {->,> = latex'}}

%H-triangle
 \draw [ultra thick, draw=black, fill=blue, opacity=0.3]
(0,0) -- (0,4) -- (4, 4) -- (0, 0);

\node[] (a) at  (-0.6, 0) {$(0, 0)$};
\node[] (a) at  (0, 0) {$\bullet$};
\node[] (b) at  (-0.6, 4) {$(0, 1)$};
\node[] (b) at  (0, 4) {$\bullet$};
\node[] (c) at  (4.6,4) {$(1, 1)$};
\node[] (c) at  (4, 4) {$\bullet$};
%\node[] (d) at  (2.5,1.5) {$\triangle^{\!n}_{_0}$};
%G-H-triangle
 \draw [ultra thick, draw=black, fill=blue, opacity=0.3]
(8,0) -- (8,4) arc[start angle=90, end angle=45, radius=4] 
(10.6, 2.6) to[out=61,in=0] (8,0);
%\draw [ultra thick, draw=black, fill=blue, opacity=0.2]
%(8,0) -- (8,4) arc[start angle=90, end angle=45, radius=4] 
%(10.81, 2.81)-- (8, 0);

%\draw [ultra thick, black, fill=blue, opacity=0.2] (10.81, 2.81) to[out=65,in=45] (8,0);

\node[] (a) at  (7.4, 0) {$(0, 0)$};
\node[] (a) at  (8, 0) {$\bullet$};
\node[] (b) at  (7.4, 4) {$(0, 1)$};
\node[] (b) at  (8, 4) {$\bullet$};
\node[] (c) at  (11.8,2.3) {$(0.618, 0.7862)$};
\node[] (c) at  (10.78, 2.83) {$\bullet$};
%\draw (8, 4) arc[start angle=90, end angle=45, radius=4];
\end{tikzpicture}
\caption{Absolute spaces
of $\triangle^{\!2}_{_0}$ and $\triangle^{\!2}_{_1},$ respectively} \label{fig}
\end{figure}

We need a variant of the $n$-dimensional Hartogs triangle in the sequel. 
\begin{example} \label{exam-Delta-P-r}
For the $n$-tuple $\bf r$ of positive real numbers $r_1, \ldots, r_n,$ consider the 
positive regular polynomial $n$-tuple $P_{\bf r}=(P_{{\bf r}, 1}, \ldots, P_{{\bf r}, n})$ given by
\beqn
P_{{\bf r}, j}(z) = \frac{z_j}{r^2_j}, \quad j=1, \ldots, n.
\eeqn
Note that by \eqref{tilde} and \eqref{gen-H-tri},
the $P_{\bf r}$-triangle is given by 
\beqn
\triangle^{\!n}_{_{P_{\bf r}}} &=& \{z \in \mathbb C \times \mathbb C^{n-1}_* : \frac{|z_j|}{|z_{j+1}|} < r_j, ~j=1, \ldots, n-1, ~|z_n| < r_n\} \\ 
&=&
\Big\{ z \in \mathbb C \times \mathbb C^{n-1}_* : |z_1| < r_1|z_2| < \ldots < \Big(\prod_{k=1}^{n-1} r_k\Big) |z_n| < \prod_{k=1}^{n} r_k\Big\}. 
\eeqn
We claim that 
\beq 
\label{image-polydisc}
\mathrm{the ~closure ~of}~ \triangle^{\!n}_{_{P_{\bf r}}} ~\mathrm{is ~equal ~to}~\varphi^{-1}\Big(\overline{\mathbb D}(0, r_1) \times   \cdots \times \overline{\mathbb D}(0, r_n)\Big).
\eeq
To see this, let $A: = \overline{\mathbb D}(0, r_1) \times \overline{\mathbb D}(0, r_2) \backslash \{0\} \times \cdots \times \overline{\mathbb D}(0, r_n)\backslash \{0\}$ and note that  
\beq
\label{A-def}
A ~\mathrm{is ~equal~ to~the~ closure ~of}~\mathbb D(0, r_1) \times {\mathbb D}^{n-1}_*(0, r')~\mathrm{in~}\mathbb C \times \mathbb C^{n-1}_*,
\eeq
where $r'=(r_2, \ldots, r_n).$
It is easy to see that 
\allowdisplaybreaks
\beqn 
%\label{varphi-image}
& &\varphi^{-1}(\mathbb D(0, r_1) \times {\mathbb D}^{n-1}_*(0, r')) \\ 
%\notag &=& \Big\{\Big(\prod_{j=1}^n z_j, \ldots, \prod_{j=n-1}^n z_j, z_n\Big) \in \mathbb C^n : |z_1| < r_1, 0 < |z_j| < r_j, ~j=2, \ldots, n\Big\} \\ \notag 
%&=& \Big\{(w_1, \ldots, w_n) \in \mathbb C^n : |w_1| < r_1 |w_2|, \ldots, |w_{n-1}| < r_{n-1}|w_n|, |w_n| < r_n\Big\} \\ \notag 
&=& \Big\{ (w_1, \ldots, w_n) \in \mathbb C^n : |w_1| < r_1|w_2| < \ldots < \Big(\prod_{k=1}^{n-1} r_k\Big) |w_n| < \prod_{k=1}^{n} r_k\Big\}.
\eeqn
Combining this with \eqref{A-def} and the fact that $\varphi^{-1}$ is an automorphism of $\mathbb C \times \mathbb C^{n-1}_*$
shows that~$\varphi^{-1}(A)$ is equal to  
\begin{equation*}
 \Big\{ (w_1, \ldots, w_n) \in \mathbb C \times \mathbb C^{n-1}_* : |w_1| \Le r_1|w_2| \Le \ldots \Le \Big(\prod_{k=1}^{n-1} r_k\Big) |w_n| \Le \prod_{k=1}^{n} r_k\Big\}.
\end{equation*}
It is now easy to see that \eqref{image-polydisc} holds and the claim stands verified. Finally, note that for the choice $\bf r$ with $r_j=1,$ $j =1, \ldots, n,$ the $P_{\bf r}$-triangle $\triangle^{\!n}_{_{P_{\bf r}}}$ is the $n$-dimensional Hartogs triangle $\triangle^{\!n}_{_0}.$
\eof  
\end{example}

%For a bounded domain $\Omega$ in $\mathbb C^n,$ let $\mathscr H^{\infty}(\Omega)$ denote the Banach algebra of bounded holomorphic functions on $\Omega$ endowed with the sup norm $\|\cdot\|_{\infty, \Omega}.$ 
%The following is a simple consequence of the corresponding result for the unit bidisc. 
%\begin{proposition}[Pick’s interpolation theorem on the Hartogs triangle]  \label{Pick}
%Let $\lambda_j =(\lambda^{(j)}_1 , \lambda^{(j)}_2),$ $j=1, \ldots, k,$ be distinct points in $\triangle^{\!2}_0$ and let $z_1, \ldots, z_k \in \mathbb C.$ There exists $\psi \in \mathscr H^{\infty}(\triangle^{\!2}_0)$ with $\|\psi\|_{\infty, \triangle^{\!2}_0} \Le 1$ such that $\psi(\lambda_j) = z_j$ for $j = 1, \ldots, k$ if and only if there exists a pair of $k \times k$ positive semi-definite matrices $A_1 = [a^{(1)}_{i, j}]$ and $A_2 = [a^{(2)}_{i, j}]$ such that
%for $i,j = 1, \ldots, k,$
%\beqn
%1-\overline{z}_i z_j = \Big({\overline{\lambda^{(2)}_i}}{{\lambda^{(2)}_j}} - \overline{\lambda^{(1)}_i} {{\lambda^{(1)}_j}}\Big) a^{(1)}_{i, j} + 
%\Big(1-\overline{\lambda^{(2)}_i} \lambda^{(2)}_j\Big)a^{(2)}_{i, j}.
%\eeqn
%\end{proposition}
%\begin{proof}
%Since $\varphi,$ as given by \eqref{tilde}, maps $\triangle^{\!2}_0$ biholomorphically onto $\mathbb D \times \mathbb D_*,$ the result may be easily deduced from \cite[Theorem~4.49]{AMY2020}.
%\end{proof}
%Needless to say, one can also get an analogue of \cite[Theorem~4.56]{AMY2020} for the Hartogs triangle. 

\section{Hardy-Hilbert spaces of generalized Hartogs triangles} \label{S3}

In this section, we associate a positive semi-definite kernel with every generalized Hartogs triangle. To do that, we need a  lemma (the part (i) below has been recorded implicitly in \cite{P1999}). 
\begin{lemma} \label{Spott} 
Let $Q(z) = \displaystyle \sum_{\gamma \in \mathbb Z^n_+} q_{\gamma} z^{\gamma}$ be a nonzero polynomial in the complex variables $z_1, \ldots, z_n,$ where $q_\gamma \in \mathbb C.$ Assume that there exists a neighborhood $\Omega_0$ of $0$ such that $\overline{\Omega}_0 \subseteq \{z \in \mathbb C^n : |Q(z)| < 1\}.$ For every integer $k \in \mathbb Z_+,$ consider the power series expansion$:$ 
\beq \label{exp-1-by-1-Q}
\frac{1}{(1-Q(z))^{k}} = 
%\bigg ( \sum_{j=0}^{\infty}Q^{j}(z)\bigg )^{k} = 
\sum_{\alpha \in \mathbb Z^n_+} A_{Q, k}(\alpha) {z}^\alpha, \quad z \in \Omega_0.
\eeq
Set $A_{Q, k}=0$ on $\mathbb Z^n \backslash \mathbb Z^n_+.$ 
Then the following statements hold$:$ 
\begin{enumerate} 
\item[$\mathrm{(i)}$] for any $\alpha \in \mathbb Z^n_+,$ we have 
\beqn A_{Q, k}(\alpha) = A_{Q, k-1}(\alpha) + \displaystyle \sum_{\gamma \in \mathbb Z^n_+}q_{\gamma} A_{Q, k}(\alpha - \gamma), \quad k \in \mathbb N, \eeqn
\item[$\mathrm{(ii)}$] if, for some $j=1, \ldots, n,$ $Q$ is independent of $z_j,$  then 
\beqn A_{Q, k}(\alpha) = 0 ~\mbox{for all~} \alpha \in \mathbb Z^n_+ ~\mbox{such that~} \alpha_j \neq 0,
\eeqn
\item[$\mathrm{(iii)}$] if $Q$ has nonnegative coefficients, then $A_{Q, k}(\alpha) \Ge 0$ for every $\alpha \in \mathbb Z^n.$ If, in addition, $q_{\varepsilon_j} \neq 0$ for some $j=1, \ldots, n,$ then $A_{Q, k}(\alpha_j \varepsilon_j) > 0$ for every $\alpha_j \in \mathbb Z_+.$
\end{enumerate}
\end{lemma}
\begin{proof} Clearly, the series in \eqref{exp-1-by-1-Q} converges absolutely on $\Omega_0.$ Let $k \in \mathbb N$ and $z \in \Omega_0.$ Note that 
\beqn
\frac{Q(z)}{(1-Q(z))^{k}} &\overset{\eqref{exp-1-by-1-Q}}=& \Big(\sum_{\alpha \in \mathbb Z_+^n} A_{Q, k}(\alpha) z^\alpha \Big)\Big(\sum_{\gamma \in \mathbb Z_+^n} q_\gamma z^\gamma \Big) \\ &=& \sum_{\alpha \in \mathbb Z_+^n} \Big(\sum_{\gamma \in \mathbb Z_+^n} q_\gamma A_{Q, k}(\alpha - \gamma)\Big) z^\alpha.
\eeqn
%If $k=1,$ then $\rho_{_Q}^1(\alpha) = \displaystyle \sum_{\gamma \in \mathbb Z_+^n}a_{\gamma} \rho_{_Q}^1(\alpha - \gamma).$
%Now for $k>1,$
This combined with \eqref{exp-1-by-1-Q} yields 
\allowdisplaybreaks
\beqn
\sum_{\alpha \in \mathbb Z_+^n} A_{Q, k-1}(\alpha) z^\alpha  &=& \frac{1}{(1-Q(z))^{k-1}}   \\
&=&  \frac{1}{(1-Q(z))^{k}} - \frac{Q(z)}{(1-Q(z))^{k}}  \\ &=&  \sum_{\alpha \in \mathbb Z_+^n} \Big(A_{Q, k}(\alpha)-\sum_{\gamma \in \mathbb Z_+^n} q_\gamma A_{Q, k}(\alpha - \gamma) \Big)  z^\alpha
\eeqn
(all the series above are absolutely convergent in a neighborhood of the origin). 
Comparing the coefficients on both sides, we obtain the desired identity.
To see (ii), note that if $Q$ is independent of $z_j,$ then the series on the right hand side of \eqref{exp-1-by-1-Q} does not contain any term involving $z_j.$ 
Similarly, (iii) may be deduced from \eqref{exp-1-by-1-Q}.
%The remaining parts are immediate from \eqref{exp-1-by-1-Q}. 
\end{proof}

Let $m=(m_1, \ldots, m_n) \in \mathbb N^n$ and let $P$ be a positive regular 
polynomial $n$-tuple. Fix $j=1, \ldots, n.$ Since $P_j(0)=0,$ 
$\Omega_j:=\{z \in \mathbb C^n : |P_j(z)|<1\}$ is a nonempty open set containing the origin. 
Hence, by Lemma~\ref{Spott}, there exists a function $A_{P_j, m_j} : \mathbb Z^n \rar [0, \infty)$ such that
\beq 
\label{coeff-Pj-mj}
\frac{1}{(1-P_j(z))^{m_j}} = 
%\bigg ( \sum_{j=0}^{\infty}Q^{j}(z)\bigg )^{k} = 
\sum_{\alpha \in \mathbb Z^n_+} A_{P_j, m_j}(\alpha) {z}^\alpha,
\eeq
which converges compactly on $\Omega_j.$
%For every $z, w \in \mathbb C^n,$ 
%\beqn 
%|Q(z \diamond w)| \Le \sqrt{Q(z \diamond \overline{z})}\sqrt{Q(w \diamond \overline{w})}, 
%\eeqn
%(i) This is a simple application of the Cauchy-Schwarz inequality. 
%By Lemma~\ref{application-CS}(i), 
An application of the Cauchy-Schwarz inequality shows that 
 \beqn 
|P_j(z \diamond w)| \Le \sqrt{P_j(z \diamond \overline{z})}\sqrt{P_j(w \diamond \overline{w})}, \quad z, w \in \mathbb C^n.
\eeqn
It follows that 
\beqn
\varphi(z) \diamond \overline{\varphi(w)} \in \Omega_j, \quad z, w \in \triangle^{\!n}_{_P}.
\eeqn
This combined with \eqref{coeff-Pj-mj} yields a function $A_{P, m} : \mathbb Z^n_+ \rar [0, \infty)$ such that  
\beq \label{exp} 
\prod_{j=1}^n \frac{1}{\Big(1-P_j(\varphi(z) \diamond \overline{\varphi(w)})\Big)^{m_j}} 
%&=& \prod_{j=1}^n \bigg ( \sum_{k=0}^{\infty} P_j^{k}(\varphi(z) \diamond \overline{\varphi(w)})\bigg )^{m_j} \\ &=& 
= \sum_{\alpha \in \mathbb Z^n_+} A_{P, m}(\alpha) \Big(\varphi(z) \diamond \overline{\varphi(w)}\Big)^{\alpha},
\eeq
which converges compactly  on  $\triangle^{\!n}_{_P} \times \triangle^{\!n}_{_P}.$
We refer to $A_{P, m}$ as the {\it coefficient-function associated with the pair $(P, m).$}
We extend $A_{P, m}$ to $\mathbb Z^n$ by setting $A_{P, m}=0$ on $\mathbb Z^n \backslash \mathbb Z^n_+.$ 
\begin{remark}  
%\label{rmk-exp}
 If we take $\varphi : \mathbb C \rar \mathbb C$ as the identity map and $P$ as any polynomial in a complex variable such that 
$P$ has nonnegative coefficients, $P(0)=0$ and $P'(0) > 0,$ then the definition of $P$-triangle makes sense also in the case of $n=1.$
In this case, the $P$-triangle is nothing but the $P$-ball $\mathbb D_{_P}$ (see \eqref{Q-ball-new}). Moreover,
the identity \eqref{exp} above takes the following form$:$
\beq \label{exp-exp} 
\frac{1}{\Big(1-P(z \overline{w})\Big)^{m}} 
= \sum_{k \in \mathbb Z_+} A_{P, m}(k) z^k \overline{w}^{k},
\eeq
which converges compactly  on  $\mathbb D_{_P} \times \mathbb D_{_P}.$
\hfill $\diamondsuit$
\end{remark}

We collect below some properties of the coefficient-function. 
\begin{lemma} \label{APM-positive}
If $A_{P, m}$ is the coefficient-function associated with the pair $(P, m),$ then  $A_{P, m}$ maps  $\mathbb Z^n_+$ into $(0, \infty)$ such that  
\beq \label{rho-alpha-f-new}
&& A_{P, m}(0)=1, ~ A_{P_j, m_j}(0)=1, ~j=1, \ldots, n, \\ \label{rho-alpha-f}
&& A_{P, m}(\alpha) = \sum_{{\substack{\gamma^{(1)}, \ldots, \gamma^{(n)} \in \mathbb Z^n_+ \\ \sum_{k=1}^n \gamma^{(k)} = \alpha}}} \prod_{j=1}^n A_{P_j, m_j}(\gamma^{(j)}).
\eeq
If, in addition, $P$ is admissible, then 
\beq \label{A-tilde-A-new}
A_{P, m}(\alpha) = \prod_{j=1}^nA_{{P}_j, m_j}(\alpha_j \varepsilon_j), \quad \alpha \in \mathbb Z^n_+. 
\eeq
\end{lemma}
\begin{proof} 
Note that for $z, w \in \triangle^{\!n}_{_P},$
\allowdisplaybreaks
\beqn
&& \sum_{\alpha \in \mathbb Z^n_+} A_{P, m}(\alpha) \Big(\varphi(z) \diamond \overline{\varphi(w)}\Big)^{\alpha} \\ &\overset{\eqref{exp}} = &\prod_{j=1}^n \frac{1}{\Big(1-P_j\Big(\varphi(z) \diamond \overline{\varphi(w)}\Big)\Big)^{m_j}} \\
%&=& \prod_{1 \Le j \Le n} \bigg ( \sum_{k_j=0}^{\infty} \Big[P_j\Big(\varphi(z) \diamond \overline{\varphi(z)}\Big)\Big]^{k_j}\bigg )^{m_j} \\
 &\overset{\eqref{coeff-Pj-mj}}=& \prod_{j=1}^n \sum_{\gamma^{(j)} \in \mathbb Z^n_+} A_{P_j, m_j}(\gamma^{(j)}) \Big(\varphi(z) \diamond \overline{\varphi(w)}\Big)^{\gamma^{(j)}},
\eeqn
where all the series above converge compactly on $\triangle^{\!n}_{_P} \times \triangle^{\!n}_{_P}.$
Comparing the coefficients of $\Big(\varphi(z) \diamond \overline{\varphi(w)}\Big)^{\alpha}$ on both sides, we get \eqref{rho-alpha-f}. To see \eqref{rho-alpha-f-new}, note that by comparing the constant term in \eqref{coeff-Pj-mj}, we obtain $A_{P_j, m_j}(0)=1,$ $j=1, \ldots, n,$ To complete the verification of \eqref{rho-alpha-f-new}, let $\alpha =0$ in \eqref{rho-alpha-f}. Next, note that
since $P$ is a positive regular polynomial $n$-tuple, Lemma~\ref{Spott}(iii) together with \eqref{rho-alpha-f} shows that $A_{P, m}$ maps  $\mathbb Z^n_+$ into $(0, \infty).$ 
Finally, if $P$ is admissible, then  
by Lemma~\ref{Spott}(ii),
\beqn
A_{P_j, m_j}(\alpha)=0~\mbox{if~}\alpha \in \mathbb Z^n_+, ~\alpha_i \neq 0, ~1 \Le i \neq j \Le n.
\eeqn
Hence, \eqref{A-tilde-A-new} follows from \eqref{rho-alpha-f}.
\end{proof}

We show below that every pair $(P, m)$ yields a positive semi-definite kernel. 
\begin{lemma} \label{lem-Moore}
Let $P$ be a positive regular polynomial $n$-tuple and let $m \in \mathbb N^n.$ The kernel function 
$\mathscr K_{_{P, m}} : \triangle^{\!n}_{_P} \times \triangle^{\!n}_{_P} \rightarrow \mathbb C$ given by
\beq \label{exp-2}
 \quad \mathscr K_{_{P, m}}(z,w) = \prod_{j=2}^n\frac{1}{z_j \overline{w}_j} \prod_{j=1}^n \frac{1}{\Big(1-P_j(\varphi(z) \diamond \overline{\varphi(w)})\Big)^{m_j}}, 
~ z, w \in \triangle^{\!n}_{_P}
\eeq
is positive semi-definite, where $\varphi$ is given by \eqref{tilde}. 
\end{lemma}
\begin{proof} Note that $(z \diamond w)^{\alpha} = z^\alpha w^{\alpha}$ for any $z, w \in \mathbb C^n$ and $\alpha \in \mathbb Z^n_+.$ Combining this identity (by letting $z =\varphi(z)$ and $w = \overline{\varphi(w)}$) together with \eqref{exp} and the fact that the coefficient-function $A_{P, m}$ takes nonnegative values yields 
%$$\Big(\varphi(z) \diamond \overline{\varphi(w)}\Big)^{\alpha} = \varphi(z)^\alpha \overline{\varphi(w)}^{\alpha}, \quad \alpha \in \mathbb Z^n_+,$$ 
the desired conclusion. 
\end{proof}

Let $P$ be a positive regular polynomial $n$-tuple and let $m \in \mathbb N^n.$
By Lemma~\ref{lem-Moore} and Moore's theorem (see \cite[Theorem~2.14]{PR2016}), there exists a unique reproducing kernel Hilbert space $\mathscr H^2_m(\triangle^{\!n}_{_P})$ associated with the kernel $\mathscr K_{_{P, m}}.$ 
In particular, $\mathscr K_{_{P, m}}$ has the reproducing property$:$
\beq \label{rp}
\inp{f}{\mathscr K_{_{P, m}}(\cdot, w)} = f(w), \quad f \in \mathscr H^2_m(\triangle^{\!n}_{_P}), ~w \in \triangle^{\!n}_{_P}. 
\eeq
We refer to $\mathscr H^2_m(\triangle^{\!n}_{_P})$ as the {\it Hardy-Hilbert space of the generalized Hartogs triangle $\triangle^{\!n}_{_P}.$} 
Since $\mathscr K_{_{P, m}}$ is holomorphic in $z$ and $\overline{w},$ the elements of $\mathscr H^2_m(\triangle^{\!n}_{_P})$ are functions holomorphic on $\triangle^{\!n}_{_P}.$

Unless stated otherwise, $n$ is a positive integer bigger than or equal to $2.$ 
In what follows, we always assume that
\begin{itemize}
%[\tiny${\bullet}$]
\vskip0.2cm
\item[\tiny $\bullet$] $P$ is a positive regular polynomial $n$-tuple and  $m \in \mathbb N^n,$
\item[\tiny $\bullet$] $\triangle^{\!n}_{_P}$ is the $P$-triangle given by \eqref{gen-H-tri}, 
\item[\tiny $\bullet$] $A_{P, m} : \mathbb Z^n_+ \rar (0, \infty)$ is the coefficient-function associated with $(P, m)$ (see Lemma~\ref{APM-positive} and \eqref{exp}),
 \item[\tiny $\bullet$]  $\mathscr K_{_{P, m}}$ is the positive semi-definite kernel given by 
\eqref{exp-2}, 
\item[\tiny $\bullet$] $\mathscr H^2_m(\triangle^{\!n}_{_P})$ is the reproducing kernel Hilbert space associated with the kernel $\mathscr K_{_{P, m}}.$ 
\vskip0.2cm
\end{itemize}

We collect some basic properties of the Hardy-Hilbert space $\mathscr H^2_m(\triangle^{\!n}_{_P}).$
\begin{proposition} \label{on-basis}
For $\alpha \in \mathbb Z^n_+,$ define 
$e_\alpha : \triangle^{\!n}_{_P} \rar \mathbb C$ by
\beq \label{e-alpha}
e_\alpha(z) = \frac{\sqrt{A_{P, m}(\alpha)} \, \varphi(z)^{\alpha}}{\prod_{j=2}^nz_j}, \quad z \in \triangle^{\!n}_{_P},
\eeq
where $\varphi$ is as given in \eqref{tilde}. Then $\{e_\alpha\}_{\alpha \in \mathbb Z^n_+}$ forms an orthonormal basis for $\mathscr H^2_m(\triangle^{\!n}_{_P}).$ Moreover, the following statements are valid$:$
\begin{itemize}
\item[$\mathrm{(i)}$] the linear span of 
\beqn
\prod_{j=1}^{n-1}\Big(\frac{z_j}{z_{j+1}}\Big)^{\alpha_j} \frac{z^{\alpha_n}_n}{\prod_{j=2}^n z_j}, \quad \alpha_1, \ldots, \alpha_n \in \mathbb Z_+,
\eeqn
is dense in $\mathscr H^2_m(\triangle^{\!n}_{_P}),$
\item[$\mathrm{(ii)}$] for any $f \in \mathscr H^2_m(\triangle^{\!n}_{_P}),$ the series $\sum_{\alpha \in \mathbb Z^n_+}\inp{f}{e_\alpha}e_\alpha$ converges compactly on $\triangle^{\!n}_{_P}$ to $f.$ 
\end{itemize}
\end{proposition}
\begin{proof}
By \eqref{exp} and \eqref{exp-2}, 
\beq \label{kernel-basis-exp}
\mathscr K_{_{P, m}}(z, w) = \sum_{\alpha \in \mathbb Z^n_+} e_\alpha(z)\overline{e_\alpha(w)}, 
\eeq
where the series on the right hand converges compactly on $\triangle^{\!n}_{_P} \times \triangle^{\!n}_{_P}.$ 
By a theorem of Papadakis (see \cite[Theorem 2.10 and Exercise~3.7]{PR2016}), for every $\alpha \in \mathbb Z^n_+,$ $e_\alpha \in \mathscr H^2_m(\triangle^{\!n}_{_P})$ 
and $\{e_\alpha\}_{\alpha \in \mathbb Z^n_+}$ forms a Parseval's frame\footnote{A set of vectors $\{ f_s : s \in S\} \subseteq H$ is called a {\it Parseval frame} for the Hilbert space $H$ if 
for every $h \in H,$
$\|h\|^2 = \sum_{s \in S}|\inp{h}{f_s}|^2.$} for $\mathscr H^2_m(\triangle^{\!n}_{_P}).$ Hence, by 
\cite[Proposition~2.8]{PR2016}, 
\beqn
e_\alpha = \sum_{\beta \in \mathbb Z^n_+} \inp{e_\alpha}{e_\beta}e_\beta, \quad \alpha \in \mathbb Z^n_+.
\eeqn
In view of \eqref{e-alpha}, 
this is equivalent to
\begin{equation*}
\sqrt{A_{P, m}(\alpha)} \, \varphi(z)^{\alpha}(\|e_\alpha\|^2-1) + \sum_{\substack{\beta \in \mathbb Z^n_+ \\ \beta \neq \alpha}} \inp{e_\alpha}{e_\beta}\sqrt{A_{P, m}(\beta)} \, \varphi(z)^{\beta} =0, \, z \in \triangle^{\!n}_{_P}.
\end{equation*}
Letting $w=\varphi(z),$ we get 
\beqn
\sqrt{A_{P, m}(\alpha)} \, w^{\alpha} (\|e_\alpha\|^2-1) + \sum_{\substack{\beta \in \mathbb Z^n_+ \\ \beta \neq \alpha}} \inp{e_\alpha}{e_\beta} \sqrt{A_{P, m}(\beta)} \, w^{\beta} =0, \, w \in \varphi(\triangle^{\!n}_{_P}). 
\eeqn
Hence, the power series above is identically zero on its domain of convergence. Consequently, $\|e_\alpha\|^2=1$ and $\inp{e_\alpha}{e_\beta} =0$ for every $\beta \in \mathbb Z^n_+$ such that $\beta \neq \alpha.$ Thus $\{e_\alpha\}_{\alpha \in \mathbb Z^n_+}$ forms an orthonormal basis for $\mathscr H^2_m(\triangle^{\!n}_{_P}).$
This also yields (i). 
To see (ii),  observe that for any compact subset $K$ of $\triangle^{\!n}_{_P},$ $ \sup_{w \in K} \mathscr K_{_{P, m}}(w, w) < \infty,$ and use \eqref{rp}. 
\end{proof}

We now present some examples of Hardy-Hilbert spaces of the generalized Hartogs triangles that are central to the present investigations. 
\begin{example}[Example~\ref{exm-triangle} continued $\cdots$] \label{exm-triangle-2}
It is easy to see from Lemma~\ref{lem-Moore} that the reproducing kernel $\mathscr K_{P_a, m}$ (for short, $ \mathscr K_{a, m}$) of the Hilbert space $\mathscr H^2_m(\triangle^{\!n}_{a})$ is given by 
\beq \label{exp-hered}
&&\mathscr K_{a, m}(z, w) \\ &=& \prod_{j=2}^n\frac{1}{z_j \overline{w}_j} \Big( \prod_{j=1}^{n-1} \frac{1}{\Big(1-\frac{z_j \overline{w}_j}{z_{j+1} \overline{w}_{j+1}}-a z_1 \overline{w}_1\Big)^{m_j}}\Big)\frac{1}{(1-z_n \overline{w}_n-az_1 \overline{w}_1)^{m_n}}, \notag \\ && \quad \quad  \quad \quad\quad \quad \quad \quad\quad \quad \quad \quad\quad \quad\quad \quad\quad \quad \quad \quad\quad \quad\quad \quad z, w \in \triangle^{\!n}_{a}. \notag 
\eeq
Here are some important cases$:$
\begin{itemize}
\item[$\mathrm{(i)}$] In case of  $a=0,$  $\mathscr K_{a, m}$ takes the form
\begin{equation*}
 \mathscr K_{a, m}(z,w)  = \prod_{j=2}^n\frac{1}{z_j \overline{w}_j} \Big( \prod_{j=1}^{n-1} \frac{1}{\Big(1-\frac{z_j \overline{w}_j}{z_{j+1} \overline{w}_{j+1}}\Big)^{m_j}}\Big)\frac{1}{(1-z_n \overline{w}_n)^{m_n}}, \, z, w \in \triangle^{\!n}_{a}. 
\end{equation*}
\item[$\mathrm{(ii)}$] In case of $a=1$ and $n=2,$ $\mathscr K_{a, m}$ takes the form
\begin{equation*}
\mathscr K_{a, m}(z,w)  = \frac{1}{z_2 \overline{w}_2} \frac{1}{\Big(1-\frac{z_1 \overline{w}_1}{z_{2} \overline{w}_{2}}-  z_1 \overline{w}_1\Big)^{m_1}(1-z_1 \overline{w}_1- z_2 \overline{w}_2)^{m_2}} , \, z, w \in \triangle^{\!n}_{a}. 
\end{equation*}
\end{itemize}
If $m_j=1,$ $j=1, \ldots, n,$ then we denote the space $\mathscr H^2_m(\triangle^{\!n}_{_a})$ by $\mathscr H^2(\triangle^{\!n}_{_a}).$ 
\eof
\end{example}

\section{Multiplication by the coordinate functions} \label{S4}

In this section, we show that the coordinate functions $z_1, \ldots, z_n$ are multipliers of $\mathscr H^2_m(\triangle^{\!n}_{_P})$ and provide upper and lower bounds for the multiplier norm of the multiplication operators $\mathscr M_{z_j},$ $j=1, \ldots, n,$ on the Hardy-Hilbert space $\mathscr H^2_m(\triangle^{\!n}_{_P}).$ 

Let $j=1, \ldots, n$ and let $\mathscr M_{z_j}$ denote the linear operator of multiplication by the coordinate function $z_j$ in $\mathscr H^2_m(\triangle^{\!n}_{_P})$:
\beqn
\mathscr M_{z_j}f=z_j f~\mbox{whenever}~f \in \mathscr H^2_m(\triangle^{\!n}_{_P})~\mbox{and}~{z_j} f \in \mathscr H^2_m(\triangle^{\!n}_{_P}).
\eeqn 
One may rewrite $z_j$ as $\prod_{k=j}^{n-1}\Big(\frac{z_k}{z_{k+1}}\Big) z_n$  and 
note that by \eqref{tilde},
\beqn
\mathscr M_{z_j} \varphi(z)^{\alpha} &=& z_j \prod_{k=1}^{n-1}\Big(\frac{z_k}{z_{k+1}}\Big)^{\alpha_k} z^{\alpha_n}_n \\
&=& \prod_{k=1}^{j-1}\Big(\frac{z_k}{z_{k+1}}\Big)^{\alpha_k} \prod_{k=j}^{n-1}\Big(\frac{z_k}{z_{k+1}}\Big)^{\alpha_k + 1}z^{\alpha_n +1}_n \\
&=& \varphi(z)^{\alpha + \sum_{k=j}^n \varepsilon_k}, \quad \alpha \in \mathbb Z^n_+.
\eeqn
This combined with \eqref{e-alpha} shows that 
\beq \label{action-basis}
\mathscr M_{z_j} e_\alpha = \frac{\sqrt{A_{P, m}(\alpha)}}{\sqrt{A_{P, m}(\alpha + \sum_{k=j}^n \varepsilon_k)}}\,e_{\alpha +  \sum_{k=j}^n \varepsilon_k}, \quad \alpha \in \mathbb Z^n_+.
\eeq
We need the following lemma in the proof of Proposition~\ref{joint-s}.
\begin{remark} 
%Let $\mathscr M_z$ be the multiplication $n$-tuple on $\mathscr H^2_m(\triangle^{\!n}_{_P}).$ 
For $\beta \in \mathbb Z^n_+,$ we claim that 
\beq 
\label{powers-orthogonal}
\mbox{the family $\{z^{\beta} e_\alpha\}_{\alpha \in \mathbb Z^n_+}$ is orthogonal in $\mathscr H^2_m(\triangle^{\!n}_{_P}).$}
\eeq
Indeed, by repeated applications of \eqref{action-basis}, we obtain  
\beqn
z^{\beta_j}_{j} e_\alpha =   \frac{\sqrt{A_{P, m}(\alpha)}}{\sqrt{A_{P, m}(\alpha + \beta_j \sum_{k=j}^n \varepsilon_k)}}\,e_{\alpha +  \beta_j \sum_{k=j}^n \varepsilon_k}, ~ \alpha \in \mathbb Z^n_+, ~j=1, \ldots, n,
\eeqn
and hence 
\beq \label{action-power-Mz}
z^{\beta} e_\alpha =   \frac{\sqrt{A_{P, m}(\alpha)}}{\sqrt{A_{P, m}(\alpha + \sum_{j=1}^n \beta_j \sum_{k=j}^n \varepsilon_k)}}\,e_{\alpha +  \sum_{j=1}^n \beta_j \sum_{k=j}^n \varepsilon_k}, ~ \alpha \in \mathbb Z^n_+.
\eeq  
This combined with the fact that $\{e_\alpha\}_{\alpha \in \mathbb Z^n_+}$ is orthogonal in $\mathscr H^2_m(\triangle^{\!n}_{_P})$ (see Proposition~\ref{on-basis}) completes the verification of \eqref{powers-orthogonal}. 
\hfill 
$\diamondsuit$
\end{remark}
It turns out that each $\mathscr M_{z_j}$ extends to a bounded linear operator on $\mathscr H^2_m(\triangle^{\!n}_{_P}).$

\begin{proposition} \label{bddness}
The multiplication $n$-tuple $\mathscr M_z = (\mathscr M_{z_1}, \ldots, \mathscr M_{z_n})$ defines a commuting $n$-tuple  on $\mathscr H^2_m(\triangle^{\!n}_{_P}).$ Moreover, 
\beq
\label{norm-estimate}
\|\mathscr M_{z_j}\| \, \Le \,  \frac{1}{\sqrt{\prod_{l=j}^n \partial_l P_l(0)}}, \quad j=1, \ldots, n.
\eeq
%In particular, if $P$ is $n$-admissible, then 
%\beq
%\|\mathscr M_{z_j}\| \Ge \frac{1}{\sqrt{\prod_{l=j}^nm_l \partial_l P_l(0)}}.
%\eeq 
\end{proposition}
\begin{proof}
For $\alpha \in \mathbb Z^n_+,$ by \eqref{rho-alpha-f}, 
\allowdisplaybreaks
\beq \label{S-lemma}
&& A_{P, m}\Big(\alpha + \sum_{k=j}^n \varepsilon_{k}\Big) \\ &=& \notag \displaystyle \sum_{{\substack{\gamma^{(1)}, \ldots, \gamma^{(n)} \in \mathbb Z^n_+ \\  \sum_{k=1}^{j-1} \gamma^{(k)} +  \sum_{k=j}^n (\gamma^{(k)} -\varepsilon_{k}) = \alpha}}} \prod_{l=1}^n A_{P_l, m_l}(\gamma^{(l)}) 
\\
&=& \sum_{{\substack{\gamma^{(1)}, \ldots, \gamma^{(n)} \in \mathbb Z^n_+ \\ \sum_{k=1}^{n} \gamma^{(k)} = \alpha}}} \Big(\prod_{l=1}^{j-1} A_{P_l, m_l}(\gamma^{(l)})\Big)\Big(\prod_{l=j}^n A_{P_l, m_l}(\gamma^{(l)}+\varepsilon_l) \Big). \notag 
\eeq
Also, by (i) and (iii) of Lemma~\ref{Spott} (applied to $Q=P_l$), 
\beq  \label{APML}
\notag 
A_{P_l, m_l}(\alpha+\varepsilon_l)  &=& A_{P_l, m_l - 1}(\alpha +\varepsilon_l) + \displaystyle \sum_{\gamma \in \mathbb Z^n_+} q^{(l)}_{\gamma} \, A_{P_l, m_l}(\alpha+\varepsilon_l - \gamma) \\ &\Ge & q^{(l)}_{\varepsilon_l} \, A_{P_l, m_l}(\alpha), \quad l=1, \ldots, n, ~\alpha \in \mathbb Z^n_+. \eeq
Replacing $\alpha $ by $\gamma^{(l)}$ and applying \eqref{rho-alpha-f} and \eqref{S-lemma} gives
\beqn
A_{P, m}\Big(\alpha + \sum_{k=j}^n \varepsilon_{k}\Big) & \Ge & \Big(\prod_{l=j}^n q^{(l)}_{\varepsilon_l}\Big) \sum_{{\substack{\gamma^{(1)}, \ldots, \gamma^{(n)} \in \mathbb Z^n_+ \\ \sum_{k=1}^{n} \gamma^{(k)} = \alpha}}} \Big(\prod_{l=1}^n A_{P_l, m_l}(\gamma^{(l)})\Big) \\ &=& \Big(\prod_{l=j}^n q^{(l)}_{\varepsilon_l}\Big) A_{P, m}(\alpha). 
\eeqn
Since $q^{(l)}_{\varepsilon_l}=\partial_l P_l(0),$ $l=1, \ldots, n$ (see Definition~\ref{reg-def}(ii)), it now follows from \eqref{action-basis} that for $j=1, \ldots, n,$ 
\beq \label{norm-at-0}
\|\mathscr M_{z_j}e_\alpha\| = \frac{\sqrt{A_{P, m}(\alpha)}}{\sqrt{A_{P, m}(\alpha + \sum_{k=j}^n \varepsilon_k)}} \Le \frac{1}{\sqrt{\prod_{l=j}^n \partial_l P_l(0)}}, \quad \alpha \in \mathbb Z^n_+,
\eeq
where $e_\alpha$ is given by \eqref{e-alpha}. 
Since $\{e_\alpha\}_{\alpha \in \mathbb Z^n_+}$ forms an orthonormal basis for $\mathscr H^2_m(\triangle^{\!n}_{_P})$ (see Proposition~\ref{on-basis}) and $\{\mathscr M_{z_j} e_\alpha\}_{\alpha \in \mathbb Z^n_+}$ forms an orthogonal set  for every $j=1, \ldots, n$ (see \eqref{powers-orthogonal}),  
we obtain \eqref{norm-estimate}. 
\end{proof}

For $n$-admissible polynomial $n$-tuples, we can also provide a lower bound for the operator norm of the multiplication by a coordinate function.
\begin{corollary} \label{norm-attained} 
%Let $\mathscr M_z = (\mathscr M_{z_1}, \ldots, \mathscr M_{z_n})$ be the multiplication $n$-tuple on $\mathscr H^2_m(\triangle^{\!n}_{_P}).$ 
Assume that $P$ is $n$-admissible. Then the multiplication $n$-tuple $\mathscr M_z = (\mathscr M_{z_1}, \ldots, \mathscr M_{z_n})$ on $\mathscr H^2_m(\triangle^{\!n}_{_P})$ satisfies 
\beqn
\frac{1}{\sqrt{\prod_{l=j}^nm_l \partial_l P_l(0)}} \Le \|\mathscr M_{z_j}\| \, \Le \, \frac{1}{\sqrt{\prod_{l=j}^n \partial_l P_l(0)}}, \quad j=1, \ldots, n.
\eeqn
Moreover, if $r_j= \frac{1}{\sqrt{\prod_{l=j}^n \partial_l P_l(0)}},$ $j=1, \ldots, n,$ then the following are valid$:$
\begin{itemize} 
\item[$\mathrm{(i)}$] 
$\sigma(\mathscr M_z) \subseteq \overline{\mathbb D}(0, r_1) \times \cdots \times \overline{\mathbb D}(0, r_n),$
\item[$\mathrm{(ii)}$] 
if $m_l=1, ~l=1, \ldots, n,$ then $\|\mathscr M_{z_j}\| =r_j.$
%\beqn
%\|\mathscr M_{z_j}\| =\frac{1}{\sqrt{\prod_{l=j}^n \partial_l P_l(0)}}, \quad j=1, \ldots, n. 
%\eeqn
\end{itemize}
\end{corollary}
\begin{proof}
Fix $j=1, \ldots, n.$ 
It follows from the first equality of \eqref{norm-at-0} that  
\beq \label{lower-bd-Mz}
\|\mathscr M_{z_j}\| \Ge \frac{\sqrt{A_{P, m}(0)}}{\sqrt{A_{P, m}(\sum_{k=j}^n \varepsilon_k)}} \overset{\eqref{rho-alpha-f-new}}=  \frac{1}{\sqrt{A_{P, m}(\sum_{k=j}^n \varepsilon_k)}}.
\eeq
To compute $A_{P, m}(\sum_{k=j}^n \varepsilon_k),$ consider 
\beq \label{admissible-norm}
&& \sum_{\alpha \in \mathbb Z^n_+} A_{P_j, m_j}(\alpha) {z}^\alpha 
\overset{\eqref{coeff-Pj-mj}}=
\frac{1}{(1-P_j(z))^{m_j}} \notag \\ 
&=&   1 + m_j P_j(z) + \sum_{k=2}^\infty \binom{m_j+k-1}{m_j-1} P_j(z)^k, 
\eeq
where all the series above converge in a neighborhood of the origin. 
Since $P$ is $n$-admissible, by \eqref{eqn-d-admissible}, for every $j=1, \ldots, n,$
\beqn
P_j(z) = \sum_{k=1}^n \frac{(\partial^k_j P_j)(0)}{k!}z^k_j + \displaystyle \sum_{k = n+1}^{N_j}Q_{jk}(z), \quad z=(z_1, \ldots, z_n) \in \mathbb C^n,
\eeqn
where $Q_{jk}$ is a homogeneous polynomial of degree $k.$ 
This combined with \eqref{admissible-norm} shows that
\beqn
A_{P_j, m_j}(\alpha) = \begin{cases} 0 & \mbox{if~}|\alpha| \Le n, ~\alpha \neq k \varepsilon_j, ~k =1, \ldots, n, \\
m_j \partial_j P_j(0) & \mbox{if~}\alpha = \varepsilon_j.
\end{cases} 
\eeqn
It is now clear from \eqref{rho-alpha-f} that 
\beqn
A_{P, m}(\sum_{k=j}^n \varepsilon_k)=
\prod_{k=j}^n A_{P_k, m_k}(\varepsilon_k) =
\prod_{k=j}^nm_j \partial_j P_j(0). 
\eeqn
This combined with \eqref{norm-estimate} and \eqref{lower-bd-Mz} yields the first half. This also gives part (ii).  The part (i) follows from the norm estimate \eqref{norm-estimate} and the following fact$:$ For any commuting $n$-tuple $T=(T_1, \ldots, T_n)$ on $H,$ 
$$\sigma(T) \subseteq \sigma(T_1) \times \cdots \times \sigma(T_n) \subseteq \overline{\mathbb D}(0, \|T_1\|) \times \cdots \times \overline{\mathbb D}(0, \|T_n\|)$$ 
(see \cite[Lemma~4.6]{C1988}).
\end{proof}
\begin{remark} \label{noeigen}
Note that none of $\mathscr M_{z_j},$ $j=1, \ldots, n,$ has any eigenvalue. To see this, assume that for some $j=1, \ldots, n,$ $\lambda \in \mathbb C$ and $f \in \mathscr H^2_m(\triangle^{\!n}_{_P}),$ $\mathscr M_{z_j}f=\lambda f,$ and note that  
\beqn
(z_j-\lambda)f(z) =0, \quad z \in \triangle^{\!n}_{_P},
\eeqn
and hence $f=0$ on $\triangle^{\!n}_{_P} \backslash Z(z_j-\lambda).$ Since $\triangle^{\!n}_{_P}$ is connected (see Proposition~\ref{DoH}) and $f$ is continuous, $f$ is identically zero. In particular, the point spectrum of $\mathscr M_z$ is empty. \hfill $\diamondsuit$
\end{remark}

For $j=1, \ldots, n,$ let $\mathscr W_j$ denote the linear operator defined by 
\beq
\label{wt-shift-action}
\mathscr W_j e_{\alpha} =  \frac{\sqrt{A_{P, m}(\alpha)}}{\sqrt{A_{P, m}(\alpha + \varepsilon_j)}}\,e_{\alpha +  \varepsilon_j}, \quad \alpha \in \mathbb Z^n_+,
\eeq 
which is extended linearly to the linear span of $e_\alpha,$ $\alpha \in \mathbb Z^n_+.$ 
%\end{definition}
As in the proof of Proposition~\ref{bddness},  it is easy to see from \eqref{rho-alpha-f} (with $\alpha$ replaced by $\alpha + \varepsilon_j$) and \eqref{APML} 
that 
\beqn
\frac{\sqrt{A_{P, m}(\alpha)}}{\sqrt{A_{P, m}(\alpha + \varepsilon_j)}} \Le 
\frac{1}{\sqrt{\partial_j P_j(0)}}, \quad j=1, \ldots, n. 
\eeqn
It follows that 
$\mathscr W_1, \ldots, \mathscr W_n$ extend to bounded linear operators on $\mathscr H^2_m(\triangle^{\!n}_{_P}).$ Clearly, $\mathscr W=(\mathscr W_1, \ldots, \mathscr W_n)$ is a commuting $n$-tuple. 
Moreover, a routine verification using \eqref{action-basis} and \eqref{wt-shift-action} shows that 
\beq \label{shift-multi}
\mathscr M_{z_j}  = \prod_{k=j}^n \mathscr W_k, \quad j = 1, \ldots, n.
\eeq
We refer to the commuting $n$-tuple $\mathscr W$ as {\it the unilateral weighted multishift associated with} $\mathscr M_z.$ 

By Proposition~\ref{bddness}, for every $j=1, \ldots, n,$ the Hilbert space adjoint  $\mathscr M^*_{z_j}$ of $\mathscr M_{z_j}$ defines a bounded linear operator on $\mathscr H^2_m(\triangle^{\!n}_{_P}).$
%For $j=1, \ldots, n,$ let $\mathscr M^*_{z_j}$ denote the Hilbert space adjoint  of $\mathscr M_{z_j}.$ 
Let $\alpha \in \mathbb Z^n_+$ and $j=1, \ldots, n.$
Note that 
\beqn
\mathscr M^*_{z_j} e_{\alpha} = \prod_{k=j}^n \mathscr W^*_k e_{\alpha} =
\begin{cases} 0 & ~\mbox{if~}\alpha_n =0, \\ \Big(\prod_{k=j}^{n-1} \mathscr W^*_k\Big) \frac{\sqrt{A_{P, m}(\alpha-\varepsilon_n)}}{\sqrt{A_{P, m}(\alpha)}}\,e_{\alpha -  \varepsilon_n} & ~\mbox{otherwise}.
\end{cases}
\eeqn
Continuing this,  we obtain 
\beq \label{action-adjoint}
\mathscr M^*_{z_j} e_{\alpha} = \begin{cases}  0 & \mbox{if}~ \alpha - \sum_{k=j}^n \varepsilon_k \notin \mathbb Z^n_+, \\
\frac{\sqrt{A_{P, m}(\alpha - \sum_{k=j}^n \varepsilon_k)}}{\sqrt{A_{P, m}(\alpha)}}\,e_{\alpha -  \sum_{k=j}^n \varepsilon_k} & \mbox{otherwise}.
\end{cases}
\eeq
\begin{remark} \label{after-action-adjoint}
It turns out that $\sigma_p(\mathscr M^*_z)$ is disjoint from $$\{w \in \mathbb C^n : w_j \neq 0\,\mbox{and\,}w_{k}=0\,\mbox{for some}\,1 \Le j < k \Le n\}.$$ Indeed, by \eqref{shift-multi}, $\mathscr M^*_{z_j} = \Big(\prod_{l=j}^{k-1}\mathscr W^*_l\Big) \mathscr M^*_{z_k},$ and hence $\ker \mathscr M^*_{z_k} \subseteq \ker \mathscr M^*_{z_j}$ for $1 \Le j < k \Le n,$ which implies that if $w_k=0,$ then $w_j=0.$  \hfill $\diamondsuit$
\end{remark}

Here is another instance in which the notion of the unilateral weighted multishift associated with $\mathscr M_z$ yields insight into the spectral theory of $\mathscr M_z.$  
\begin{proposition} \label{spectral-coro-new}
The Taylor spectrum of the multiplication $n$-tuple $\mathscr M_z$ on $\mathscr H^2_m(\triangle^{\!n}_{_P})$ is a connected subset of $\mathbb C^n$ that contains $\overline{\triangle^{\!n}_{_P}}.$
%In particular, $\mathscr M_z$ is never Taylor invertible. 
\end{proposition}
\begin{proof}
An application of the spectral mapping property (see \cite[Corollary~3.5]{C1988}) together with \eqref{shift-multi} shows that
\beqn
\sigma(\mathscr M_z) = \Big\{\Big(\prod_{j=1}^n z_j, \prod_{j=2}^n z_j, \ldots, z_n\Big) \in \mathbb C^n : z \in \sigma(\mathscr W)\Big\} \overset{\eqref{phi-inverse}}= \varphi^{-1}\Big(\sigma(\mathscr W)\Big).
\eeqn
Since $\sigma(\mathscr W)$ is connected (see \cite[Proposition 3.2.4]{CPT2017}) and $\varphi^{-1}$ is continuous on $\mathbb C^n,$ $\sigma(\mathscr M_z)$ is connected.  
%One may now apply Corollary~\ref{eigen-v}.
To see the remaining part, note that the  reproducing property \eqref{rp} of $\mathscr K_{_{P, m}}$ yields the following$:$
\beq \label{eigen-v}
\mathscr M^*_{z_j} \mathscr K_{_{P, m}}(\cdot, w)=\overline{w}_j \mathscr K_{_{P, m}}(\cdot, w), \quad w \in \triangle^{\!n}_{_P}, ~j=1, \ldots, n.
\eeq
This gives the inclusion $\triangle^{\!n}_{_P} \subseteq \sigma_p(\mathscr M^*_z).$ It follows that 
\beqn
\triangle^{\!n}_{_P} \subseteq \sigma(\mathscr M^*_z)=\{z \in \mathbb C^n : \overline{z} \in \sigma(\mathscr M_z)\}.
\eeqn
Since
$\triangle^{\!n}_{_P}$ is a Reinhardt domain and $\sigma(\mathscr M_z)$ is closed in $\mathbb C^n,$ $\overline{\triangle^{\!n}_{_P}} \subseteq \sigma(\mathscr M_z),$ which completes the proof. 
\end{proof}

\section{Circularity, trace estimates and contractivity} \label{S5} 

In this section, we discuss several basic properties of the multiplication $n$-tuple $\mathscr M_z$ on $\mathscr H^2_m(\triangle^{\!n}_{_P}).$ These include failure of doubly commutativity, finite cyclicity and essential normality. We also discuss 
circularity, trace-class membership of the determinant operator and $\triangle^{\!n}_{_P}$-contractivity.

Let $T=(T_1, \ldots, T_n)$ be a commuting $n$-tuple on $H.$ We say that $T$ is 
\begin{itemize}
\item[\tiny $\bullet$] {\it doubly commuting} if $[T^*_j, T_k]=0$ for every $1 \Le j \neq k \Le n,$
\item[\tiny $\bullet$] {\it essentially normal} if $[T^*_j, T_k]$ is compact for every $j, k =1, \ldots, n,$
\item[\tiny $\bullet$] {\it finitely rationally cyclic} if for some finite subset $F$ of $H,$ 
\begin{equation*}
H = \bigvee \{r(T)h : r ~\mbox{is a rational function with poles off}~\sigma(T)~\mbox{and}~h \in F\}.
\end{equation*}
\end{itemize}  
Following \cite{MPS2022}, we define the {\it determinant operator} $\det \big([T^*, T]\big)$ of a commuting pair $T=(T_1, T_2)$ by
\beq \label{det-def} \det \big([T^*, T]\big)  \notag
&=& [T^*_1, T_1][T^*_2, T_2]+[T^*_2, T_2][T^*_1, T_1] \\ &-& [T^*_1, T_2][T^*_2, T_1]-[T^*_2, T_1][T^*_1, T_2].
\eeq
The reader is referred to \cites{AW1990, C1988, MPS2022} for some basic facts pertaining to rational cyclicity, essential normality  and related notions.

The following proposition summarizes some properties of the multiplication $n$-tuples $\mathscr M_z$ on $\mathscr H^2_m(\triangle^{\!n}_{_P}),$ which makes them different from their polydisc counterpart.  
\begin{proposition} \label{not-F} Let $\mathscr M_z$ be the multiplication $n$-tuple on $\mathscr H^2_m(\triangle^{\!n}_{_P}).$
The following statements are valid$:$
\begin{itemize}
\item[$\mathrm{(i)}$] $\mathscr M_z$ is never doubly commuting,
\item[$\mathrm{(ii)}$] 
$\mathscr M_z$ is never finitely rationally cyclic.
%\item[$\mathrm{(iii)}$] $\mathscr M_z$ is never Fredholm.
%\marginpar{\color{blue}not essentially doubly commuting}
\end{itemize}
\end{proposition}
\begin{proof}  
(i) Let $\alpha \in \mathbb Z^n_+.$ If $\alpha -  \varepsilon_{n} \notin \mathbb Z^n_+,$ then  
by \eqref{action-adjoint}, 
$\mathscr M_{z_n}\mathscr M^*_{z_{n-1}}(e_\alpha) = 0.$ If $\alpha -  \varepsilon_{n-1} \in \mathbb Z^n_+,$ then 
\beqn
\mathscr M^*_{z_{n-1}}\mathscr M_{z_n}(e_\alpha) &\overset{\eqref{action-basis}}=& \frac{\sqrt{A_{P, m}(\alpha)}}{\sqrt{A_{P, m}(\alpha + \varepsilon_n)}}\, \mathscr M^*_{z_{n-1}}(e_{\alpha +  \varepsilon_n}) \\
&\overset{\eqref{action-adjoint}}=& 
\frac{\sqrt{A_{P, m}(\alpha)}}{\sqrt{A_{P, m}(\alpha + \varepsilon_n)}}\, \frac{\sqrt{A_{P, m}(\alpha - \varepsilon_{n-1})}}{\sqrt{A_{P, m}(\alpha+\varepsilon_n)}}\,e_{\alpha -   \varepsilon_{n-1}},
\eeqn 
which is clearly nonzero.

(ii) By \eqref{action-adjoint}, 
\beq \label{0-pt-s}
\ker(\mathscr M^*_z) = \bigvee \{e_{\alpha} : \alpha \in \mathbb Z^n_+, \, \alpha_n =0\}.
\eeq
%which shows that $0 \in \sigma_p(\mathscr M^*_z).$ However, since 
%$$\sigma(\mathscr M_z)=\{z \in \mathbb C^n : \overline{z} \in \sigma(\mathscr M^*_z)\},$$ $\mathscr M_z$ is not Taylor invertible. 
Thus the joint kernel of $\mathscr M^*_z$ is of infinite dimension. One may now apply \cite[Proposition~1.1]{AW1990}. 
%(iii)
% This may be concluded from \eqref{0-pt-s} and \cite[Property~(v), p.~73]{C1988}.
\end{proof}

If $\Omega$ is a complete Reinhardt domain in $\mathbb C^2,$ which is not pseudoconvex, then the multiplication $2$-tuple on the Bergman space of $\Omega$ is not essentially normal (see \cite[Theorem~4.9(g)]{CS1985}).  
Interestingly, the following proposition yields a
non-complete pseudoconvex Reinhardt domain $\triangle^{\!n}_{_P}$ (see Proposition~\ref{always-DoH})
%a large class of positive regular polynomial $n$-tuples $P$ 
for which the multiplication $n$-tuple $\mathscr M_z$ on $\mathscr H^2_m(\triangle^{\!n}_{_P})$ is not essentially normal. 
\begin{proposition} 
\label{not-en}
If $P=(P_1, \ldots, P_n)$ is a positive regular polynomial $n$-tuple such that for some $i=1, \ldots, n-1,$ 
\beq \label{condition-n-ess}
\partial_n P_i=0~\mbox{and}~ \partial_i P_j =0, \quad 1 \Le j \neq i \Le n,
\eeq 
then the multiplication operator $\mathscr M_{z_n}$ on $\mathscr H^2_m(\triangle^{\!n}_{_P})$ is not essentially normal.
\end{proposition}
\begin{proof} Assume that \eqref{condition-n-ess} holds for some $i=1, \ldots, n-1.$ 
%It suffices to check that the self-commutator $C$ of $\mathscr M_{z_n}$ is never compact. 
Note that for $\alpha = \ell \varepsilon_i,$ by \eqref{wt-shift-action}, \eqref{shift-multi} and \eqref{action-adjoint}, 
\beqn
\inp{[\mathscr M^*_{z_n}, \mathscr M_{z_n}]e_{\alpha}}{e_\alpha} =\|\mathscr M_{z_n}e_\alpha \|^2 = \frac{{A_{P, m}(\ell \varepsilon_i)}}{{A_{P, m}(\ell \varepsilon_i + \varepsilon_n)}}, \quad \ell \in \mathbb Z_+.
\eeqn
It is now sufficient to show that $\displaystyle \inf_{\ell \in \mathbb Z_+}\frac{{A_{P, m}(\ell \varepsilon_i)}}{{A_{P, m}(\ell \varepsilon_i + \varepsilon_n)}} > 0.$ 
By \eqref{rho-alpha-f}, for any $\ell \in \mathbb Z_+,$ 
\beq \label{trace-Apm}
{A_{P, m}(\ell \varepsilon_i + \varepsilon_n)} = \sum_{{\substack{\gamma^{(1)}, \ldots, \gamma^{(n)} \in \mathbb Z^n_+ \\ \sum_{k=1}^n \gamma^{(k)} = \ell \varepsilon_i + \varepsilon_n}}} \prod_{j=1}^n A_{P_j, m_j}(\gamma^{(j)}).
\eeq
For $1 \le j\neq i \Le n,$
by \eqref{condition-n-ess} and Lemma~\ref{Spott}(ii), 
\beq \label{provided} 
\mbox{$A_{P_j, m_j}(\gamma^{(j)}) =0$ for $\gamma^{(j)} \in \mathbb Z^n_+$ with 
$\gamma^{(j)}_i \neq 0.$}
\eeq 
Similarly, since 
$\partial_n P_i =0$ (see \eqref{condition-n-ess}), once again by Lemma~\ref{Spott}(ii), 
\beq \label{provided-new} 
\mbox{$A_{P_i, m_i}(\gamma^{(i)}) =0$ for $\gamma^{(i)} \in \mathbb Z^n_+$ with  $\gamma^{(i)}_n \neq 0.$}
\eeq
It is now easy to see using \eqref{rho-alpha-f-new}, \eqref{rho-alpha-f} and \eqref{provided} that $A_{P, m}(\ell \varepsilon_i) = A_{P_i, m_i}(\ell \varepsilon_i).$ 
This combined with \eqref{trace-Apm}-\eqref{provided-new} shows that 
for any $\ell \in \mathbb Z_+,$
\beqn
\frac{{A_{P, m}(\ell \varepsilon_i)}}{{A_{P, m}(\ell \varepsilon_i + \varepsilon_n)}} &=& \frac{{A_{P_i, m_i}(\ell \varepsilon_i)}}{\sum_{1 \Le j \neq i \Le n} {A_{P_i, m_i}(\ell \varepsilon_i)}{A_{P_j, m_j}(\varepsilon_n)}} \\
&=& \frac{1}{\sum_{1 \Le j \neq i \Le n} {A_{P_j, m_j}(\varepsilon_n)}}, 
\eeqn
which is a positive real number (see Lemma~\ref{Spott}(iii)) independent of $\ell.$ 
\end{proof}

The following is immediate from Remark~\ref{rmk-d-admissible} and Proposition~\ref{not-en}. 
\begin{corollary} \label{coro-not-en}
If $P$ is an admissible polynomial $n$-tuple, then the multiplication $n$-tuple $\mathscr M_{z}$ on $\mathscr H^2_m(\triangle^{\!n}_{_P})$ is not essentially normal.
\end{corollary}

\subsection{Circularity}

Following \cites{JL1979, AHHK1984}, we say that a commuting $n$-tuple $T=(T_1, \ldots, T_n)$ on $H$ is {\it circular} if for every $\theta := (\theta_1, \ldots, \theta_n) \in \mathbb R^n$, there exists a unitary operator $\varGamma_{\theta}$ on $H$ such that $$\varGamma^*_{\theta} T_j \varGamma_{\theta} = e^{\iota \theta_j} T_j, \quad  j=1, \ldots, n,$$ where $\iota$ is the unit imaginary number. 
We say that $T$ is {\it strongly circular} if, in addition, $\varGamma_{\theta}$ can be chosen to be a strongly continuous unitary representation of $\mathbb T^n$ in the following sense: For every $h \in \mathcal H$, the function
$\theta \mapsto \varGamma_{\theta}h$ is continuous on $\mathbb R^n.$
\begin{proposition} \label{p-circular}
The multiplication $n$-tuple $\mathscr M_z$ on $\mathscr H^2_m(\triangle^{\!n}_{_P})$ is strongly circular. 
%In particular, the Taylor spectrum of $\mathscr M_z$ is invariant under $\mathbb T^n$ and connected. 
\end{proposition}
\begin{proof} Let $\theta = (\theta_1, \ldots, \theta_n) \in \mathbb R^n.$ Since $\varphi$ maps $\mathbb T^n$ onto itself, there exists $\tilde{\theta} = (\tilde{\theta}_1, \ldots,  \tilde{\theta}_n) \in \mathbb R^n$ such that $$(e^{\iota \theta_1}, \ldots, e^{\iota\theta_n}) =\varphi^{-1}((e^{\iota \tilde{\theta}_1}, \ldots, e^{\iota\tilde{\theta}_n}))$$ 
(see \eqref{tilde}).
By \cite[Example~3.1.4 and Proposition~3.2.1]{CPT2017} (see also \cite[Corollary~3]{JL1979}), $\mathscr W$ is strongly circular. Thus, there exists a strongly continuous unitary representation $\varGamma_{\tilde{\theta}}$ of $\mathbb T^n$ on $\mathscr H^2_m(\triangle^{\!n}_{_P})$ such that $$\varGamma^*_{\tilde{\theta}} \mathscr W_j \varGamma_{\tilde{\theta}} = e^{\iota\tilde{\theta}_j}\mathscr W_j, \quad j=1, \ldots, n.$$
This combined with \eqref{shift-multi} (used twice) shows that 
\beqn
\mathscr M_{z_j} = \prod_{k=j}^n \mathscr W_k = \Big(\prod_{k=j}^n  e^{\iota\tilde{\theta}_k}\Big) \varGamma_{\tilde{\theta}} \mathscr M_{z_j} \varGamma^*_{\tilde{\theta}} \overset{\eqref{tilde}}= e^{\iota\theta_j} \varGamma_{\tilde{\theta}} \mathscr M_{z_j} \varGamma^*_{\tilde{\theta}}, \quad j=1, \ldots, n.  
\eeqn
This shows that $\mathscr M_z$ is strongly circular. 
\end{proof}

The following result contains some geometric information about the Taylor spectrum of $\mathscr M_z$ on $\mathscr H^2_m(\triangle^{\!n}_{_P}).$ 
\begin{corollary} \label{circular-spectral}
The Taylor spectrum of the multiplication $n$-tuple $\mathscr M_z$ on $\mathscr H^2_m(\triangle^{\!n}_{_P})$ is a $\mathbb T^n$-invariant connected subset of $\mathbb C^n$ that contains $\overline{\triangle^{\!n}_{_P}}.$
%In particular, $\mathscr M_z$ is never Taylor invertible. 
\end{corollary}
\begin{proof}
By Proposition~\ref{p-circular}, $\mathscr M_z$ is circular, and hence by the spectral mapping property (see \cite[Corollary~3.5]{C1988}),  $\sigma(\mathscr M_z)$  is $\mathbb T^n$-invariant. 
%Another application of the spectral mapping property together with \eqref{shift-multi} shows that
%\beqn
%\sigma(\mathscr M_z) = \Big\{\Big(\prod_{j=1}^n z_j, \prod_{j=2}^n z_j, \ldots, z_n\Big) \in \mathbb C^n : z \in \sigma(\mathscr W)\Big\} \overset{\eqref{phi-inverse}}= \varphi^{-1}\Big(\sigma(\mathscr W)\Big).
%\eeqn
%Since $\sigma(\mathscr W)$ is connected (see \cite[Proposition 3.2.4]{CPT2017}) and $\varphi^{-1}$ is continuous on $\mathbb C^n,$ $\sigma(\mathscr M_z)$ is connected.  
One may now apply Proposition~\ref{spectral-coro-new}.
\end{proof}

\subsection{Trace estimates}

In this subsection, we obtain trace estimates of the determinant operator of the multiplication $2$-tuple $\mathscr M_z$ for a family of reproducing kernel Hilbert spaces on generalized Hartogs triangles (cf. \cite[Definition 4.4]{MPS2022}). The reader is referred to \cite{Si2015} for the basic properties of trace class operators and related notions.

\begin{proposition} \label{B-S-Hartogs}
For an admissible polynomial $2$-tuple $P=(P_1, P_2),$ 
consider the sequences
\begin{equation*}
a_{_{P_1, m_1}}(k) = \frac{A_{P_1, m_1}(k \varepsilon_1)}{A_{P_1, m_1}((k +1)\varepsilon_1)}, ~ 
a_{_{P_2, m_2}}(k) =  \frac{A_{P_2, m_2}(k \varepsilon_2)}{A_{P_2, m_2}((k +1)\varepsilon_2)}, ~ k \in \mathbb Z_+.
\end{equation*}
If $\mathscr M_z$ is the multiplication $2$-tuple on $\mathscr H^2_m(\triangle^{\!2}_{_P}),$ then the following are valid$:$
\begin{itemize}
\item[$\mathrm{(i)}$] $\det\, [\mathscr M^*_z, \mathscr M_z]$ is positive if and only if the sequences $a_{_{P_1, m_1}}$ and $a_{_{P_2, m_2}}$ are increasing,
\item[$\mathrm{(ii)}$] if $\det\, [\mathscr M^*_z, \mathscr M_z]$ is positive, then $\det\, [\mathscr M^*_z, \mathscr M_z]$ is of trace-class if and only if the sequences $a_{_{P_1, m_1}}$ and $a_{_{P_2, m_2}}$ are bounded,
\item[$\mathrm{(iii)}$] if the sequences $a_{_{P_1, m_1}}$ and $a_{_{P_2, m_2}}$ are bounded and increasing, then 
\beqn
\mathrm{trace}(\det\, [\mathscr M^*_z, \mathscr M_z]) = \lim_{k \rar \infty} a_{_{P_1, m_1}}(k) \Big(\lim_{l \rar \infty} a_{_{P_2, m_2}}(l)\Big)^2.
\eeqn
\end{itemize} 
\end{proposition}
\begin{proof} 
Let $\alpha =(\alpha_1, \alpha_2) \in \mathbb Z^2_+.$ 
Set
\beqn
m^{(1)}_\alpha = \prod_{j=1}^2 \frac{\sqrt{A_{P_j, m_j}(\alpha_j \varepsilon_j)}}{\sqrt{A_{P_j, m_j}((\alpha_j +1)\varepsilon_j)}}, \, m^{(2)}_\alpha = \frac{\sqrt{A_{P_2, m_2}(\alpha_2 \varepsilon_2)}}{\sqrt{A_{P_2, m_2}((\alpha_2 +1)\varepsilon_2)}}. 
\eeqn
Since $P$ is admissible, by \eqref{A-tilde-A-new}, 
\beqn
A_{P, m}(\alpha)= A_{P_1, m_1}(\alpha_1 \varepsilon_1) A_{P_2, m_2}(\alpha_2 \varepsilon_2).
\eeqn
A routine calculation using \eqref{action-basis} and \eqref{action-adjoint} shows that 
\beqn
&& [\mathscr M^*_{z_j}, \mathscr M_{z_j}]e_{\alpha} =
\Big((m^{(j)}_{\alpha})^2-(m^{(j)}_{\alpha-\sum_{k=j}^2\varepsilon_k})^2\Big)e_{\alpha}, \quad j=1, 2,\\
  && [\mathscr M^*_{z_2}, \mathscr M_{z_1}]e_{\alpha}
= (m^{(1)}_{\alpha}m^{(2)}_{\alpha+\varepsilon_1}
- m^{(1)}_{\alpha-\varepsilon_2}m^{(2)}_{\alpha-\varepsilon_2})e_{\alpha + \varepsilon_1}, \\
&& [\mathscr M^*_{z_1}, \mathscr M_{z_2}]e_{\alpha}
= (m^{(1)}_{\alpha-\varepsilon_1}m^{(2)}_{\alpha}-m^{(2)}_{\alpha-\varepsilon_1-\varepsilon_2}m^{(1)}_{\alpha-\varepsilon_1-\varepsilon_2})e_{\alpha - \varepsilon_1}.
\eeqn
It is now easy to see that 
\beqn
&& [\mathscr M^*_{z_1}, \mathscr M_{z_1}][\mathscr M^*_{z_2}, \mathscr M_{z_2}]e_{\alpha} \\
&&=
	\Big((m^{(1)}_{\alpha})^2-(m^{(1)}_{\alpha-\varepsilon_1-\varepsilon_2})^2\Big)	\Big((m^{(2)}_{\alpha})^2-(m^{(2)}_{\alpha-\varepsilon_2})^2\Big)e_{\alpha}, \\
& & [\mathscr M^*_{z_1}, \mathscr M_{z_2}][\mathscr M^*_{z_2}, \mathscr M_{z_1}]e_{\alpha}=
\Big(m^{(1)}_{\alpha}m^{(2)}_{\alpha+\varepsilon_1}-m^{(1)}_{\alpha-\varepsilon_2}m^{(2)}_{\alpha-\varepsilon_2}\Big)^2e_{\alpha}, \\
&&	[\mathscr M^*_{z_2}, \mathscr M_{z_1}][\mathscr M^*_{z_1}, \mathscr M_{z_2}]e_{\alpha}=
\Big(m^{(1)}_{\alpha-\varepsilon_1}m^{(2)}_{\alpha}-m^{(2)}_{\alpha-\varepsilon_1-\varepsilon_2}m^{(1)}_{\alpha-\varepsilon_1-\varepsilon_2}\Big)^2e_{\alpha}.
	\eeqn
This combined with \eqref{det-def} shows that the determinant operator $\det\,[\mathscr M^*_z, \mathscr M_z]$ is a diagonal operator (with respect to the orthonormal basis $\{e_{\alpha}\}_{\alpha \in \mathbb Z^2_+}$) with diagonal entries given by 
\beq
\label{det-positive}
&& \Big \langle{\det\,[\mathscr M^*_z, \mathscr M_z] e_{\alpha}}, {e_{\alpha}}\Big\rangle \\
&=& (a_{_{P_1, m_1}}(\alpha_1) - a_{_{P_1, m_1}}(\alpha_1 - 1)) (a^2_{_{P_2, m_2}}(\alpha_2) - a^2_{_{P_2, m_2}}(\alpha_2 - 1)), ~ \alpha \in \mathbb Z^2_+. \notag
\eeq

(i) If $\det\,[\mathscr M^*_z, \mathscr M_z]$ is positive, then by \eqref{det-positive} (with $\alpha_1 =0$) and the fact that  $a_{_{P_1, m_1}}(0) - a_{_{P_1, m_1}}(-1) > 0,$ $a_{_{P_2, m_2}}$ is necessarily increasing. 
Similarly, by letting $\alpha_2=0$ in \eqref{det-positive}, we may conclude  that $a_{_{P_1, m_1}}$ is also increasing. 
Conversely, if $a_{_{P_1, m_1}}$ and $a_{_{P_2, m_2}}$ are increasing, then by
\eqref{det-positive},  $\det\,[\mathscr M^*_z, \mathscr M_z]$ is positive. 

(ii)$\&$(iii) Assume that $\det\, [\mathscr M^*_z, \mathscr M_z]$ is a positive operator. Note that $\det\,[\mathscr M^*_z, \mathscr M_z]$ is of trace-class if and only if 
\begin{align*}
 &\mathrm{trace}(\det\,[\mathscr M^*_z, \mathscr M_z])  \\
&= \! \! \sum_{\alpha \in \mathbb Z^2_+}\! \Big(a_{_{P_1, m_1}}(\alpha_1) - a_{_{P_1, m_1}}(\alpha_1 - 1)\Big) \Big(a^2_{_{P_2, m_2}}(\alpha_2) - a^2_{_{P_2, m_2}}(\alpha_2 - 1)\Big) < \infty.
\end{align*}
By (i), the series above consists of nonnegative terms, and hence 
\beq \label{trace-cgt}
  \mathrm{trace}(\det\,[\mathscr M^*_z, \mathscr M_z]) 
= \prod_{j=1}^2 \sum_{\alpha_j \in \mathbb Z_+}\Big(a^j_{_{P_j, m_j}}(\alpha_j) - a^j_{_{P_j, m_j}}(\alpha_j - 1)\Big).
%\sum_{\alpha_2 \in \mathbb Z_+} \Big(a^2_{_{P_2, m_2}}(\alpha_2) - a^2_{_{P_2, m_2}}(\alpha_2 - 1)\Big). 
\eeq 
Thus
$\det\,[\mathscr M^*_z, \mathscr M_z]$ is of trace-class if and only if the series
\begin{equation*}
 \sum_{\alpha_1 \in \mathbb Z_+}\Big(a_{_{P_1, m_1}}(\alpha_1) - a_{_{P_1, m_1}}(\alpha_1 - 1)\Big), ~ \sum_{\alpha_2 \in \mathbb Z_+} \Big(a^2_{_{P_2, m_2}}(\alpha_2) - a^2_{_{P_2, m_2}}(\alpha_2 - 1)\Big) 
\end{equation*}
are convergent. In view of (i), 
the latter one happens if and only if $a_{_{P_1, m_1}}$ and $a_{_{P_2, m_2}}$ are bounded, which gives (ii). Finally, (ii) combined with \eqref{trace-cgt} yields (iii).
\end{proof}

\subsection{$\triangle^{\!n}_{_P}$-contractions} \label{S5.3}

In this subsection, we discuss the notion of contractivity closely associated with the shape of a generalized Hartogs triangle. 
\begin{definition} 
\label{def-D-contraction}
{\it 
Let $m \in \mathbb N^n$ and let $P$ be a positive regular polynomial $n$-tuple such that $\frac{1}{\mathscr K_{_{P, m}}}$ is a hereditary polynomial, that is, 
$\frac{1}{\mathscr K_{_{P, m}}}$ is a polynomial in $z$ and $\overline{w}.$ 
%(see \cite[Chapter~4]{AMY2020}).  
A commuting $n$-tuple $T$ on $H$ is said to be a {\it $\triangle^{\!n}_{_P}$-contraction of order $m$} if 
$\frac{1}{\mathscr K_{_{P, m}}}(T, T^*) \Ge 0.$
We say that $T$ is a  {\it $\triangle^{\!n}_{_P}$-isometry of order $m$} if 
$\frac{1}{\mathscr K_{_{P, m}}}(T, T^*) = 0.$
If $P=P_0$ $($see \eqref{def-P-a}$),$ then we refer to  $\triangle^{\!n}_{_P}$-contraction $($resp. $\triangle^{\!n}_{_P}$-isometry$)$ of order $m$ as {\it $\triangle^{\!n}_{_0}$-contraction} $($resp. {\it $\triangle^{\!n}_{_0}$-isometry}$)$ of order $m.$ In case $m$ is the $n$-tuple ${\bf 1}$ with all entries equal to $1,$ we drop the term ``of order $m$".} 
\end{definition}
\begin{remark} \label{rmk-H-triangle}
For a commuting $n$-tuple $T=(T_1, \ldots, T_n)$ on $H,$ let
\beqn
D^{(1)}_n(T) \!\!\! \! &=& \!\!\! \! T^*_nT_n-T^*_{n-1}T_{n-1}, \\ 
%\quad D_{n-1}= T^*_{n-1}D_n T_{n-1}-T^*_{n-2}D_nT_{n-2} \\
D^{(k)}_{n}(T) \!\!\! \! &=& \!\!\! \! T^*_{n-k+1}D^{(k-1)}_{n}(T) T_{n-k+1}-T^*_{n-k}D^{(k-1)}_{n}(T)T_{n-k}, ~2 \Le  k \Le n-1.
\eeqn
It is easy to see that
a commuting $n$-tuple $(T_1, \ldots, T_n)$ is a $\triangle^{\!n}_{_0}$-contraction (resp. $\triangle^{\!n}_{_0}$-isometry) if and only if 
\beqn 
\mbox{$T^*_nD^{(n-1)}_n(T)T_n \Le D^{(n-1)}_n(T)$ (resp.
$T^*_nD^{(n-1)}_n(T)T_n = D^{(n-1)}_n(T)$).}
\eeqn 
In particular, a commuting pair $(T_1, T_2)$ is a $\triangle^{\!2}_{_0}$-contraction if and only if 
\beq \label{2-H-con}
%T^*_2D^{(1)}_2T_2 \Le D^{(1)}_2 \Longleftrightarrow 
T^*_2(T^*_2T_2-T^*_{1}T_{1})T_2 \Le T^*_2T_2-T^*_{1}T_{1}.
\eeq
Also, a commuting triple $T=(T_1, T_2, T_3)$ is a $\triangle^{\!3}_{_0}$-contraction if and only if  
\allowdisplaybreaks
\beq
\label{3-H-con}
%T^*_3D^{(2)}_3T_3 \Le D^{(2)}_3.
&& T^*_3\Big(T^*_2(T^*_3T_3-T^*_{2}T_{2})T_2 - T^*_1(T^*_3T_3-T^*_{2}T_{2})T_1\Big)T_3 \notag \\
& \Le & T^*_2(T^*_3T_3-T^*_{2}T_{2})T_2 - T^*_1(T^*_3T_3-T^*_{2}T_{2})T_1.
\eeq
Furthermore,  
$(T_1, T_2)$ is a $\triangle^{\!2}_{_0}$-isometry if and only if equality holds in \eqref{2-H-con}, and  $(T_1, T_2, T_3)$ is a $\triangle^{\!3}_{_0}$-isometry if and only if equality holds in \eqref{3-H-con}. It is easy to construct examples of $\triangle^{\!n}_{_0}$-contractions and $\triangle^{\!n}_{_0}$-isometries. For example, if $D^{(1)}_n(T)=0$ (this happens if $T^*_nT_n=T^*_{n-1}T_{n-1}$), then $D^{(n-1)}_n(T)=0,$ and hence $T$ is a $\triangle^{\!n}_{_0}$-isometry. 
%In particular, if $\phi \in \mathscr H^{\infty}(\mathbb D)$ and $\psi$ is an invertible element $\mathscr H^{\infty}(\mathbb D)$ such that $\frac{\phi}{\psi}$ is an inner function, then by
%\cite[Proposition~V.2.2]{SFBK2010}, 
%$\mathscr M_\phi^* \mathscr M_\phi=\mathscr M^*_\psi \mathscr M_\psi,$ and hence 
%%$D^{(1)}_3(T_1, T_2, T_3)=0,$  
%$(\mathscr M_z, \mathscr M_\phi, \mathscr M_\psi)$ is a $\triangle^{\!3}_{_0}$-isometry,
%where $\mathscr M_{f} \in \mathcal B(\mathscr H^2(\mathbb D))$ denotes the operator of multiplication by $f \in \mathscr H^{\infty}(\mathbb D).$ 
\hfill 
$\diamondsuit$
\end{remark}

The $n$-tuple $\mathscr M^*_z$ on $\mathscr H^2_m(\triangle^{\!n}_{_P})$ is a $\triangle^{\!n}_{_P}$-contraction provided the reciprocal of the reproducing kernel $\mathscr K_{_{P, m}}$ is a hereditary polynomial.  
%that is, $\frac{1}{\mathscr K_{_{P, m}}}$ is a polynomial in $z$ and $\overline{w}$ (see \cite[Chapter~4]{AMY2020}). 
\begin{proposition} \label{positivity-prop}
Assume that  $\frac{1}{\mathscr K_{_{P, m}}}$ is a hereditary polynomial.
%If $\sigma(\mathscr M^*_z) \subseteq   \overline{\triangle^{\!n}_{_P}} \subset \Omega,$ then  
Then $\mathscr M^*_z$ on $\mathscr H^2_m(\triangle^{\!n}_{_P})$ is a $\triangle^{\!n}_{_P}$-contraction of order $m.$
%\beq 
%\label{positivity}
%\frac{1}{\mathscr K_{_{P, m}}}(\mathscr M^*_z, \mathscr M_z) \Ge 0.
%\eeq
\end{proposition}
\begin{proof} 
Since $\frac{1}{\mathscr K_{_{P, m}}}$ is a hereditary polynomial,  for a finite subset $F$ of $\mathbb Z^n_+$ and some complex numbers $a_{\alpha},$
\beqn
\frac{1}{\mathscr K_{_{P, m}}(z, w)} \overset{\eqref{exp-2}}= \sum_{\alpha \in F} a_{\alpha} \overline{w}^\alpha z^{\alpha}.
\eeqn
Applying the hereditary functional calculus to $\mathscr M^*_z,$ we obtain
\beqn
\frac{1}{\mathscr K_{_{P, m}}}(\mathscr M^*_z, \mathscr M_z)= \sum_{\alpha \in F} a_{\alpha} \mathscr M^{\alpha}_z \mathscr M^{*\alpha}_z.
\eeqn 
It now follows from \eqref{rp}  and \eqref{eigen-v}  that for any $x, y \in \triangle^{\!n}_{_P},$
\allowdisplaybreaks
\beq
\label{here-calculus}
&& \Big \langle \frac{1}{\mathscr K_{_{P, m}}}(\mathscr M^*_z, \mathscr M_z)\mathscr K_{_{P, m}}(\cdot, x), \, \mathscr K_{_{P, m}}(\cdot, y) \Big \rangle \notag \\ &=& \sum_{\alpha \in F} a_{\alpha} \inp{\mathscr M^{*\alpha}_z \mathscr K_{_{P, m}}(\cdot, x)}{\mathscr M^{*\alpha}_z \mathscr K_{_{P, m}}(\cdot, y)} \notag \\
%&=& \sum_{\alpha, \beta \in \mathbb Z^n_+} a_{\alpha, \beta} \overline{x}^{\alpha} y^{\beta} \mathscr K_{P, m}(y, x)
&=& 1.
\eeq
For a positive integer $N,$ $c_1, \ldots, c_N \in \mathbb C$ and $w^{(1)}, \ldots, w^{(N)} \in \triangle^{\!n}_{_P},$ let $f=\sum_{j=1}^N c_j \mathscr K_{_{P, m}}(\cdot, w^{(j)}).$ By \eqref{here-calculus},
\allowdisplaybreaks
\beqn
&& \Big\langle \frac{1}{\mathscr K_{_{P, m}}}(\mathscr M^*_z, \mathscr M_z)f, f \Big \rangle \\ &=& \sum_{j, k =1}^N c_j \overline{c}_k \Big\langle \frac{1}{\mathscr K_{_{P, m}}}(\mathscr M^*_z, \mathscr M_z)\mathscr K_{_{P, m}}(\cdot, w^{(j)}), \mathscr K_{_{P, m}}(\cdot, w^{(k)}) \Big \rangle \\
&=& \Big|\sum_{j=1}^N c_j \Big|^2.
%&=&  \sum_{j, k =1}^N c_j \overline{c}_k \frac{1}{\kappa_P}(w^{(k)}, w^{(j)}) \kappa_P(w^{(k)}, w^{(j)})
\eeqn
Since any arbitrary member of  $\mathscr H^2_m(\triangle^{\!n}_{_P})$ can be approximated in $\mathscr H^2_m(\triangle^{\!n}_{_P})$ by elements of the form $f,$ we get the desired positivity condition. 
\end{proof}

In view of \eqref{exp-hered} (see Example~\ref{exm-triangle-2}), the following is immediate from Proposition~\ref{positivity-prop}. 
\begin{corollary} \label{Coro-H-contraction}
If $a$ is a nonnegative real number, then 
$\mathscr M^*_z$ on $\mathscr H^2(\triangle^{\!n}_{_a})$ is a $\triangle^{\!n}_{_a}$-contraction.
\end{corollary}

\section{Commutant, multiplier algebra and point spectrum} \label{S6}

The {\it commutant} $\mathscr S'$ of a subset $\mathscr S$ of $\mathcal B(H)$ is given by
\beqn
\mathscr S' := \{T \in \mathcal B(H) : ST=TS~\mbox{for all~}S \in \mathscr S\}.
\eeqn 
For a commuting $n$-tuple $T=(T_1, \ldots, T_n)$ on $H$, we use the simpler notation $\{T\}'$ for $\mathscr S',$ where $\mathscr S=\{T_1, \ldots, T_n\}.$ 
Note that $\mathscr S'$ is a unital closed subalgebra of $\mathcal B(H).$

The following theorem describes the commutant of the multiplication $n$-tuple $\mathscr M_z$ on $\mathscr H^2_m(\triangle^{\!n}_{_P}).$
\begin{theorem} \label{mult-h-infty} If $\mathscr M_z$ is the multiplication $n$-tuple on $\mathscr H^2_m(\triangle^{\!n}_{_P}),$ then the following statements are valid$:$
\begin{itemize}
\item[$\mathrm{(i)}$] $\{\mathscr M_z\}'$ is a maximal abelian subalgebra of $\mathcal B(\mathscr H^2_m(\triangle^{\!n}_{_P})),$
\item[$\mathrm{(ii)}$] $\{\mathscr M_z\}'$ is equal to the multiplier algebra of $\mathscr H^{2}_m(\triangle^{\!n}_{_P}),$
\item[$\mathrm{(iii)}$]  $\{\mathscr M_z\}'$ is contractively contained in
$\mathscr H^{\infty}(\triangle^{\!n}_{_P}).$
\end{itemize}
\end{theorem}

We need the following lemma in the proof of Theorem~\ref{mult-h-infty}. 

\begin{lemma} \label{spectral-lem}
Let $\mathscr M_z$ be the multiplication $n$-tuple on $\mathscr H^2_m(\triangle^{\!n}_{_P}).$ Then 
\beqn
\ker(\mathscr M^*_z- \overline{w}) = \begin{cases} \{\eta \mathscr K_{_{P, m}}(\cdot, w) : \eta \in \mathbb C\} & \mbox{if}~w \in \triangle^{\!n}_{_P}, \\
\bigvee \{e_{\alpha} : \alpha \in \mathbb Z^n_+, \, \alpha_n =0\} & \mbox{if}~w = 0.
\end{cases}
\eeqn
Moreover, if $w \in \mathbb C \times \mathbb C^{n-1}_*,$ then   
$f \in \ker(\mathscr M^*_z- \overline{w})$ takes the form
\beq \label{f-must-be-k} 
f(z) = \inp{f}{e_0} \prod_{k=2}^n\overline{w}_k \sum_{\alpha \in \mathbb Z^n_+} \overline{e_{\alpha}(w)}e_{\alpha}(z), \quad z \in \triangle^{\!n}_{_P},
\eeq
where $e_\alpha,$ as defined in \eqref{e-alpha}, extends with the same definition to $\mathbb C \times \mathbb C^{n-1}_*.$
\end{lemma}
\begin{proof}
Let $f$ belong to $\ker(\mathscr M^*_z - \overline{w})$ for some $w \in \mathbb C \times \mathbb C^{n-1}_*.$ 
By Proposition~\ref{on-basis}, $f=\sum_{\alpha \in \mathbb Z^n_+} d_{\alpha}e_{\alpha}$ for some complex numbers $d_\alpha.$ 
Set 
\beqn \varepsilon^{(j)} = \sum_{k=j}^n \varepsilon_k, 
%~ \alpha \cdot \varepsilon^{(j)} = \sum_{k=j}^n \alpha_k\varepsilon_k, 
\quad j=1, \ldots, n. 
%~\alpha \in \mathbb Z^n_+.
\eeqn
Note that for $j=1, \ldots, n,$ 
\beqn
\mathscr M^*_{z_j}f &\overset{\eqref{action-adjoint}}=& 
\sum_{\substack{\alpha \in \mathbb Z^n_+ \\ \alpha \Ge  \varepsilon^{(j)}}} 
d_\alpha \, \frac{\sqrt{A_{P, m}(\alpha - \varepsilon^{(j)})}}{\sqrt{A_{P, m}(\alpha)}}\,e_{\alpha -  \varepsilon^{(j)}} \\
&=& \sum_{\alpha \in \mathbb Z^n_+} d_{\alpha + \varepsilon^{(j)}}\,
\frac{\sqrt{A_{P, m}(\alpha)}}{\sqrt{A_{P, m}(\alpha + \varepsilon^{(j)})}}\,e_{\alpha}.
% \sum_{i, j \in \mathbb Z_+} a_{(i, j)} e_{(i, j)}
\eeqn
Since $\mathscr M^*_{z_j}f = \overline{w}_j f,$ by comparing coefficients of $e_\alpha$ on both sides, we get 
\beq \label{letting-j}
d_{\alpha + \varepsilon^{(j)}}
\frac{\sqrt{A_{P, m}(\alpha)}}{\sqrt{A_{P, m}(\alpha + \varepsilon^{(j)})}} = \overline{w}_j d_\alpha, \quad \alpha \in \mathbb Z^n_+,~ j=1, \ldots, n,
\eeq
and therefore, after replacing $\alpha$ by $\alpha -\alpha \diamond \varepsilon^{(j)},$ we obtain 
\beq \label{4.6-new}
d_{\alpha -\alpha \diamond \varepsilon^{(j)}  + \varepsilon^{(j)}}
 = \overline{w}_j \frac{\sqrt{A_{P, m}(\alpha -\alpha \diamond \varepsilon^{(j)}  + \varepsilon^{(j)})}}{\sqrt{A_{P, m}(\alpha-\alpha \diamond \varepsilon^{(j)} )}}\,d_{\alpha -\alpha \diamond \varepsilon^{(j)}}, \quad \alpha \in \mathbb Z^n_+. 
\eeq
Fix $\alpha \in \mathbb Z^n_+$ and note that    
\beqn
&&  d_{\alpha -\alpha \diamond \varepsilon^{(j)}  + 2\varepsilon^{(j)}} \\
 &\overset{\eqref{letting-j}}=&  \overline{w}_j \frac{\sqrt{A_{P, m}(\alpha -\alpha \diamond \varepsilon^{(j)}  + 2\varepsilon^{(j)})}}{\sqrt{A_{P, m}(\alpha-\alpha \diamond \varepsilon^{(j)} + \varepsilon^{(j)})}}\,d_{\alpha -\alpha \diamond \varepsilon^{(j)} + \varepsilon^{(j)}}
 \\
 &\overset{\eqref{4.6-new}}=& \overline{w}^2_j \frac{\sqrt{A_{P, m}(\alpha -\alpha \diamond \varepsilon^{(j)}  + 2\varepsilon^{(j)})}}{\sqrt{A_{P, m}(\alpha-\alpha \diamond \varepsilon^{(j)} )}}\,d_{\alpha -\alpha \diamond  \varepsilon^{(j)}}, 
\eeqn
and hence by a finite induction, we get
\beq \label{ind-a-alpha}
& d_{\alpha -\alpha \diamond \varepsilon^{(j)} + \ell \varepsilon^{(j)}}
 = \displaystyle \overline{w}^{\ell}_j \frac{\sqrt{A_{P, m}(\alpha -\alpha \diamond \varepsilon^{(j)}  + \ell \varepsilon^{(j)})}}{\sqrt{A_{P, m}(\alpha -\alpha \diamond \varepsilon^{(j)})}} d_{\alpha -\alpha \diamond \varepsilon^{(j)}}, \notag \\ & \quad  j=1, \ldots, n, ~\ell \in \mathbb Z_+.
\eeq
Letting $(j, \ell)=(n-1, \alpha_{n-1}), (n, \alpha_{n}), (n, \alpha_{n-1})$ in \eqref{ind-a-alpha}, we obtain
\beq \label{middle-eqn} \notag
\displaystyle d_{\alpha - \alpha_n \varepsilon_n + \alpha_{n-1} \varepsilon_n}
 &=& \displaystyle \overline{w}^{\alpha_{n-1}}_{n-1} \frac{\sqrt{A_{P, m}(\alpha - \alpha_n \varepsilon_n + \alpha_{n-1} \varepsilon_n)}}{\sqrt{A_{P, m}(\alpha -\alpha \diamond \varepsilon^{(n-1)})}}\, d_{\alpha -\alpha \diamond \varepsilon^{(n-1)}}, \\
 d_{\alpha} &=& \overline{w}^{\alpha_n}_n \frac{\sqrt{A_{P, m}(\alpha)}}{\sqrt{A_{P, m}(\alpha - \alpha_n \varepsilon_n)}} d_{\alpha - \alpha_n \varepsilon_n}, 
\\ \notag 
 d_{\alpha - \alpha_n \varepsilon_n + \alpha_{n-1} \varepsilon_n} &=& \overline{w}^{\alpha_{n-1}}_n \frac{\sqrt{A_{P, m}(\alpha - \alpha_n \varepsilon_n + \alpha_{n-1} \varepsilon_n)}}{\sqrt{A_{P, m}(\alpha - \alpha_n \varepsilon_n)}} d_{\alpha - \alpha_n \varepsilon_n}, \quad  \alpha \in \mathbb Z^n_+. 
\eeq
After equating the first and the third identity, and solving it for $d_{\alpha - \alpha_n \varepsilon_n},$  for any $\alpha \in \mathbb Z^n_+,$ we get  
\beqn
\displaystyle d_{\alpha - \alpha_n \varepsilon_n}=\Big(\frac{\overline{w}_{n-1}}{\overline{w}_n}\Big)^{\alpha_{n-1}} \frac{\sqrt{A_{P, m}(\alpha - \alpha_n \varepsilon_n)}}{\sqrt{A_{P, m}(\alpha -\alpha \diamond \varepsilon^{(n-1)})}} \,  \displaystyle d_{\alpha -\alpha \diamond \varepsilon^{(n-1)}}.
\eeqn
This combined with \eqref{middle-eqn} gives 
\beq \label{upperwale ka naam}
d_{\alpha}=\overline{w}^{\alpha_n}_n \Big(\frac{\overline{w}_{n-1}}{\overline{w}_n}\Big)^{\!\alpha_{n-1}} \!\! \frac{\sqrt{A_{P, m}(\alpha)}}{\sqrt{A_{P, m}(\alpha -\alpha \diamond \varepsilon^{(n-1)})}} \displaystyle d_{\alpha -\alpha \diamond \varepsilon^{(n-1)}}. 
\eeq
Letting $(j, \ell)=(n-2, \alpha_{n-2}), (n-1, \alpha_{n-2})$ in \eqref{ind-a-alpha}, and arguing similarly, we obtain  
\beqn
 && d_{\alpha -\alpha \diamond \varepsilon^{(n-1)} + \alpha_{n-2} \varepsilon^{(n-1)}} \\
 &=& \displaystyle \overline{w}^{\alpha_{n-2}}_{n-2} \frac{\sqrt{A_{P, m}(\alpha -\alpha \diamond \varepsilon^{(n-1)} + \alpha_{n-2} \varepsilon^{(n-1)})}}{\sqrt{A_{P, m}(\alpha -\alpha \diamond \varepsilon^{(n-2)})}} d_{\alpha -\alpha \diamond \varepsilon^{(n-2)}},   \\
&&  d_{\alpha -\alpha \diamond \varepsilon^{(n-1)}  + \alpha_{n-2} \varepsilon^{(n-1)}} \\
 &=& \displaystyle \overline{w}^{\alpha_{n-2}}_{n-1} \frac{\sqrt{A_{P, m}(\alpha -\alpha \diamond \varepsilon^{(n-1)}  + \alpha_{n-2} \varepsilon^{(n-1)})}}{\sqrt{A_{P, m}(\alpha -\alpha \diamond \varepsilon^{(n-1)})}} d_{\alpha -\alpha \diamond \varepsilon^{(n-1)}}. 
\eeqn
Equating these identities, we get
\beqn
d_{\alpha -\alpha \diamond \varepsilon^{(n-1)}} 
= 
\displaystyle 
\Big(\frac{\overline{w}_{n-2}}{\overline{w}_{n-1}}\Big)^{\alpha_{n-2}} \frac{\sqrt{A_{P, m}(\alpha -\alpha \diamond \varepsilon^{(n-1)})}}{\sqrt{A_{P, m}(\alpha -\alpha \diamond \varepsilon^{(n-2)})}} d_{\alpha -\alpha \diamond \varepsilon^{(n-2)}}.
\eeqn
Combining this with \eqref{upperwale ka naam} and continuing along similar lines yields 
\allowdisplaybreaks
\beqn
 d_{\alpha} && = \prod_{k=n-2}^{n-1}\Big(\frac{\overline{w}_{k}}{\overline{w}_{k+1}}\Big)^{\alpha_{k}}\overline{w}^{\alpha_n}_n  \frac{\sqrt{A_{P, m}(\alpha)}}{\sqrt{A_{P, m}(\alpha - \alpha \diamond \varepsilon^{(n-2)})}} d_{\alpha - \alpha \diamond \varepsilon^{(n-2)}} \\
& & \vdots  \\
&&  = \prod_{k=1}^{n-1}\Big(\frac{\overline{w}_{k}}{\overline{w}_{k+1}}\Big)^{\alpha_{k}}\overline{w}^{\alpha_n}_n  \frac{\sqrt{A_{P, m}(\alpha)}}{\sqrt{A_{P, m}(\alpha - \alpha \diamond \varepsilon^{(1)})}} d_{\alpha - \alpha \diamond \varepsilon^{(1)}} \\ &&  = \prod_{k=1}^{n-1}\Big(\frac{\overline{w}_{k}}{\overline{w}_{k+1}}\Big)^{\alpha_{k}}\overline{w}^{\alpha_n}_n \sqrt{A_{P, m}(\alpha)} \, d_{{0}}, \quad \alpha \in \mathbb Z^n_+,
\eeqn
where we used the fact that $A_{P, m}(0)=1$ (see \eqref{rho-alpha-f-new}). 
Thus $f$ takes the form
\beqn 
%\label{f-must-be-k} \notag
f(z) &=& d_{{0}} \sum_{\alpha \in \mathbb Z^n_+} \prod_{k=1}^{n-1}\Big(\frac{\overline{w}_{k}}{\overline{w}_{k+1}}\Big)^{\alpha_{k}}\overline{w}^{\alpha_n}_n \sqrt{A_{P, m}(\alpha)} \,  e_{\alpha}(z) \\
&\overset{\eqref{e-alpha}}=& \inp{f}{e_0} \prod_{k=2}^n\overline{w}_k \sum_{\alpha \in \mathbb Z^n_+} \overline{e_{\alpha}(w)}e_{\alpha}(z).
\eeqn
This yields \eqref{f-must-be-k}. 

If $w \in \triangle^{\!n}_{_P},$ then  
by \eqref{kernel-basis-exp} and \eqref{f-must-be-k}, 
$\triangle^{\!n}_{_P} \subseteq \sigma_p(\mathscr M^*_z)$ and 
$f$ is a scalar multiple of $\mathscr K_{_{P, m}}(\cdot, w).$ This together with \eqref{eigen-v} and \eqref{0-pt-s} completes the proof of the lemma.
\end{proof}

\begin{proof}[Proof of Theorem~\ref{mult-h-infty}]
(i) Let $\mathscr W$ be the unilateral weighted multishift associated with $\mathscr M_z$ (see \eqref{wt-shift-action}). By Remark~\ref{noeigen},  $\mathscr M_{z_1}$ is injective, and hence by \eqref{shift-multi}, $\mathscr W_1, \ldots, \mathscr W_n$ are injective. Hence, by \cite[Corollary~12]{JL1979}, the commutant $\{\mathscr W\}'$ of $\mathscr W$ is a maximal abelian subalgebra of $\mathcal B(\mathscr H^2_m(\triangle^{\!n}_{_P})).$ Thus, it suffices to check that $\{\mathscr M_z\}'=\{\mathscr W\}'.$ Clearly, by
\eqref{shift-multi}, $\{\mathscr W\}'$  is a subset of  $\{\mathscr M_z\}'.$ To see the reverse inclusion, let $A \in \mathcal B(\mathscr H^2_m(\triangle^{\!n}_{_P}))$ be such that 
\beq \label{prop-commutant}
A \mathscr M_{z_j} = \mathscr M_{z_j}A, \quad j=1, \ldots, n. 
\eeq
Since $\mathscr M_{z_n}=\mathscr W_n$ (see \eqref{shift-multi}), $A$ commutes with $\mathscr W_n.$ This combined with \eqref{prop-commutant} (with $j=n-1$) and \eqref{shift-multi} yields  
\beqn
\mathscr W_n A \mathscr W_{n-1}  &=&  A \mathscr W_n \mathscr W_{n-1} = A \mathscr W_{n-1} \mathscr W_{n} \\
&=& A \mathscr M_{z_{n-1}}=\mathscr M_{z_{n-1}}  A \\
&=& \mathscr W_n \mathscr W_{n-1}A.
%A \mathscr W_{n-1} \mathscr W_n = \mathscr W_{n-1} \mathscr W_n A = 
\eeqn
Since $\mathscr W_n$ is injective, $A$ commutes with $\mathscr W_{n-1}.$ One may now proceed by a finite induction to conclude that $A$ belongs to the commutant of $\mathscr W.$ 

(ii) Clearly, the multiplier algebra of $\mathscr H^{2}_m(\triangle^{\!n}_{_P})$ is a subset of $\{\mathscr M_z\}'.$ To see the reverse inclusion, 
let $A \in \mathcal B(\mathscr H^2_m(\triangle^{\!n}_{_P}))$ be such that $A\mathscr M_{z_j} = \mathscr M_{z_j} A,$ $j=1, \ldots, n.$ 
Since $A^*\mathscr M^*_{z_j} = \mathscr M^*_{z_j} A^*,$ $j=1, \ldots, n,$
\beqn
\mathscr M^*_{z_j} A^*\mathscr K_{_{P, m}}(\cdot, w) = \overline{w}_j A^* \mathscr K_{_{P, m}}(\cdot, w), \quad w \in \triangle^{\!n}_{_P}.
\eeqn
Thus $A^*\mathscr K_{_{P, m}}(\cdot, w)$ belongs to $\ker \mathscr M^*_z,$ and hence, by Lemma~\ref{spectral-lem}, there exists $\eta(w) \in \mathbb C$ such that
\beqn
A^*\mathscr K_{_{P, m}}(\cdot, w) = \eta(w) \mathscr K_{_{P, m}}(\cdot, w), \quad w \in \triangle^{\!n}_{_P}.
\eeqn
Since $A^*\mathscr K_{_{P, m}}(\cdot, w)$ and $\frac{1}{\mathscr K_{_{P, m}}(\cdot, w)}$ are anti-holomorphic in $w,$ so is $\eta.$ Also, it is easy to see using \eqref{rp} that for any $f \in \mathscr H^{2}_m(\triangle^{\!n}_{_P}),$ $Af=\overline{\eta}f.$ Since $A$ is a bounded linear operator on $\mathscr H^{2}_m(\triangle^{\!n}_{_P}),$ $\overline{\eta}$ is a multiplier of $\mathscr H^{2}_m(\triangle^{\!n}_{_P}).$ 
This shows that $\{\mathscr M_{z}\}'$ is equal to the multiplier algebra of $\mathscr H^{2}_m(\triangle^{\!n}_{_P}).$ 

(iii) This is obtained by combining (ii) with \cite[Corollary~5.22]{PR2016}. 
\end{proof}

The multiplication $n$-tuple $\mathscr M_z$ on  $\mathscr H^2_m(\triangle^{\!n}_{_P})$ is always irreducible. 
\begin{corollary} \label{irrducible}
%Assume that $P$ is an admissible polynomial $n$-tuple. 
The multiplication $n$-tuple $\mathscr M_z$ on $\mathscr H^2_m(\triangle^{\!n}_{_P})$ does not have a proper joint reducing subspace. 
\end{corollary}
\begin{proof} If $\mathscr M_z$ has a joint reducing subspace $\mathscr N,$ then the orthogonal projection $P$ of $\mathscr H^2_m(\triangle^{\!n}_{_P})$ onto $\mathscr N$ belongs to the commutant of $\mathscr M_z.$ By Theorem~\ref{mult-h-infty}(ii),  $P=\mathscr M_\phi$ for some $\phi \in \mathscr H^\infty(\triangle^{\!n}_{_P}).$ Since $P^2=P,$ 
by \eqref{rp}, $(\phi(w)-1)\phi(w)=0$ for every $w \in \triangle^{\!n}_{_P}.$ Since $\triangle^{\!n}_{_P}$ is connected (see Proposition~\ref{DoH}), $\phi=0$ or $\phi=1,$ or equivalently, $P=0$ or $P=I.$
\end{proof}

As a by-product of the proof of Theorem~\ref{mult-h-infty}, we describe a part of the point spectrum of the adjoint of the multiplication $n$-tuple $\mathscr M_z$ on $\mathscr H^2_m(\triangle^{\!n}_{_P}).$
\begin{theorem} 
\label{theorem-pt-spectrum}
If $\mathscr M_z$ is the multiplication $n$-tuple on $\mathscr H^2_m(\triangle^{\!n}_{_P}),$ then 
$$\sigma_p(\mathscr M^*_z) \cap (\mathbb C \times \mathbb C^{n-1}_*) =\triangle^{\!n}_{_P}.$$
\end{theorem}
\begin{proof} 
By Lemma~\ref{spectral-lem}, $\triangle^{\!n}_{_P} \subseteq \sigma_p(\mathscr M^*_z) \cap (\mathbb C \times \mathbb C^{n-1}_*).$ To see the reverse inclusion, let $w \in \mathbb C \times \mathbb C^{n-1}_*.$
%Note that $e_\alpha(z),$ as defined in \eqref{e-alpha}, extends with the same definition to any $z \in \mathbb C \times \mathbb C^{n-1}_*.$
Let $f \in \ker(\mathscr M^*_z-\overline{w})$ be nonzero. By Proposition~\ref{on-basis} and  
\eqref{f-must-be-k}, 
\beqn
\inp{f}{e_0} \neq 0~\mbox{and}~ 
\sum_{\alpha \in \mathbb Z^n_+}|e_\alpha(w)|^2 < \infty.
\eeqn
This together with \eqref{e-alpha} shows that 
$\ker(\mathscr M^*_z- \overline{w})$ is not orthogonal to $\{e_0\}$ and 
\beqn 
%\label{cgt-general}
\sum_{\alpha \in \mathbb Z^n_+} A_{P, m}(\alpha) |\varphi(w)^{\alpha}|^2 
%= \sum_{\alpha \in \mathbb Z^n_+} A_{P, m}(\alpha) \prod_{j=1}^{n-1}\Big(\frac{|w_j|}{|w_{j+1}|}\Big)^{2\alpha_j} |w|^{2\alpha_n}_n 
< \infty.
\eeqn
This combined with \eqref{coeff-Pj-mj} and \eqref{exp} yields  $\sigma_p(\mathscr M^*_z) \cap (\mathbb C \times \mathbb C^{n-1}_*) \subseteq \triangle^{\!n}_{_P}.$ 
\end{proof}

One can completely describe the point spectrum of the adjoint of the multiplication $2$-tuple $\mathscr M_z$ on $\mathscr H^2_m(\triangle^{\!2}_{_P}).$
\begin{corollary} \label{pt-spec-in-2}
If $\mathscr M_z$ is the multiplication $2$-tuple on $\mathscr H^2_m(\triangle^{\!2}_{_P}),$ then $$\sigma_p(\mathscr M^*_z) = \triangle^{\!2}_{_P} \cup \{0\}.$$
\end{corollary}
\begin{proof}
Apply Remark~\ref{after-action-adjoint}, Lemma~\ref{spectral-lem} and Theorem~\ref{theorem-pt-spectrum}. 
\end{proof}

\section{Subnormality, Hardy spaces and von Neumann's inequality} \label{S7}

In this section, we discuss the notion of subnormality in the context of the multiplication $n$-tuples $\mathscr M_z$ on $\mathscr H^2_m(\triangle^{\!n}_{_P})$ and discuss its 
role in the definition of a Hardy space on the Hartogs triangle (cf. \cites{GGLV2021, M2021}). 
%This method extends naturally to the $n$-dimensional Hartogs triangle. 
Before we present a characterization of jointly subnormal $n$-tuples of the multiplication by the coordinate functions on $\mathscr H^2_m(\triangle^{\!n}_{_P}),$ let us recall some definitions. A multisequence $\{a_\beta\}_{\beta \in \mathbb Z^n_+}$ of positive real numbers is {\it Hausdorff moment} if there exists a finite positive Borel regular measure $\mu$ concentrated in some cube $C$ in the positive orthant of the
Euclidean space $\mathbb R^n$ such that 
\beqn
a_\beta  = \int_{C} t^\beta d\mu(t), \quad \beta \in \mathbb Z^n_+.
\eeqn
Note that for any positive constant $c,$ $\{c^{|\beta|}\}_{\beta \in \mathbb Z^n_+}$ is a Hausdorff moment multisequence. Indeed, 
\beq \label{c-eq}
c^{|\beta|} = \int_{[0, ~c]^n} t^{\beta} d\delta_{\bf c}, \quad \beta \in \mathbb Z^n_+, 
\eeq
where $\delta_{\bf c}$ denotes the Dirac delta measure with point mass at the $n$-tuple ${\bf c} = (c, \ldots, c).$
Recall from \cite[Lemma~8.2.1(v)]{BCR1984} that
\beq \label{product-moment}
&& \mbox{\it the product of Hausdorff moment multisequences is again a Hausdorff} \notag \\ && \mbox{\it moment multisequence.} 
\eeq

A commuting $n$-tuple $S=(S_1, \ldots, S_n)$ on $H$ is said to be {\it jointly subnormal} if there exist
a Hilbert space $K$ containing $H$ and a commuting $n$-tuple $
N$ on $K$ consisting of normal operators $N_1, \ldots, N_n$ such that 
\begin{eqnarray*}
N_jh = S_jh~\mbox{for~ every~}h \in H~ \mbox{and}~1\Le j \Le n.
\end{eqnarray*}
We say that $S$ is {\it separately subnormal} if $S_1, \ldots, S_n$ are subnormal operators. 
A commuting $n$-tuple $S=(S_1, \ldots, S_n)$ on $H$ is said to be {\it separately hyponormal} if $[S^*_j, S_j]$ is a positive operator for every $j=1, \ldots, n.$ It turns out that a jointly subnormal tuple is separately subnormal and a separately subnormal tuple is separately hyponormal (
the reader is referred to \cites{AP1990, C1988} for the basics of jointly subnormal tuples and related classes).   
\begin{remark}
%\label{commutator-hypo}
The multiplication $n$-tuple  $\mathscr M_z$ on $\mathscr H^2_m(\triangle^{\!n}_{_P})$ is separately hyponormal if and only if 
\begin{equation*}
A_{P, m}\Big(\alpha + \sum_{k=j}^n \varepsilon_k\Big)A_{P, m}\Big(\alpha - \sum_{k=j}^n \varepsilon_k\Big) \Le A_{P, m}(\alpha)^2, ~ \alpha \in \mathbb Z^n_+,~j=1, \ldots, n. 
\end{equation*}
Indeed, 
a simple calculation using \eqref{action-basis} and \eqref{action-adjoint} shows that $[\mathscr M^*_{z_j}, \mathscr M_{z_j}]$ is a diagonal operator with diagonal entries 
%$\inp{[\mathscr M^*_{z_j}, \mathscr M_{z_j}]e_\alpha}{e_\alpha},$ $\alpha \in \mathbb Z^n_+,$ 
given by 
\begin{equation*}
\inp{[\mathscr M^*_{z_j}, \mathscr M_{z_j}]e_\alpha}{e_\alpha} = \frac{A_{P, m}(\alpha)}{A_{P, m}\big(\alpha + \sum_{k=j}^n \varepsilon_k\big)} - \frac{A_{P, m}\big(\alpha - \sum_{k=j}^n \varepsilon_k\big)}{A_{P, m}(\alpha)}, ~\alpha \in \mathbb Z^n_+,
\end{equation*}
and hence the conclusion is immediate from \cite[Proposition~II.6.6]{Co1991}. \hfill $\diamondsuit$
\end{remark}

The following proposition characterizes jointly subnormal multiplication $n$-tuples $\mathscr M_z$ on $\mathscr H^2_m(\triangle^{\!n}_{_P}).$ 
\begin{proposition} \label{joint-s} 
%Assume that $\min_{j=1}^n \{\prod_{l=j}^n a_l\} \Ge 1.$ 
The multiplication $n$-tuple $\mathscr M_z$ on $\mathscr H^2_m(\triangle^{\!n}_{_P})$ is jointly subnormal if and only if for every $\gamma \in \mathbb Z^n_+,$ $\Big\{\frac{1}{A_{P, m}(\gamma + \sum_{j=1}^n \beta_j \sum_{k=j}^n \varepsilon_k)}\Big\}_{\beta \in \mathbb Z^n_+}$ is a Hausdorff moment multisequence.
\end{proposition}
\begin{proof}
%[Proof of Proposition~\ref{joint-s}]
%After multiplying $\mathscr M_z$ by a suitable constant, we may assume
%By the assumption and \eqref{norm-estimate}, $\|\mathscr M_{z_j}\| \Le 1$ for every $j=1, \ldots, n.$ 
Let $c=\max\{\|\mathscr M_{z_j}\| : j=1, \ldots, n\}$ and note that 
$c$ is a positive real number. Clearly, 
$\mathscr M_z$ is jointly subnormal if and only if $c^{-1}\mathscr M_z := (c^{-1}\mathscr M_{z_1}, \ldots, c^{-1}\mathscr M_{z_n})$ is jointly subnormal. Moreover, the $n$-tuple $c^{-1}\mathscr M_z$ consists of commuting contractions (that is, $\|c^{-1}\mathscr M_{z_j}\| \Le 1,$ $j=1, \ldots, n$).   
By the discussion following \cite[Proposition~0]{AP1990}, 
$\mathscr M_z$ is jointly subnormal if and only if for every $f \in \mathscr H^2_m(\triangle^{\!n}_{_P}),$
\beqn 
\{c^{-2|\beta|}\|\mathscr M^{\beta}_zf\|^2\}_{\beta \in \mathbb Z^n_+}~ \mbox{is a Hausdorff moment multisequence}.
\eeqn
This combined with \eqref{c-eq} (applied to $c^2$ and $c^{-2}$) and \eqref{product-moment} yields the fact$:$ 
\begin{align}
\label{equivalence-HMS}
& \mbox{ $\mathscr M_z$ is jointly subnormal if and only if for every $f \in \mathscr H^2_m(\triangle^{\!n}_{_P}),$} \notag \\
& \{\|\mathscr M^{\beta}_zf\|^2\}_{\beta \in \mathbb Z^n_+}~ \mbox{is a Hausdorff moment multisequence}.
\end{align}
For $f \in \mathscr H^2_m(\triangle^{\!n}_{_P}),$ write $f = \sum_{\alpha \in \mathbb Z^n_+} \inp{f}{e_\alpha} e_\alpha$ (see Proposition~\ref{on-basis}). By \eqref{powers-orthogonal},  for any $\beta \in \mathbb Z^n_+,$ the family $\{\mathscr M^{\beta}_z e_\alpha\}_{\alpha \in \mathbb Z^n_+}$ is orthogonal in $\mathscr H^2_m(\triangle^{\!n}_{_P}),$ and hence 
\beqn
\|\mathscr M^{\beta}_z f\|^2 &=& \sum_{\alpha \in \mathbb Z^n_+} |\inp{f}{e_\alpha}|^2 \|\mathscr M^{\beta}_z e_\alpha\|^2 \\ &\overset{\eqref{action-power-Mz}}=&  \sum_{\alpha \in \mathbb Z^n_+} |\inp{f}{e_\alpha}|^2   \frac{A_{P, m}(\alpha)}{A_{P, m}(\alpha + \sum_{j=1}^n \beta_j \sum_{k=j}^n \varepsilon_k)}. 
\eeqn
This combined with the fact that the set of Hausdorff moment multisequences forms a closed convex cone in $\mathbb R^{\mathbb Z^n_+}$ (see \cite[p.~130]{BCR1984}) shows  that $\{\|\mathscr M^{\beta}_zf\|^2\}_{\beta \in \mathbb Z^n_+}$ is a Hausdorff moment multisequence for every $f \in \mathscr H^2_m(\triangle^{\!n}_{_P})$ if and only if for every $\alpha \in \mathbb Z^n_+,$ $\{\|\mathscr M^{\beta}_ze_\alpha\|^2\}_{\beta \in \mathbb Z^n_+}$ is a Hausdorff moment multisequence.
This together with \eqref{equivalence-HMS} completes the proof.
\end{proof}

The criterion for the joint subnormality, as provided in Proposition~\ref{joint-s}, takes a simple form if the polynomial $n$-tuple $P$ in question is admissible. 
\begin{corollary} \label{coro-subnormal}
Assume that $P$ is an admissible polynomial $n$-tuple. Then the multiplication $n$-tuple $\mathscr M_z$ on $\mathscr H^2_m(\triangle^{\!n}_{_P})$ is jointly subnormal if and only if for every $\gamma \in \mathbb Z^n_+,$ $\Big\{\displaystyle \prod_{j=1}^n \frac{1}{A_{{P}_j, m_j}((\gamma_j  + \sum_{l=1}^j \beta_l)\varepsilon_j)}\Big\}_{\beta \in \mathbb Z^n_+}$ is a Hausdorff moment multisequence.
\end{corollary}
\begin{proof}
In view of \eqref{A-tilde-A-new}, this is immediate from Proposition~\ref{joint-s}. 
\end{proof}

We now apply Corollary \ref{coro-subnormal} to the case of $n$-dimensional Hartogs triangle. 
\begin{example}[Example~\ref{exm-triangle-2} continued $\cdots$] \label{exm-triangle-3} 
We claim that $\mathscr M_z$ on $\mathscr H^2_m(\triangle^{\!n}_{_{0}})$ is jointly subnormal.
%Let us compute the coefficient-function $A_{P_0, m}$ (see \eqref{exp}). 
Note first that by the identity
\beqn 
\frac{1}{(1-x)^k} = \sum_{l=0}^\infty \binom{l+k-1}{k-1} x^l, \quad x \in (0, 1), ~k \in \mathbb N,
\eeqn
we may conclude from \eqref{def-P-a} and \eqref{coeff-Pj-mj} that 
\beq 
\label{A-P0-m-alpha-j}
A_{P_{{j, 0}}, m_j}(k\varepsilon_j) = \binom{k + m_j -1}{m_j-1}, \quad k \in \mathbb Z_+, ~j=1, \ldots, n.
\eeq
 Since $P_0$ is admissible, by Corollary~\ref{coro-subnormal}, it suffices to check that for every $\alpha \in \mathbb Z^n_+,$
\beqn
\Big\{\displaystyle \prod_{j=1}^n \frac{1}{A_{{P}_{j, 0}, m_j}((\alpha_j  + \sum_{l=1}^j \beta_l)\varepsilon_j)}\Big\}_{\beta \in \mathbb Z^n_+}. 
\eeqn
is a Hausdorff moment multisequence. 
To see this, note that for $\alpha, \beta \in \mathbb Z^n_+,$
\beqn 
\displaystyle \prod_{j=1}^n \frac{1}{A_{{P}_{j, 0}, m_j}((\alpha_j  + \sum_{l=1}^j \beta_l)\varepsilon_j)} &\overset{\eqref{A-P0-m-alpha-j}}=& \displaystyle\prod_{j=1}^n \frac{1}{\binom{\alpha_j + \sum_{l=1}^j \beta_l + m_j -1}{m_j-1}} \\ 
&=& \prod_{j=1}^n \frac{(m_j-1)!(\alpha_j + \sum_{l=1}^j \beta_l)!}{ 
(\alpha_j + \sum_{l=1}^j \beta_l + m_j -1 )!} \\
&=& \prod_{j=1}^n \prod_{k=1}^{m_j-1} \frac{(m_j-1)!}{\alpha_j + k +  \sum_{l=1}^j \beta_l}. 
\eeqn
If $J = \{j \in \{1, \ldots, n\} : m_j > 1\},$ then for every $\alpha, \beta \in \mathbb Z^n_+,$
\beq
\label{moment-h}
\displaystyle \prod_{j=1}^n \frac{1}{A_{{P}_{j, 0}, m_j}((\alpha_j  + \sum_{l=1}^j \beta_l)\varepsilon_j)} &\overset{\eqref{A-P0-m-alpha-j}}  = \displaystyle \prod_{j \in J} \prod_{k=1}^{m_j-1} \frac{(m_j-1)!}{\alpha_j + k +  \sum_{l=1}^j \beta_l}.
\eeq 
%Without loss of generality, we may assume that $J$ is nonempty. 
On the other hand, for any positive real number $a$ and $j=1, \ldots, n,$
\beqn
&& \frac{1}{a+\sum_{k=1}^j \beta_k} \\ &=&  \int_{[0, 1]^{n}} t^{\beta} t^{a-1}_1 dt_1 d\delta_{t_1}(t_2) \cdots d\delta_{t_1}(t_j) d\delta_1(t_{j+1}) \cdots d\delta_1(t_n), 
\quad \beta \in \mathbb Z^n_+,
\eeqn
where $\delta_x$ denotes the Dirac delta measure with point mass at $x \in [0, 1].$
This combined with \eqref{product-moment} and \eqref{moment-h} shows that $\mathscr M_z$ on $\mathscr H^2_m(\triangle^{\!n}_{_{0}})$ is jointly subnormal. 
\eof
\end{example}

It is worth mentioning that \eqref{A-tilde-A-new} together with \eqref{A-P0-m-alpha-j} yields the formula$:$
\beq
\label{A-P0-m-alpha}
A_{P_0, m}(\alpha) = \prod_{j=1}^n \binom{\alpha_j + m_j -1}{m_j-1}, \quad \alpha \in \mathbb Z^n_+.
\eeq
In the next proposition, we need the following identity (see \cite[p.~4]{HKZ2000})$:$
\beq
\label{HKZ-id}
\int_{\mathbb D} |w|^{2l} (1- |w|^2)^{k} dw =
\frac{\pi}{(k+1)\binom{l+k+1}{k+1}}, \quad k, l \in \mathbb Z_+. 
\eeq
The following shows that the space $\mathscr H^2(\triangle^{\!n}_{_{0}})$ (see Example~\ref{exm-triangle-2}) provides a befitting candidate for the Hardy space of $\triangle^{\!n}_{_{0}}.$ 
\begin{proposition} 
\label{norm-Hardy}
%Let $P_0$ be as defined in \eqref{def-P-a}. Then the multiplication $n$-tuple $\mathscr M_z$ is jointly subnormal. Moreover, 
The following statements are valid$:$
\begin{itemize}
\item[$\mathrm{(i)}$] 
%if $m_j=1,$ $j=1, \ldots, n,$ then 
the norm of $f \in \mathscr H^2(\triangle^{\!n}_{_{0}})$ is given by
\begin{equation*}
%&& \|f\|^2 \\
\|f\|^2=  \sup_{\underset{t_j \in (0, 1)}{j=1, \ldots, n}} \int_{[0, 2\pi]^n}  \!\Big|f\Big(\prod_{j=1}^n t_j e^{i \theta_1}, \prod_{j=2}^n t_j e^{i \theta_2}, \ldots, t_n e^{i \theta_n}\Big)\Big|^2 \prod_{j=1}^n t^{2j-1}_j \frac{d\theta}{(2\pi)^n}, 
  %~  f \in \mathscr H^2(\triangle^{\!n}_{_{0}}), 
\end{equation*}
where $d\theta$ denotes the Lebesgue measure on $[0, 2\pi]^n,$
\item[$\mathrm{(ii)}$] if $m_j \Ge 2$ for every $j=1, \ldots, n,$ then $\mathscr H^2_m(\triangle^{\!n}_{_{0}})$ is embedded isometrically into $L^2(\triangle^{\!n}_{_{0}}, w(z)dz),$ 
%\beqn
%\|f\|^2 = \int_{\triangle^{\!n}_{_{0}}} |f(z)|^2 w(z)dz, \quad f \in \mathscr H^2_m(\triangle^{\!n}_{_{0}}),
%\eeqn
where the weight $w(z)$ is given by
\beqn
w(z) = \frac{1}{\pi^n}\Big(\prod_{j=1}^{n-1} (m_j-1)\Big(1- \frac{|z_j|^2}{|z_{j+1}|^2}\Big)^{m_j -2}\Big) (m_n-1)(1-|z_n|^2)^{m_n-2}, \\ z \in \triangle^{\!n}_{_{0}}.
\eeqn
\end{itemize}
\end{proposition} 
\begin{proof} 
(i) 
%Assume that $m_j=1,$ $j=1, \ldots, n.$ 
Note that 
for $\alpha \in \mathbb Z^n_+,$ 
\beqn
&&  \int_{[0, 2\pi]^n}  \Big|e_\alpha \Big(\prod_{j=1}^n t_j e^{i \theta_1}, \prod_{j=2}^n t_j e^{i \theta_2}, \ldots, t_n e^{i \theta_n}\Big)\Big|^2 \prod_{j=1}^n t^{2j-1}_j d\theta \\
&\overset{\eqref{e-alpha} \& \eqref{A-P0-m-alpha}}=&  \int_{[0, 2\pi]^n}   \Big|\varphi \Big(\prod_{j=1}^n t_j e^{i \theta_1}, \prod_{j=2}^n t_j e^{i \theta_2}, \ldots, t_n e^{i \theta_n}\Big)^{\alpha}\Big|^2 \prod_{j=1}^n t_j d\theta \\
&\overset{\eqref{tilde}}=&   \int_{[0, 2\pi]^n}    \prod_{j=1}^n t^{2\alpha_j+1}_j d\theta \\
&=& (2\pi)^n \prod_{j=1}^n t^{2\alpha_j+1}_j.
\eeqn
Now taking supremum over $t_j \in (0, 1),$ $j=1, \ldots, n,$ we get the desired formula in (i) for $f=e_\alpha,$ $\alpha \in \mathbb Z^n_+.$ Since $\{e_\alpha\}_{\alpha \in \mathbb Z^n_+}$ is an orthonormal basis for $\mathscr H^2(\triangle^{\!n}_{_{0}})$ (see Proposition~\ref{on-basis}), the conclusion in (i) follows. 

(ii) Assume that $m_j \Ge 2,$ $j=1, \ldots, n.$
Note that for any $\alpha \in \mathbb Z^n_+,$ by Proposition~\ref{DoH}(ii),
\allowdisplaybreaks 
\beqn
&  &  \int_{\triangle^{\!n}_{_{0}}} |e_\alpha(z)|^2 \prod_{j=1}^{n-1}\Big(1- \frac{|z_j|^2}{|z_{j+1}|^2}\Big)^{m_j -2} (1-|z_n|^2)^{m_n-2} dz\\
&\overset{\eqref{e-alpha}}=& A_{P_0, m}(\alpha) \int_{\triangle^{\!n}_{_{0}}} \frac{|\varphi(z)^{\alpha}|^2}{\prod_{j=2}^n|z_j|^2}   \prod_{j=1}^{n-1}\Big(1- \frac{|z_j|^2}{|z_{j+1}|^2}\Big)^{m_j -2} (1-|z_n|^2)^{m_n-2}dz \\
& \overset{\eqref{Jaco}}= & A_{P_0, m}(\alpha) \int_{\mathbb D^n} |w^{\alpha}|^2   \prod_{j=1}^{n} \Big(1- |w_j|^2\Big)^{m_j -2} dw \\
&\overset{\eqref{A-P0-m-alpha}}=&  \prod_{j=1}^{n} \binom{\alpha_j +m_j-1}{m_j -1} \int_{\mathbb D} |w_j|^{2\alpha_j} \Big(1- |w_j|^2\Big)^{m_j -2} dw_j \\
%&=& \prod_{j=1}^{n} {\alpha_j +m_j-1 \choose m_j -1} \frac{\alpha_j !(m_j-2)!}{(\alpha_j +m_j -1)!} \\
&\overset{\eqref{HKZ-id}}=& \frac{\pi^n}{\prod_{j=1}^{n} (m_j-1)}.
 \eeqn 
Thus the identity mapping from $\mbox{span}\{e_\alpha : \alpha \in \mathbb Z^n_+\}$ into $L^2(\triangle^{\!n}_{_{0}}, w(z)dz)$ is isometric. The conclusion in (ii) now follows from Proposition~\ref{on-basis}. 
\end{proof}
\begin{remark}
In case of $m_1=1$ and $m_2 =2,$ it is not difficult to see that the norm on $\mathscr H^2_m(\triangle^{\!2}_{_{0}})$ is given by
\beqn
\|f\|^2 = \frac{1}{2\pi^2} \sup_{t \in (0, 1)} \int_{\mathbb D} \int_{[0, 2\pi]}   \big|f\big(t e^{i \theta}w, w\big)|^2t|w|^2 d\theta dw, \quad f \in \mathscr H^2(\triangle^{\!2}_{_{0}}). 
\eeqn
Similarly, one can also deduce the formula for norm  on $\mathscr H^2_m(\triangle^{\!2}_{_{0}})$  when $m_1 = 2$ and $m_2=1.$ 
\hfill $\diamondsuit$
\end{remark}

% If $m_j=1,$ $j=1, \ldots, n,$ then 
We refer to the space $\mathscr H^2(\triangle^{\!n}_{_{0}})$ as the  {\it Hardy space of the $n$-dimensional Hartogs triangle $\triangle^{\!n}_{_0}.$} In the case of $n=2,$ this space was constructed independently in \cite[Section~3]{M2021} and \cite[Section~6]{GGLV2021} by entirely different methods.  In case $m_j = 2,$ $j=1, \ldots, n,$ one may refer to  $\mathscr H^2_m(\triangle^{\!n}_{_{0}})$ as the {\it Bergman space of $\triangle^{\!n}_{_{0}}.$}  
If $m_j \Ge 2,$ $j=1, \ldots, n,$ then we refer to $\mathscr H^2_m(\triangle^{\!n}_{_{0}})$  as a {\it weighted Bergman space of $\triangle^{\!n}_{_{0}}$}. 
%(refer to \cite[Section~4.9*]{JP2008} for the basics of the Bergman spaces of general domains in $\mathbb C^n$). 
\begin{remark}
Unlike the case of polydisc, the {\it Bergman kernel} $\mathscr K_{0, \bf 2}$ of $\triangle^{\!n}_{_{0}}$ is not a power of the {\it Szeg$\ddot{\mbox{o}}$ kernel} $\mathscr K_{0, \bf 1}$ of $\triangle^{\!n}_{_{0}}$ (see Example~\ref{exm-triangle-2}), where $\bf k$ denotes the $n$-tuple with all entries equal to $k \in \mathbb Z_+.$ 
In connection with the zeros 
of the Bergman kernel, the following is worth noting$:$
%\begin{corollary}
The Bergman kernel $\mathscr K_{0, \bf 2}(\cdot, \cdot)$ of $\triangle^{\!n}_{_{0}}$ (seen as a rational function) has zeroes precisely at $$\{z \in \mathbb C^n : \mbox{for some}~j=2, \ldots, n, ~z_j =0~\mbox{and}~z_i \neq 0, ~1 \Le i \neq j \Le n-1\}.$$ 
%\end{corollary}
%\begin{proof}
Indeed, by Example~\ref{exm-triangle-2}, we have 
\beqn 
 \mathscr K_{0, \bf 2}(z, z) = \frac{1}{(1-|z_n|^{2})}\prod_{j=1}^{n-1} \frac{|z_{j+1}|^2}{\Big(|z_{j+1}|^2-|z_j|^2\Big)^{2}}, \quad z \in \triangle^{\!n}_{0}. 
\eeqn
(cf. \cite[Section~3.1]{BFS1999}).
\hfill $\diamondsuit$
\end{remark}
%The desired conclusion is now clear.  
%\end{proof} 

We conclude this section with an analog of von Neumann's inequality for the $n$-dimensional Hartogs triangle (cf. \cite[Corollary~5.4]{P1999}). This generalizes \cite[Theorem~1.2]{CJP2022} (the case of $n=2$). 
%\marginpar{\color{blue}case $n=3$}
\begin{theorem} \label{vn-H}  
%Let $\mathscr K_{0, \bf 1}$ be the reproducing kernel of the Hardy space $\mathscr H^2(\triangle^{\!n}_{_{0}})$ of the $n$-dimensional Hartogs triangle $\triangle^{\!n}_{_0}.$
If a commuting $n$-tuple $T=(T_1 , \ldots, T_n)$ is a $\triangle^{\!n}_{_0}$-contraction on $H$ such that $\sigma(T) \subset \triangle^{\!n}_{_{0}},$ 
then  
\beqn
%\label{vn-H-ineq}
\|\phi(T)\| \Le\|\phi\|_{\infty, \triangle^{\!n}_{_{0}}}, \quad \phi \in \mathscr H^{\infty}(\triangle^{\!n}_{_{0}}).
\eeqn
\end{theorem}
\begin{proof}
In view of \cite[Theorem 2.1]{CJP2022}, it suffices to check that the multiplier algebra of $\mathscr H^2(\triangle^{\!n}_{_{0}})$ is equal to $\mathscr H^\infty(\triangle^{\!n}_{_P})$ (with equality of norms). This follows from Theorem~\ref{mult-h-infty} and Proposition~\ref{norm-Hardy}(i). 
\end{proof}

Any $\triangle^{\!n}_{_0}$-contraction with Taylor spectrum contained in $\triangle^{\!n}_{_{0}}$ is a contractive $n$-tuple of special kind. 
\begin{corollary} \label{coro-v-N-i}
Assume that $T=(T_1 , \ldots, T_n)$ is a $\triangle^{\!n}_{_0}$-contraction on $H$ such that $\sigma(T) \subset \triangle^{\!n}_{_{0}}.$ 
Then the operators $T_2, \ldots, T_n$ are invertible and    
\beq
\label{vn-H-ineq-new-new}
T^*_jT_j \Le T^*_{j+1}T_{j+1} \Le I, \quad j=1, \ldots, n-1. 
\eeq
In particular, $r(T_j) \Le r(T_{j+1}) < 1$ for every $ j=1, \ldots, n-1.$ 
%where $r(\cdot)$ denotes the spectral radius.   
\end{corollary}
\begin{proof} Since $\sigma(T) \subset \triangle^{\!n}_{_{0}},$ by the projection property (see \cite[Theorem~4.9]{C1988}), the spectra of $T_2, \ldots, T_n$ are contained in $\mathbb D_*.$ In particular, $T_2, \ldots, T_n$ are invertible operators with spectral radii less than $1.$ 
We may now apply Theorem~\ref{vn-H} to the functions $\phi_1, \ldots, \phi_n \in \mathscr H^{\infty}(\triangle^{\!n}_{_{0}})$ given by 
\beqn
\phi_j(z)=\begin{cases} \frac{z_j}{z_{j+1}}, & \mbox{if}~j=1, \ldots, n-1, \\
z_n & \mbox{if~}j=n,
\end{cases}
\eeqn
to obtain \eqref{vn-H-ineq-new-new}. Since $T$ is a commuting $n$-tuple, an induction on $k \Ge 1$ shows that 
\beqn
T^{*k}_jT^k_j \Le T^{*k}_{j+1}T^k_{j+1} \Le I, \quad j=1, \ldots, n-1.
\eeqn
This  
together with the spectral radius formula (see \cite[Theorem~2.2.10]{Si2015}) completes the proof. 
\end{proof}

\section{Hilbert spaces of polydiscs associated with admissible tuples} \label{S8}

One of the essential characteristics of the Hardy space $\mathscr H^2(\triangle^{\!n}_{_{0}})$ of $\triangle^{\!n}_{_0}$ is its  intimate relation with its polydisc counter-part. We make this precise in this section (see Proposition~\ref{Mz-tilde-Mz}). 

Consider the polydisc $$\mathbb D^n_{_{\widetilde{P}}}: = \mathbb D_{_{\widetilde{P}_1}} \times  \cdots \times \mathbb D_{_{\widetilde{P}_n}},$$ 
(see \eqref{Q-disc-is-disc}) and the 
positive semi-definite kernel $\widetilde{\mathscr K}_{{P, m}} : \mathbb D^n_{_{\widetilde{P}}} \times \mathbb D^n_{_{\widetilde{P}}} \rar \mathbb C$ given~by 
\beq \label{kernel-tilde-decom}
\widetilde{\mathscr K}_{{P, m}}(z, w) = \prod_{j=1}^n \frac{1}{(1-\widetilde{P}_j(z_j \overline{w}_j))^{m_j}}, \quad z, w \in \mathbb D^n_{_{\widetilde{P}}}.
\eeq
Applying \eqref{exp-exp} and simplifying the expression above, we obtain   
\beq \label{tilde-kernel-exp}
\widetilde{\mathscr K}_{{P, m}}(z, w) = \sum_{\alpha \in \mathbb Z^n_+} \prod_{j=1}^n A_{\widetilde{P}_j, m_j}(\alpha_j) z^{\alpha}\overline{w}^\alpha, 
\eeq 
where the series above converges compactly on $\mathbb D^n_{_{\widetilde{P}}} \times \mathbb D^n_{_{\widetilde{P}}}.$
Let $\mathscr H^2_{m}( \mathbb D^n_{_{\widetilde{P}}})$ be the reproducing kernel Hilbert space associated with the kernel $\widetilde{\mathscr K}_{{P, m}}.$ 

\begin{proposition} \label{onb-polydisc}
For $\alpha \in \mathbb Z^n_+, $ define $f_\alpha : \mathbb D^n_{_{\widetilde{P}}} \rar \mathbb C$ by 
\beq \label{f-alpha-new}
f_\alpha(z)= \prod_{j=1}^n \sqrt{A_{\widetilde{P}_j, m_j}(\alpha_j)}\, z^{\alpha}, \quad z \in \mathbb D^n_{_{\widetilde{P}}}.
\eeq
Then $\{f_{\alpha}\}_{\alpha \in \mathbb Z^n_+}$ forms an orthonormal basis for $\mathscr H^2_{m}( \mathbb D^n_{_{\widetilde{P}}}).$ Moreover,
$z_k$ is a multiplier of $\mathscr H^2_{m}( \mathbb D^n_{_{\widetilde{P}}})$ and  
\beq \label{action-tilde-Mz}
z_k f_\alpha  = \frac{\sqrt{A_{\widetilde{P}_k, m_k}(\alpha_k)}}{ \sqrt{A_{\widetilde{P}_k, m_k}(\alpha_k + 1)}} f_{\alpha + \varepsilon_k}, \quad \alpha \in \mathbb Z^n_+, ~k=1, \ldots, n.
\eeq
\end{proposition}
\begin{proof}[Outline of the proof] 
As in the proof of Proposition~\ref{on-basis}, one may deduce the first part from \cite[Proposition~2.8, Theorem 2.10 and Exercise~3.7]{PR2016}.
The verification of the remaining part is similar to that of Proposition~\ref{bddness}. Note that for any $\alpha \in \mathbb Z^n_+$ and $k=1, \ldots, n,$ 
\beqn
z_k f_\alpha  = \frac{\prod_{j=1}^n \sqrt{A_{\widetilde{P}_j, m_j}(\alpha_j)}}{\prod_{j=1}^n \sqrt{A_{\widetilde{P}_j, m_j}((\alpha + \varepsilon_k)_j)}} f_{\alpha + \varepsilon_k} = \frac{\sqrt{A_{\widetilde{P}_k, m_k}(\alpha_k)}}{ \sqrt{A_{\widetilde{P}_k, m_k}(\alpha_k + 1)}} f_{\alpha + \varepsilon_k},
\eeqn
and hence we may infer from Lemma~\ref{Spott}(i) (applied to $Q=\widetilde{P}_k$) that 
\beq \label{bdd-single-multi}
\|z_k f_\alpha\|^2  =  \frac{A_{\widetilde{P}_k, m_k}(\alpha_k)}{ A_{\widetilde{P}_k, m_k}(\alpha_k + 1)} \Le \frac{1}{\partial_k{P}_k(0)}.
\eeq
This completes the proof. 
\end{proof}

For $k=1, \ldots, n,$ let $\widetilde{\mathscr M}_{z_k}$ denote the operator of multiplication by the coordinate function $z_k$ on $\mathscr H^2_{m}( \mathbb D^n_{_{\widetilde{P}}}).$ Combining Proposition~\ref{onb-polydisc} with \eqref{bdd-single-multi} yields the following estimate$:$
$$
\Big\|\widetilde{\mathscr M}_{z_k}\Big\|\Le \frac{1}{\sqrt{\partial_k{P}_k(0)}}, \quad k=1, \ldots, n.$$
Unlike the case of $\mathscr H^2(\triangle^{\!n}_{_{P}})$ (see Proposition~\ref{not-F}(i)),  the multiplication $n$-tuple $\widetilde{\mathscr M}_z$ on $\mathscr H^2_{m}( \mathbb D^n_{_{\widetilde{P}}})$ is always doubly commuting.
\begin{corollary} \label{coro-d-comm}
The multiplication $n$-tuple $\widetilde{\mathscr M}_z=(\widetilde{\mathscr M}_{z_1}, \ldots, \widetilde{\mathscr M}_{z_n})$ on the space $\mathscr H^2_{m}( \mathbb D^n_{_{\widetilde{P}}})$ is  doubly commuting.
\end{corollary} 
\begin{proof} Clearly, $\widetilde{\mathscr M}_z$ is a commuting $n$-tuple. 
For $k=1, \ldots, n,$ let $\widetilde{\mathscr M}^*_{z_k}$ denote the Hilbert space adjoint of $\widetilde{\mathscr M}_{z_k}.$
It is easy to see from \eqref{action-tilde-Mz} that for any $\alpha \in \mathbb Z^n_+$ and $k=1, \ldots, n,$
\beqn
\widetilde{\mathscr M}^*_{z_k}  f_\alpha  = \begin{cases} 
0 & \mbox{if}~\alpha_k = 0, \\
\frac{\sqrt{A_{\widetilde{P}_k, m_k}(\alpha_k-1)}}{ \sqrt{A_{\widetilde{P}_k, m_k}(\alpha_k)}} f_{\alpha - \varepsilon_k}, & \mbox{otherwise}.  
\end{cases}
\eeqn
It is now easy to see that 
\beqn
\Big(\widetilde{\mathscr M}_{z_j}\widetilde{\mathscr M}^*_{z_k} - \widetilde{\mathscr M}^*_{z_k}\widetilde{\mathscr M}_{z_j} \Big)f_\alpha =0, \quad \alpha \in \mathbb Z^n_+, ~1 \Le j \neq k \Le n.
\eeqn
The desired conclusion now follows from the fact that $\{f_{\alpha}\}_{\alpha \in \mathbb Z^n_+}$ forms an orthonormal basis for $\mathscr H^2_{m}( \mathbb D^n_{_{\widetilde{P}}})$ (see Proposition~\ref{onb-polydisc}). 
\end{proof}

For an admissible polynomial $n$-tuple $P,$ the relation between $\mathscr H^2_m(\triangle^{\!n}_{_P})$  and $\mathscr H^2_{m}( \mathbb D^n_{_{\widetilde{P}}})$ and the respective multiplication $n$-tuples can be made precise. 
\begin{proposition} \label{Mz-tilde-Mz}
Assume that $P$ is an admissible polynomial $n$-tuple.  
Then the linear transformation $\Psi : \mathscr H^2_m(\triangle^{\!n}_{_P}) \rar \mathscr H^2_{m}( \mathbb D^n_{_{\widetilde{P}}})$ given by
\beqn
\Psi(f) = J_{_{\varphi^{-1}}} \cdot f \circ \varphi^{-1}, \quad f \in \mathscr H^2_m(\triangle^{\!n}_{_P}),
\eeqn
defines a unitary transformation. Moreover, if $\mathscr M_z$ and $\widetilde{\mathscr M}_{z}$ are the multiplication $n$-tuples on $\mathscr H^2_m(\triangle^{\!n}_{_P})$ and $\mathscr H^2_{m}( \mathbb D^n_{_{\widetilde{P}}}),$ respectively, then   
\beqn
\Psi \mathscr M_{z_j} = \Big(\prod_{k=j}^n \widetilde{\mathscr M}_{z_k}\Big) \Psi, \quad j=1, \ldots, n. 
\eeqn 
\end{proposition}
\begin{proof} 
Let $f_\alpha,$ $\alpha \in \mathbb Z^n_+,$ be as in \eqref{f-alpha-new}.  
Define
\beqn 
%\label{def-lambda}
\Phi(f_\alpha) = \frac{\prod_{j=1}^n\sqrt{A_{\widetilde{P}_j, m_j}(\alpha_j)}}{\sqrt{A_{P, m}(\alpha)} } \, e_\alpha, \quad \alpha \in \mathbb Z^n_+,
%J_{\varphi^{-1}} f_ \circ \varphi^{-1}
\eeqn
which extends linearly to the linear span of $\{f_\alpha\}_{\alpha \in \mathbb Z^n_+}.$ By Propositions~\ref{on-basis} and \ref{onb-polydisc}, for any finite subset $F$ of $\mathbb Z^n_+$ and $\{a_\alpha : \alpha \in F\} \subset \mathbb C,$
\beq  \label{equality-holds}
\Big\|\Phi \Big(\sum_{\alpha \in F} a_\alpha f_\alpha \Big)\Big\|^2 \notag &=& \sum_{\alpha \in F} |a_\alpha|^2\, \frac{\prod_{j=1}^nA_{\widetilde{P}_j, m_j}(\alpha_j)}{A_{P, m}(\alpha)} \\ & \overset{\eqref{rho-alpha-f}}  \Le & \sum_{\alpha \in F} |a_\alpha|^2   \\
&=& \Big\|\sum_{\alpha \in F} a_\alpha f_\alpha\Big\|^2. \notag 
\eeq
Hence, $\Phi$ extends boundedly from $\mathscr H^2_{m}( \mathbb D^n_{_{\widetilde{P}}})$ into $\mathscr H^2_m(\triangle^{\!n}_{_P}).$ 
Since $P$ is admissible, 
$\widetilde{P}_j(z_j)=P_j(z),$ $z \in \mathbb C^n,$ and hence by \eqref{exp}, \eqref{A-tilde-A-new}, \eqref{tilde-kernel-exp} and Proposition~\ref{DoH}(ii), 
\beq \label{A-tilde-A}
A_{P, m}(\alpha) = \prod_{j=1}^nA_{\widetilde{P}_j, m_j}(\alpha_j), \quad \alpha \in \mathbb Z^n_+. 
\eeq
It follows that equality holds in \eqref{equality-holds} and $\Phi$ is a unitary map from $\mathscr H^2_{m}(\mathbb D^n_{_{\widetilde{P}}})$ onto $\mathscr H^2_m(\triangle^{\!n}_{_P}).$

Since $\{e_\alpha\}_{\alpha \in \mathbb Z^n_+}$ is an orthonormal basis for $\mathscr H^2_m(\triangle^{\!n}_{_P}),$ it suffices to check that 
 $\Psi(e_\alpha) = \Phi^{-1}(e_\alpha),$ $\alpha \in \mathbb Z^n_+.$ Indeed, 
 \allowdisplaybreaks
\beqn
\Psi(e_\alpha)(z) =  J_{_{\varphi^{-1}}}(z)  e_\alpha \circ \varphi^{-1}(z) &\overset{\eqref{phi-inverse}\&\eqref{e-alpha}}=&  
J_{_{\varphi^{-1}}}(z) \frac{\sqrt{A_{P, m}(\alpha)} \, z^{\alpha}}{\prod_{j=2}^nz^{j-1}_j} \\
& \overset{\eqref{Jaco}\&\eqref{A-tilde-A}}= & 
\prod_{j=1}^n \sqrt{A_{\widetilde{P}_j, m_j}(\alpha_j)} z^{\alpha}
\\
& \overset{\eqref{f-alpha-new}} = &
f_\alpha(z), \quad z \in \mathbb D^n_{_{\widetilde{P}}}, ~\alpha \in \mathbb Z^n_+.
\eeqn
This 
completes the proof of the first part. 
This also shows that
\beq \label{PSI-e-alpha-f-alpha}
\Psi(e_\alpha) = f_\alpha, \quad \alpha \in \mathbb Z^n_+.
\eeq
To see the remaining part, note that 
for any $\alpha \in \mathbb Z^n_+$ and $j=1, \ldots, n,$
\beqn
\Psi \mathscr M_{z_j}(e_\alpha) & \overset{\eqref{action-basis}} =& \frac{\sqrt{A_{P, m}(\alpha)}}{\sqrt{A_{P, m}(\alpha + \sum_{k=j}^n \varepsilon_k)}}\,f_{\alpha +  \sum_{k=j}^n \varepsilon_k} \\ & \overset{\eqref{A-tilde-A}} =& \frac{\prod_{k=j}^n \sqrt{A_{\widetilde{P}_k, m_k}(\alpha_k)}}{\prod_{k=j}^n \sqrt{A_{\widetilde{P}_k, m_k}(\alpha_k+1)}}\,f_{\alpha +  \sum_{k=j}^n \varepsilon_k} \\
& \overset{\eqref{action-tilde-Mz}\&\eqref{PSI-e-alpha-f-alpha}} = & \Big(\prod_{k=j}^n \widetilde{\mathscr M}_{z_k}\Big) \Psi(e_\alpha).
\eeqn
One may now apply Proposition~\ref{on-basis} to complete the proof. 
\end{proof}
The families of separately subnormal and jointly subnormal multiplication $n$-tuples $\widetilde{\mathscr M}_{z}$ on $\mathscr H^2_{m}( \mathbb D^n_{_{\widetilde{P}}})$ coincide provided $P$ is admissible. 
\begin{corollary} \label{sep-joint-sub}
Assume that $P$ is an admissible polynomial $n$-tuple. 
%If $\widetilde{\mathscr M}_{z_1}, \ldots, \widetilde{\mathscr M}_{z_n}$ are subnormal operators, 
Then the multiplication $n$-tuple $\widetilde{\mathscr M}_{z}$ on $\mathscr H^2_{m}( \mathbb D^n_{_{\widetilde{P}}})$ is separately subnormal if and only if 
$\widetilde{\mathscr M}_{z}$ is jointly subnormal. In particular, if $\widetilde{\mathscr M}_{z}$ is separately subnormal, then the multiplication $n$-tuple $\mathscr M_z$ on $\mathscr H^2_m(\triangle^{\!n}_{_P})$ is jointly subnormal. 
\end{corollary}
\begin{proof}
Since $\widetilde{\mathscr M}_z$ is doubly commuting (see Corollary~\ref{coro-d-comm}), if $\widetilde{\mathscr M}_{z}$ is separately subnormal, then by \cite[Proposition~1]{AP1990}, $\widetilde{\mathscr M}_{z}$ is jointly subnormal. The converse is trivial. The remaining part now follows  Proposition~\ref{Mz-tilde-Mz}.
\end{proof}

\subsection{Taylor spectrum} 
%\label{S9}

In this subsection, we compute the Taylor spectrum of $\mathscr M_z$ on $\mathscr H^2_m(\triangle^{\!n}_{_P})$ provided $P$ is an admissible polynomial $n$-tuple. This computation relies heavily on a result from \cite{CV1978} pertaining to the Taylor spectra of tensor products of commuting operator tuples on a Hilbert space.  

\begin{proposition} \label{prop-spectrum}
Assume that $P$ is an admissible polynomial $n$-tuple. Then the Taylor spectrum of $\mathscr M_z$ on $\mathscr H^2_m(\triangle^{\!n}_{_P})$ is equal to $\overline{\triangle^{\!n}_{_{P_{\bf r}}}}$ $($see Example~\ref{exam-Delta-P-r}$),$ 
where ${\bf r}=(r_1, \ldots, r_n)$ is given by 
\beq \label{rj}
r_j = \lim_{n \rar \infty} \sup_{k \Ge 0} \Bigg(\frac{\sqrt{A_{\widetilde{P}_j, m_j}(k)}}{ \sqrt{A_{\widetilde{P}_j, m_j}(k + n)}}\Bigg)^{\frac{1}{n}}, \quad j=1, \ldots, n
\eeq
$($see \eqref{tilde-kernel-exp}$).$ 
Moreover, $r_j \Le \frac{1}{\sqrt{\partial_j{P}_j(0)}},$ $j=1, \ldots, n.$
\end{proposition}
\begin{proof}
%[Proof of Theorem~\ref{prop-spectrum}] 
Let $j=1, \ldots, n.$ Consider the positive semi-definite kernel 
\beqn
\widetilde{\mathscr K}_{{P_j, m_j}}(s, t) = \frac{1}{\Big(1-\widetilde{P}_j(s  \overline{t})\Big)^{m_j}}, ~ s, t \in \mathbb D_{_{\widetilde{P}_j}}.
\eeqn
Let $\mathscr H^2_{m_j}(\mathbb D_{_{\widetilde{P}_j}})$ denote the reproducing kernel Hilbert space associated with $\widetilde{\mathscr K}_{{P_j, m_j}}$ and   
let $M_j$ denote the operator (possibly unbounded) of multiplication by the coordinate function in $\mathscr H^2_{m_j}(\mathbb D_{_{\widetilde{P}_j}}).$ 
By \eqref{action-tilde-Mz} (applied to $n=1$), the multiplication operator $M_j$ is the weighted shift on $\mathscr H^2_{m_j}(\mathbb D_{_{\widetilde{P}_j}})$  with the bounded weight sequence $$\Big\{\sqrt{A_{\widetilde{P}_j, m_j}(k)}/\sqrt{A_{\widetilde{P}_j, m_j}(k + 1)} : k \in \mathbb Z_+\Big\}$$ (see \eqref{bdd-single-multi}). It follows that $M_j$ is a bounded linear operator on $\mathscr H^2_{m_j}(\mathbb D_{_{\widetilde{P}_j}}).$  
In particular,  
\beq \label{pectrum}
\text{the spectrum of $M_j$ is the closed disc $\overline{\mathbb D}(0, r_j),$}
\eeq
where $r_j$ is given by \eqref{rj}  (see \cite[Theorem~4]{S1974}). 

Note that the reproducing kernel $\widetilde{\mathscr K}_{{P, m}}$ of $\mathscr H^2_m(\mathbb D^n_{_{\widetilde{P}}})$ takes the form
\beqn
\widetilde{\mathscr K}_{{P, m}}(z, w) \overset{\eqref{kernel-tilde-decom}}= \prod_{j=1}^n 
\widetilde{\mathscr K}_{{P_j, m_j}}(z_j, w_j), \quad z, w \in \mathbb D^n_{_{\widetilde{P}}}.
\eeqn
Hence, by repeated applications of \cite[Theorem~5.11]{PR2016}, 
there exists a well-defined, linear isometry $\mathcal U$ from $\otimes_{j=1}^n \mathscr H^2_{m_j}(\mathbb D_{_{\widetilde{P}_j}})$ onto the reproducing kernel Hilbert space $\mathscr H^2_m(\mathbb D^n_{_{\widetilde{P}}})$ such that
\beqn
\mathcal U(\otimes_{j=1}^n h_j)(z) = \prod_{j=1}^n h_j(z_j), \quad h_j \in \mathscr H^2_{m_j}(\mathbb D_{_{\widetilde{P}_j}}), ~z = (z_1, \ldots, z_n) \in \mathbb D^n_{_{\widetilde{P}}},
\eeqn
where $\otimes$ denotes the Hilbert space tensor product. 
Consider the commuting $n$-tuple $\widetilde{M}=(\widetilde{M}_1, \ldots, \widetilde{M}_n),$  where $\widetilde{M}_j$ acting on $\otimes_{j=1}^n \mathscr H^2_{m_j}(\mathbb D_{_{\widetilde{P}_j}})$  is given by
\beqn
\widetilde{M}_j = I \otimes  \cdots I \otimes \underbrace{M_j}_{\text{\tiny jth place}} \otimes I \cdots \otimes I, \quad j=1, \ldots, n.
\eeqn
It is easy to see that 
\beqn
\mathcal U \widetilde{M}_j (\otimes_{k=1}^n h_k)
%&=&  \mathcal U( h_1 \otimes  \cdots h_{j-1} \otimes \underbrace{M_jh_j}_{\text{\tiny jth place}} \otimes h_{j+1} \cdots \otimes h_n)(z) \\
%&=& z_j \prod_{k=1}^nh_k(z_k)\\
= \widetilde{\mathscr M}_{z_j} \mathcal U (\otimes_{k=1}^n h_k), \quad h_k \in\mathscr H^2_{m_k}(\mathbb D_{_{\widetilde{P}_k}}), ~k, j=1, \cdots, n.
\eeqn
Since linear span of $\big\{\otimes_{k=1}^n h_k : h_k \in\mathscr H^2_{m_k}(\mathbb D_{_{\widetilde{P}_k}})\big\}$ is dense in $\otimes_{j=1}^n \mathscr H^2_{m_j}(\mathbb D_{_{\widetilde{P}_j}})$ (see \cite[Proof of Theorem~5.11]{PR2016}), we conclude that $\widetilde{M}$ and $\widetilde{\mathscr M}_z$ are unitarily equivalent. Hence, by \cite[Theorem~2.2]{CV1978}, 
\beq 
\label{phi-spectrum-tilde}
\sigma(\widetilde{\mathscr M}_z) = \sigma(\widetilde{M}) =  \sigma(M_1) \times \cdots \times \sigma(M_n). 
\eeq
Since $P$ is an admissible polynomial $n$-tuple, 
by Proposition~\ref{Mz-tilde-Mz}, $\sigma(\mathscr M_z)=\sigma(\varphi^{-1}(\widetilde{\mathscr M}_{z}))$ (see \eqref{phi-inverse}). It now follows from \eqref{phi-spectrum-tilde} and the spectral mapping property that 
\allowdisplaybreaks
\beqn
\sigma(\mathscr M_z) 
%&=& \Big\{\Big(\prod_{j=1}^n z_j, \prod_{j=2}^n z_j, \ldots, z_n\Big) \in \mathbb C^n : z \in \sigma(\widetilde{\mathscr M}_z)\Big\} \\
%&=& \Big\{\Big(\prod_{j=1}^n z_j, \prod_{j=2}^n z_j, \ldots, z_n\Big) \in \mathbb C^n : z_j \in \sigma(M_j)\Big\} 
= \varphi^{-1}\Big(\sigma(M_1) \times \cdots \times \sigma(M_n)\Big). 
\eeqn
This combined with \eqref{image-polydisc} and \eqref{pectrum} completes the proof of the first part. The rest follows from \eqref{bdd-single-multi} and the fact that the spectral radius of a bounded linear operator is at most its norm.  
\end{proof}

For an admissible polynomial $n$-tuple $P,$ the Taylor spectrum of a  separately hyponormal multiplication $n$-tuple $\mathscr M_z$ on
$\mathscr H^2_{m}(\triangle^{\!n}_{_P})$ can be computed  explicitly.  
\begin{corollary} \label{spec-precise}
Assume that $P$ is an admissible polynomial $n$-tuple and let $\mathscr M_z$ be the multiplication $n$-tuple on $\mathscr H^2_{m}(\triangle^{\!n}_{_P}).$ If $\mathscr M_z$ is separately hyponormal, then   
$\sigma(\mathscr M_z)=\overline{\triangle^{\!n}_{_{P_{\bf r}}}},$ where ${\bf r}=\varphi(\|\mathscr M_{z_1}\|, \ldots, \|\mathscr M_{z_n}\|)$ $($see \eqref{tilde}$).$  If, in addition, $m$ is the $n$-tuple with all entries equal to $1,$ then $\sigma(\mathscr M_z)=\overline{\triangle^{\!n}_{_{P_{\bf r}}}},$ 
where ${\bf r}=(r_1, \ldots, r_n)$ with $r_j=\frac{1}{\sqrt{\partial_j P_j(0)}},$ $j=1, \ldots, n.$
\end{corollary}
%\begin{corollary} \label{spec-precise}
%Assume that $P$ is an admissible polynomial $n$-tuple. Then the following statements are valid$:$ 
%\begin{enumerate}
%\item if the multiplication $n$-tuple $\mathscr M_z$ on $\mathscr H^2_{m}(\triangle^{\!n}_{_P})$ is separately hyponormal, then   
%$$\sigma(\mathscr M_z)=\overline{\triangle^{\!n}_{_{P_{\bf r}}}},$$ where ${\bf r}=\varphi(\|\mathscr M_{z_1}\|, \ldots, \|\mathscr M_{z_n}\|)$ $($see \eqref{tilde}$),$  
%\item if the multiplication $n$-tuple $\mathscr M_z$ on $\mathscr H^2(\triangle^{\!n}_{_P})$ is separately hyponormal, then $$\sigma(\mathscr M_z)=\overline{\triangle^{\!n}_{_{P_{\bf r}}}},$$ 
%where ${\bf r}=(r_1, \ldots, r_n)$ with $r_j=\frac{1}{\sqrt{\partial_j P_j(0)}},$ $j=1, \ldots, n.$
%\end{enumerate}
%\end{corollary}
\begin{proof} Fix $j=1, \ldots, n.$
By Proposition~\ref{prop-spectrum}, there exists an $n$-tuple $\bf r$ of positive real numbers $r_1, \ldots, r_n$ such that  
\beq \label{spectrum-gen-H-T}
\sigma(\mathscr M_z) = \overline{\triangle^{\!n}_{_{P_{\bf r}}}}.
\eeq
This combined with the projection property of the Taylor spectrum (see \cite[Theorem~4.9]{C1988}) shows that 
\beqn
\sigma(\mathscr M_{z_j}) = \{w \in \mathbb C : \mathrm{there ~exists}~z \in \overline{\triangle^{\!n}_{_{P_{\bf r}}}}~\mathrm{with~}z_j=w\},
\eeqn
which is easily seen to be equal to the closed disc centered at the origin and of radius $\prod_{k=j}^nr_k.$  
On the other hand, by the assumption, $\mathscr M_{z_j}$ is hyponormal. Hence, by \cite[Proposition~II.4.6]{Co1991}, 
\beqn 
%\label{norm-sp-rad}
\|\mathscr M_{z_j}\|=r(\mathscr M_{z_j}) = \prod_{k=j}^nr_k.
\eeqn
%An application of   (with $m={\bf 1}$) together with 
This together with \eqref{phi-inverse} and \eqref{spectrum-gen-H-T} yields the first part. The remaining part now follows from Corollary~\ref{norm-attained}(ii). 
\end{proof}

\section{Closing remarks} \label{S9}
The aim of this work was to introduce generalized Hartogs triangles $\triangle^{\!n}_{_P}$ and to provide a theoretical framework for discussing operator theory on these domains.
One of the main outcomes of all this analysis is the von Neumann's inequality for the $n$-dimensional Hartogs triangle.  Following \cite{AMY2020}, it is natural to seek a counterpart of dilation theory, model theory and function theory on these domains. 
In this regard, the following, a simple consequence of the corresponding result for the unit bidisc (see \cite[Theorem~4.49]{AMY2020}), is worth noting. 
\begin{proposition}[Pick’s interpolation theorem on the Hartogs triangle]  
\label{Pick}
Let $\lambda_j =(\lambda^{(j)}_1 , \lambda^{(j)}_2),$ $j=1, \ldots, k,$ be distinct points in $\triangle^{\!2}_0$ and let $z_1, \ldots, z_k \in \mathbb C.$ There exists $\psi \in \mathscr H^{\infty}(\triangle^{\!2}_0)$ with $\|\psi\|_{\infty, \triangle^{\!2}_0} \Le 1$ such that $\psi(\lambda_j) = z_j$ for $j = 1, \ldots, k$ if and only if there exists a pair of $k \times k$ positive semi-definite matrices $A_1 = [a^{(1)}_{i, j}]$ and $A_2 = [a^{(2)}_{i, j}]$ such that
for $i,j = 1, \ldots, k,$
\beqn
1-\overline{z}_i z_j = \Big({\overline{\lambda^{(2)}_i}}{{\lambda^{(2)}_j}} - \overline{\lambda^{(1)}_i} {{\lambda^{(1)}_j}}\Big) a^{(1)}_{i, j} + 
\Big(1-\overline{\lambda^{(2)}_i} \lambda^{(2)}_j\Big)a^{(2)}_{i, j}.
\eeqn
\end{proposition}
\begin{proof}
Since $\varphi,$ as given by \eqref{tilde}, maps $\triangle^{\!2}_0$ biholomorphically onto $\mathbb D \times \mathbb D_*,$ the result may be easily deduced from \cite[Theorem~4.49]{AMY2020}.
\end{proof}
Needless to say, one can also get an analogue of \cite[Theorem~4.56]{AMY2020} for the Hartogs triangle. 
This raises, in particular, the question for which positive regular polynomial $2$-tuples $P,$ the generalized Hartogs triangle $\triangle^{\!2}_{_P}$ has the Pick’s interpolation.  

In Subsection~\ref{S5.3}, we discussed the notion of $\triangle^{\!n}_{_P}$-isometry (see Definition~\ref{def-D-contraction}). Below are some simple but basic examples of $\triangle^{\!n}_{_0}$-isometries. 
\begin{example} \label{last-ex}
Let $\phi \in \mathscr H^{\infty}(\mathbb D)$ and $\psi$ be an invertible element of $\mathscr H^{\infty}(\mathbb D).$ If $\frac{\phi}{\psi}$ is an inner function, then by
\cite[Proposition~V.2.2]{SFBK2010}, 
$\mathscr M_\phi^* \mathscr M_\phi=\mathscr M^*_\psi \mathscr M_\psi,$ and hence by Remark~\ref{rmk-H-triangle}, 
%$D^{(1)}_3(T_1, T_2, T_3)=0,$  
$(\mathscr M_\phi, \mathscr M_\psi)$ is a $\triangle^{\!2}_{_0}$-isometry and 
$(\mathscr M_z, \mathscr M_\phi, \mathscr M_\psi)$ is a $\triangle^{\!3}_{_0}$-isometry,
where $\mathscr M_{f} \in \mathcal B(\mathscr H^2(\mathbb D))$ denotes the operator of multiplication by $f \in \mathscr H^{\infty}(\mathbb D).$ It is also worth mentioning that the commuting $2$-tuple $(\mathscr M_\phi, \mathscr M_z)$ is a $\triangle^{\!2}_{_0}$-isometry. 
Indeed, if $T=(T_1, T_2)$ is a commuting $2$-tuple and $T_2$ is an isometry, then  
\beqn
T^*_2\big(T^*_2T_2-T^*_1T_1\big)T_2 = I - T^*_1T_1=T^*_2T_2-T^*_1T_1,
\eeqn
and hence $T$ is a $\triangle^{\!2}_{_0}$-isometry. 
\eof 
\end{example}
In addition to the previous example, we would like to draw attention to the following connection between the notions of $\triangle^{\!n}_{_0}$-contractivity and hypercontractivity (see \cites{A1985, AS2001} for the necessary definitions)$:$ 
\begin{align} \label{connection}
&\mbox{\it If $T$ is a toral $n$-hypercontraction $($resp. a toral $n$-isometry$),$ then} &\\ \notag 
& \mbox{\it the commuting $n$-tuple $\varphi^{-1}(T)$ is a $\triangle^{\!n}_{_0}$-contraction $($resp. a $\triangle^{\!n}_{_0}$-isometry$).$}&
\end{align} 
This fact is a consequence of Agler's hereditary functional calculus (see \cite[Chapter~4]{AMY2020}) and \eqref{phi-inverse}.
In view of Example~\ref{last-ex} and the preceding discussion, 
one may ask for an analog of the von Neumann-Wold decomposition (see \cite{Co1991}). 
%It is also natural to explore the Toeplitz operators on the Hardy space of the $n$-dimensional Hardy space (cf. \cite{BNRWZ2022}). 
We hope to explore these and related issues in a sequel to this paper.

{}
\end{document}